\documentclass[a4paper, 11pt]{amsart}

\usepackage[utf8]{inputenc}
\usepackage[all,cmtip]{xy}
\usepackage[english]{babel}
\usepackage{mathtools}
\usepackage{graphicx}
\usepackage{multicol}

\usepackage{amsmath,amssymb,amsthm,mathrsfs}
\usepackage[alphabetic]{amsrefs}
\let\amsrtimes=\rtimes
\usepackage{xfrac}
\usepackage{pifont}

\usepackage{array,tabularx}
\usepackage{xcolor}

\usepackage{tikz,tikz-cd}
\usetikzlibrary{positioning, shapes.geometric}
\usetikzlibrary{arrows,arrows.meta}
\usetikzlibrary{calc}
\tikzcdset{arrow style=tikz, diagrams={>={Computer Modern Rightarrow[width=5.5pt,length=3pt]}}}

\usepackage{enumitem}
\usepackage{cleveref}

\usepackage[retainorgcmds]{IEEEtrantools}

\theoremstyle{plain}
\newtheorem{thm}{Theorem}[section] 
\newtheorem*{thm*}{Theorem}

\newtheorem{prop}[thm]{Proposition}
\newtheorem{cor}[thm]{Corollary}
\newtheorem{lem}[thm]{Lemma}
\newtheorem{letthm}{Theorem}

\newtheorem*{question}{Question}

\counterwithin{equation}{section}

\theoremstyle{definition}
\newtheorem{defn}[thm]{Definition}
\newtheorem{exa}[thm]{Example}
\newtheorem{rem}[thm]{Remark}
\newtheorem{strat}[thm]{Strategy}

\newtheorem*{rem*}{Remark}

\newcommand*{\myproofname}{Proof of Theorem \ref{1thm:inertness}}

\newcommand*{\myproofnames}{Proof of Proposition \ref{prop:pdhmlgy}}

\newcommand{\Z}{\varmathbb{Z}}

\newcommand{\C}{\varmathbb{C}}
\renewcommand{\H}{\varmathbb{H}}
\renewcommand{\O}{\varmathbb{O}}

\usepackage{geometry}
\usepackage{libertine}
\usepackage[libertine]{newtxmath}
\renewcommand{\mathbb}{\varmathbb}

\input xypic 
\input xy 
\xyoption{all} 
\usepackage{amssymb} 
\usepackage{bbm}

\newcounter{bean}

\newcommand{\namedright}[3]{\ensuremath{#1\stackrel{#2}
 {\longrightarrow}#3}}

\newcommand{\qqed}{\hfill\Box}

\title{Gyration Stability for Projective Planes}

\author{Sebastian Chenery}
\address{(Chenery) University of Bristol, School of Mathematics, Fry Building, Woodland Road, Bristol, BS8 1UG}
\email{seb.chenery@bristol.ac.uk}
\author{Stephen Theriault}
\address{(Theriault) University of Southampton, Mathematical Sciences, Building 54, Southampton, SO17 1BJ}
\email{s.d.theriault@soton.ac.uk}


\begin{document}

\begin{abstract}
 Gyrations are operations on manifolds that arise in geometric topology, where a manifold~\(M\) may exhibit distinct gyrations depending on the chosen twisting. For a given \(M\), we ask a natural question: do all gyrations of \(M\) share the same homotopy type regardless of the twisting? A manifold with this property is said to have gyration stability. Inspired by recent work by Duan, which demonstrated that the quaternionic projective plane is not gyration stable with respect to diffeomorphism, we explore this question for projective planes in general. We obtain a complete description of gyration stability for the complex, quaternionic, and octonionic projective planes up to homotopy. 
\end{abstract}

\maketitle
\tableofcontents

\section*{Introduction}

The classification of manifolds up to a notion of equivalence, be it diffeomorphism, homeomorphism or homotopy equivalence, is a motivating problem in geometry. Classically, Milnor classified closed simply-connected 4-manifolds up to homotopy equivalence~\cite{milnor1958} and then Freedman famously classified them up to homeomorphism~\cite{freedman}. Their classification up to diffeomorphism remains one of the great outstanding problems in the subject (see \cite{milnorclass} for context). Among the other families of manifolds that have been classified are closed smooth simply-connected $5$-manifolds~\cites{smale, barden}, closed smooth simply-connected $6$-manifolds~\cites{wall6mfld, jupp, zhubr}, and closed smooth simply-connected spin $7$-manifolds~\cite{crowley_nordstrom}. Other work has also considered classifications of families of manifolds with specific properties, such as circle bundles over simply-connected $4$-manifolds \cite{dl}, and more recently, $5$-manifolds $M$ whose fundamental group is isomorphic to $\mathbb{Z}/2\mathbb{Z}$ and with $\pi_{2}(M)$ being free abelian and a trivial module over the group ring \cite{hambleton_su}.

A gyration is a surgery on the Cartesian product of a given manifold and a sphere. Originally defined by Gonz{\'a}lez Acu{\~n}a \cite{gonzalezacuna}, they have since appeared in at least three seemingly distinct contexts. One is in work of Bosio--Meersseman \cite{bosio_meersseman} and Gitler--Lopez-de-Medrano \cite{gitler-ldm} on intersections of quadrics, carrying with them deep links to the topology of polyhedral products and their underlying combinatorics. A second is in recent work of Duan \cite{duan} on circle actions on smooth manifolds, subsequently developed by Galaz-Garc\'ia--Reiser \cite{galaz-garcia--reiser}, and used to great effect to classify simply-connected 6-manifolds that admit circle actions. Third, there is work of Kasprowski--Land--Powell--Teichner \cite{klpt} on 4-manifolds with the fundamental group of an aspherical 3-manifold, which used 
gyrations of 3-manifolds and connected sums thereof (cf. \cite{klpt}*{Sections 7.2 and 7.3}) to show that two such 4-manifolds with isomorphic fundamental groups are stably diffeomorphic if they are stably homotopy equivalent. Recently, direct study of homotopy theoretic properties of gyrations has appeared in work of Huang and the second author \cite{huangtheriault}, Basu-Ghosh \cite{basu-ghosh} and \cite{huang_inertness24}. In this paper we give a (nearly) complete classification of the homotopy types of gyrations on the projective planes \(\C P^2\), \(\H P^2\) and \(\O P^2\).

Let $M$ be a closed simply-connected \(n\)-manifold. Such an \(M\) may be thought of as a Poincar\'{e} Duality complex with a single $n$-dimensional cell, so we let $\overline{M}$ be its $(n-1)$-skeleton, and there is a homotopy cofibration 
\[
    S^{n-1} \xrightarrow{f_M} \overline{M} \rightarrow M
\] 
where $f_M$ is the attaching map for the top-cell. Let $k\geq 2$ be an integer and take a class \(\tau\in\pi_{k-1}(\mathrm{SO}(n))\). Using the standard linear action of \(\mathrm{SO}(n)\) on \(S^{n-1}\), define the map 
\[
    t:S^{n-1}\times S^{k-1}\rightarrow S^{n-1}\times S^{k-1}
\] 
by \(t(a, x)=(\tau(x)\cdot a,x)\). The \textit{\(k\)-gyration of \(M\) by \(\tau\)} is defined to be the manifold given by the (strict) pushout
    \begin{equation*}
        \begin{tikzcd}[row sep=3em, column sep=3em]
            S^{n-1}\times S^{k-1} \arrow[r, "1\times \iota"] \arrow[d, "(f_M\times 1)\circ t"] & S^{n-1}\times D^k \arrow[d] \\
            \overline{M}\times S^{k-1} \arrow[r] & \mathcal{G}^k_\tau(M) 
        \end{tikzcd}
    \end{equation*} 
where $\iota$ is the inclusion of the boundary of the disc. When \(\tau\) is trivial the class \(t\) is homotopic to the identity and the above pushout constitutes a $(k-1,n)$-type surgery on $M\times S^{k-1}$ with respect to this trivial choice - we call this the \textit{trivial \(k\)-gyration} and write it as \(\mathcal{G}^k_0(M)\). Otherwise, the surgery is twisted by the action of $\tau$ considered as a diffeomorphism. For this reason the homotopy class \(\tau\) is referred to as a \textit{twisting} in the context of gyrations. This definition of a general \(k\)-gyration via pushouts was introduced by \cite{huangtheriault}, generalising the \(k=2\) case used in both \cite{duan} and \cite{gitler-ldm}, the second being in the context of the trivial 2-gyration. Via the alternative surgery definition (see for example \cite{huang_inertness24}*{Section 12}), it follows that a gyration is an $(n+k-1)$-manifold with an orientation inherited from that of \(M\). This more geometric formulation is written in our notation as
\[
    \mathcal{G}_\tau^k(M)=\left((M-Int(D^n))\times S^{k-1}\right)\cup_t\left(S^{n-1}\times D^k\right)
\]
where \(D^n\subset M\) is an embedded \(n\)-disc centred at a chosen base point of \(M\).

\begin{rem*}
    In previous literature, gyrations have been denoted by `\(\mathcal{G}^\tau(M)\)' which does not keep track of the index \(k\). Given the clarity necessary for the arguments in this paper, we have adopted `\(\mathcal{G}^k_\tau(M)\)' as our notation.
\end{rem*}

This leads us to the two central ``gyration stability" questions, referred to throughout this paper as GSI and GSII.

\begin{question}[GSI]
    For a given \(k\geq2\) and \(n\)-manifold \(M\), do we have \(\mathcal{G}^{k}_{\tau}(M)\simeq\mathcal{G}^{k}_{\omega}(M)\) for all twistings \(\tau,\omega\in\pi_{k-1}(\mathrm{SO}(n))\)? 
\end{question}
\noindent If the answer is yes, we say \(M\) is \textit{\(\mathcal{G}^k\)-stable} or that \(M\) has \textit{\(\mathcal{G}^k\)-stability}; when the context is clear this property is called \textit{gyration stability}. For a fixed \(k\), the property is equivalent to all \(k\)-gyrations having a single homotopy type whatever twisting \(\tau\) is taken. A refined version of the question asks for homotopy types to be distinguished. 

\begin{question}[GSII]
    For a given \(k\geq2\) and \(n\)-manifold \(M\), how many different homotopy types can \(\mathcal{G}^{k}_{\tau}(M)\) have as the homotopy class \(\tau\) is varied? 
\end{question}
\noindent Note that the answer to GSI is ``yes'' if and only if the answer to GSII is ``one'' and so enumerating the possible homotopy types of \(\mathcal{G}^{k}_{\tau}(M)\) is a valid strategy for answering GSI in either the negative or the positive. Moreover, GSII is the stronger version of the statement - one may think of it as asking that if we do not have \(\mathcal{G}^k\)-stability for a given \(k\), exactly how unstable are we?

Observe that gyration stability occurs in index \(k\) whenever \(\pi_{k-1}(\mathrm{SO}(n))\) is the trivial group, as there is only one (i.e.~the trivial) twisting to pick. If \(\pi_{k-1}(SO(n))\) is not trivial then there are distinct non-homotopic twistings, but they may nevertheless result in homotopy equivalent gyrations. For example, when \(M\) is a sphere it is straightforward to show that gyration stability occurs for all $k$ (cf. Example~\ref{ex:spheres}). Projective planes, on the other hand, behave much more delicately. Duan \cite{duan}*{Example 3.4} showed that the quaternionic projective plane \(\H P^2\) is not \(\mathcal{G}^2\)-stable up to diffeomorphism by invoking spin structures. This motivated us to reformulate his result homotopy theoretically in order to consider other projective planes. Our results are summarised in the following table, with indications of where in the paper the results are proved.

\begin{center}
\vspace{5pt}
\begin{tabularx}{\columnwidth}{|X|X|X|X|X|}
    \hline
     \(M\) & \(k\) & GSI? & GSII? & cf.  \\ \hline
    \(\C P^2\) & 2 & Yes & 1 & Theorem \ref{thm:g2cp2} \\ \hline
    \(\H P^2\) & 2 & No & 2 & Theorem \ref{thm:g2hp2} \\ 
     & 4 & Yes & 1 & Theorem \ref{thm:g4hp2} \\ \hline
     \(\O P^2\) & 2 & No & 2 & Theorem \ref{thm:g2op2} \\ 
     & 4 & No & 2, 3 or 5 & Theorem \ref{thm:g4op2}\\
     & 8 & Yes & 1 & Theorem \ref{thm:g8op2} \\
     & 9 & Yes & 1 & Theorem \ref{thm:g9op2} \\
     & 10 & Yes & 1 & Theorem \ref{thm:g10op2} \\
     & 12 & No & 4, 6 or 10 & Theorem \ref{thm:g12op2} \\ \hline
\end{tabularx}
\vspace{5pt}
\end{center}

Note in particular that we have \(\mathcal{G}^2\)-instability of \(\H P^2\) up to homotopy equivalence, so we are able to rule out stability with respect to homeomorphism, and therefore have a subtly stronger result than in~\cite{duan}. 

\begin{rem*}
    A comment should be made about the inexact answers to GSII for \(\O P^2\) when \(k=4\) and \(k=12\). These are the result of certain relations between compositions of elements in the homotopy groups of spheres, and depend on the values of two odd integers when taken modulo 8. These integers arose in the calculation of the $2$-primary homotopy groups of spheres by Toda~\cite{toda} and Oda~\cite{oda_unstable_spheres}. Their being odd sufficed for Toda and Oda's calculations. However, in our case, their being unspecified is an 
    obstruction to precisely enumerating the homotopy types of gyrations. Details are in Propositions \ref{prop:xi} and \ref{prop:theta}.
\end{rem*}
 
The arguments producing the GSII statements in the table above also give classification results. Write $\mathbb{F}$ for one of $\mathbb{C}$, $\mathbb{H}$ or $\mathbb{O}$. Note that if the answer to GSII is $1$ then the classification is that there is a homotopy equivalence $\mathcal{G}^{k}_{\tau}(\mathbb{F}P^{2})\simeq\mathcal{G}^{k}_{0}(\mathbb{F}P^{2})$ for all $\tau\in\pi_{k-1}(SO(n))$. 
The inexact cases when $k=4,12$ for $\mathbb{O}P^{2}$ imply no classification is yet possible. The $k=2$ case is complete.

\begin{letthm}\label{thm:classification}
    Let \(\tau,\omega\in\pi_1(\mathrm{SO}(2m))\) be twistings. Then:
    \begin{enumerate}
        \item[(i)] \(\mathcal{G}_\tau^2(\C P^2)\simeq\mathcal{G}_0^2(\C P^2)\) for all \(\tau\);
        \item [(ii)] \(\mathcal{G}_\tau^2(\H P^2)\simeq\mathcal{G}_\omega^2(\H P^2)\) if and only if \(\tau\simeq\omega\);
        \item [(iii)] \(\mathcal{G}_\tau^2(\O P^2)\simeq\mathcal{G}_\omega^2(\O P^2)\) if and only if \(\tau\simeq\omega\).
    \end{enumerate}
\end{letthm}

This paper is divided into two parts, with Part I containing the theoretical basis for later arguments and Part II being the site of computations. Part I describes a gyration as a certain homotopy cofibre, and in the case of a projective plane gives a careful analysis of the homotopy class of the attaching map for the top dimensional cell. Part II is a systematic study of gyration stability for projective planes, beginning with some general examples in Section \ref{sec:exgystab}. Though we do not present a uniform argument applicable to all cases from the above table, we lay out a general framework in Section \ref{sec:selfequiv}. Detailed computations are contained in Sections \ref{sec:G(hp2)}, \ref{sec:G(op2)-k=<8} and \ref{sec:G(op2)-9to14}, each case relying heavily on known compositions of elements in the homotopy groups of spheres in the relevant dimensional ranges. We conclude by proving Theorem \ref{thm:classification}. 

\subsubsection*{Acknowledgement} During preparation of this work, the first author was by supported EPSRC grant EP/W524621/1 and the Heilbronn Institute for Mathematical Research. The authors also wish to thank the reviewer for their insightful comments.  

\section*{\textbf{Part I: Deviations and Gyrations}}  

This paper works with maps between wedges of spaces in great detail; there are many ways to assemble such maps, so for the sake of clarity we set up the following notation before beginning in earnest. Given based maps \(f:A\rightarrow X\) and \(g:B\rightarrow Y\) we define the \textit{wedge} of \(f\) and \(g\) to be the map \[f\vee g:A\vee B\longrightarrow X \vee Y\] which is to say, \(f\) on the first summand and \(g\) on the second. Furthermore, if \(Y=X\) then we may define the \textit{wedge sum} of \(f\) and \(g\) to be the composite \[f\perp g: A\vee B\xlongrightarrow{f\vee g} X\vee X\xlongrightarrow{\nabla} X\] where \(\nabla\) denotes the fold map. If in addition we have \(A=B\) and \(A\) is a co-\(H\)-space with comultiplication~\(\sigma\), then the \textit{sum} of \(f\) and \(g\) is the composite \[f+g:A\xlongrightarrow{\sigma} A\vee A\xlongrightarrow{f\perp g} X.\]

\section{Maps Between Half-Smashes and a Deviation}  
\label{sec:deviation} 

Let $A$ and $B$ be path-connected spaces. The \emph{right half-smash} is the quotient space \[A\amsrtimes B=(A\times B)/\sim\] obtained by collapsing $B$ to the basepoint. There is a canonical inclusion, projection and quotient map  \[i:A \rightarrow A\amsrtimes B \qquad \pi: A\amsrtimes B \rightarrow A \qquad q:A\amsrtimes B \rightarrow A\wedge B\] where $q$ is given by collapsing $A$ to the basepoint. 

Suppose that there are maps \(f:A \rightarrow C\) and \(g:B \rightarrow D\)  
where $C$ and $D$ are path-connected spaces. Consider the map \[A\amsrtimes B \xlongrightarrow{f\amsrtimes g} C\amsrtimes D.\] The naturality of $\pi$ and $q$ imply that there are commutative diagrams 
\begin{equation} 
  \label{halfsmashnat} 
    \begin{tikzcd}[row sep=3em, column sep = 3em]
        A \arrow[d, "f"] \arrow[r, "i"] & A\amsrtimes B \arrow[d, "f\amsrtimes g"] \\
        C \arrow[r, "i"] & C\amsrtimes D
    \end{tikzcd}
    \qquad
    \begin{tikzcd}[row sep=3em, column sep = 3em]
        A\amsrtimes B \arrow[d, "f\amsrtimes g"] \arrow[r, "\pi"] & A \arrow[d, "f"] \\
        C\amsrtimes D \arrow[r, "\pi"] & C
    \end{tikzcd}
    \qquad
    \begin{tikzcd}[row sep=3em, column sep = 3em]
        A\amsrtimes B \arrow[d, "f\amsrtimes g"] \arrow[r, "q"] & A\wedge B \arrow[d, "f\wedge g"] \\
        C\amsrtimes D \arrow[r, "q"] & C\wedge D
    \end{tikzcd}
\end{equation}

Next, suppose that $A$ is a co-$H$-space with comultiplication $\sigma$. Then $A\amsrtimes B$ is a co-$H$-space with comultiplication 
\[\overline{\sigma}:A\amsrtimes B \xrightarrow{\sigma\amsrtimes 1} (A\vee A)\amsrtimes B \xrightarrow{\cong} (A\amsrtimes B)\vee (A\amsrtimes B).\] The following lemma is well known but we give its statement and proof to be explicit about the choices of maps involved. 
     
\begin{lem} 
  \label{halfsmashcoH} 
   If $A$ is a simply-connected co-$H$-space then the composite \[e:A\amsrtimes B \xrightarrow{\overline{\sigma}} (A\amsrtimes B)\vee(A\amsrtimes B) \xrightarrow{\pi\vee q} A\vee(A\wedge B)\] is a homotopy equivalence. This is natural for co-$H$-maps \(A \rightarrow C\) between simply-connected co-$H$-spaces and any map \(B \rightarrow D\). 
\end{lem} 

\begin{proof} 
Consider the homotopy cofibration \(A \xrightarrow{i} A\amsrtimes B \xrightarrow{q} A\wedge B\). Observe that $\pi$ is a left inverse for $i$. Thus the map $e$ splits this homotopy cofibration, implying that it induces an isomorphism in homology. As $A$ is simply-connected, so is $A\amsrtimes B$. Therefore Whitehead's Theorem implies that $e$ is a homotopy equivalence. 

As the co-$H$-structure on $A\amsrtimes B$ is induced by that from $A$, the naturality of $\pi$ and $q$ imply the naturality of $e$ for co-$H$-maps \(A \rightarrow C\) between simply-connected co-$H$-spaces and any map \(B \rightarrow D\). 
\end{proof} 

Now, suppose that $A$ and $C$ are simply-connected co-$H$-spaces and there is a map \(f:A \rightarrow C\) which is not necessarily a co-$H$-map, and consider the diagram 
\begin{equation} 
   \label{devdgrm} 
    \begin{tikzcd}[row sep=3em, column sep = 3em]
        A\amsrtimes B \arrow[d, "f\amsrtimes g"] \arrow[r, "e"] & A\vee(A\wedge B) \arrow[d, "f\vee(f\wedge g)"] \\
        C\amsrtimes D \arrow[r, "e"] & C\vee(C\wedge D)
    \end{tikzcd}
\end{equation}
If $f$ is a co-$H$-map then the naturality statement of Lemma~\ref{halfsmashcoH} implies that it homotopy commutes. However, if $f$ is not a co-$H$-map then it may not homotopy commute. Analogously to the co-$H$-deviation of a map, let \[\delta:A\amsrtimes B \rightarrow C\vee (C\wedge D)\] be the difference $\delta=e\circ(f\amsrtimes g)-((f\vee(f\wedge g))\circ e)$ of the two directions around the diagram. In particular, \(\delta\simeq\ast\) if and only if \(f\amsrtimes g\simeq f\vee(f\wedge g)\). Analysing this deviation $\delta$ is the objective of the rest of this section. 

\begin{lem} 
   \label{dev} 
   The composition \(A\amsrtimes B \xrightarrow{\delta} C\vee(C\wedge D)\xrightarrow{I} C\times (C\wedge D)\) is null homotopic, where~$I$ is the inclusion of the wedge into the product. 
\end{lem} 

\begin{proof} 
Since $\overline{\sigma}$ is a comultiplication, $I\circ\overline{\sigma}\simeq\Delta$, where $\Delta$ is the diagonal map. The naturality of~$I$ and the definition of $e$ then imply that $I\circ e=I\circ (\pi\vee q)\circ\overline{\sigma}=(\pi\times q)\circ I\circ\overline{\sigma}\simeq(\pi\times q)\circ\Delta$. Therefore the naturality of $\Delta$ and $I$ give homotopy commutative diagrams 
\begin{equation*}
    \begin{tikzcd}[row sep=3em, column sep=3.5em]
        A\amsrtimes B \arrow[d, "\Delta"] \arrow[r, "f\amsrtimes g"] & C\amsrtimes D \arrow[d, "\Delta"] \arrow[r, "e"] & C\vee(C\wedge D) \arrow[d, "I"] \\
        (A\amsrtimes B)\times(A\amsrtimes B) \arrow[r, "(f\amsrtimes g)\times(f\amsrtimes g)"] & (C\amsrtimes D)\times(C\amsrtimes D) \arrow[r, "\pi\times q"] & C\times(C\wedge D)
    \end{tikzcd}
\end{equation*}
\begin{equation*}
    \begin{tikzcd}[row sep=3em, column sep=3em]
        A\amsrtimes B \arrow[d, "\Delta"] \arrow[r, "e"] & A\vee(A\wedge B) \arrow[d, "I"] \arrow[r, "f\vee(f\wedge g)"] & C\vee(C\wedge D) \arrow[d, "I"] \\
        (A\amsrtimes B)\times(A\amsrtimes B) \arrow[r, "\pi\times q"] & A\times(A\wedge B) \arrow[r, "f\times(f\wedge g)"] & C\times(C\wedge D)
    \end{tikzcd}
\end{equation*}
In the first diagram, along the bottom row the naturality of $\pi$ and $q$ imply that 
$\pi\circ(f\amsrtimes g)\simeq f\circ\pi$ and $q\circ(f\amsrtimes g)\simeq (f\wedge g)\circ q$. Thus the homotopy commutativity of the diagram gives 
\[
I\circ e\circ (f\amsrtimes g)\simeq((f\circ\pi))\times((f\wedge g)\circ q)\circ\Delta.
\] 
Notice that the right side of this homotopy is exactly the counter-clockwise direction around the second diagram, which is homotopic to $I\circ(f\vee(f\wedge g))\circ e$. Thus $I\circ e\circ (f\amsrtimes g)\simeq I\circ(f\vee(f\wedge g))\circ e$. In general, for a co-\(H\)-space \(X\), left-distributivity holds for sums in $[X,Y]$ when composed with a map \(Y \rightarrow Z\), so in our case we obtain 
\[
I\circ\big(e\circ(f\amsrtimes g)-(f\vee(f\wedge g))\circ e\big)\simeq\ast.
\] 
But by definition, $\delta=e\circ(f\amsrtimes g)-(f\vee(f\wedge g))\circ e$, implying that $I\circ \delta\simeq\ast$, as asserted. 
\end{proof} 

In general, let $X$ and $Y$ be path-connected spaces. 
The \emph{join} of $X$ and $Y$ is the quotient space 
\[X\ast Y=(X\times I\times Y)/\sim\] 
where $I=[0,1]$ is the unit interval with basepoint $0$ and 
the relation is given by $(a,0,b)\sim (a,0,b)$, $(a,1,b)\sim (a,1,b')$ and $(\ast,t,\ast)\sim (\ast,0,\ast)$ for all $a,a'\in A$, $b,b'\in B$ and $t\in I$. It is well known that there is a homotopy equivalence $X\ast Y\simeq\Sigma X\wedge Y$. 
Let \(i_1:X \rightarrow X\vee Y\) 
and \(i_2:Y \rightarrow X\vee Y\) be the inclusions of the left and right wedge summands respectively. Define maps $ev_{1}$ and $ev_{2}$ by the composites 
\[ev_1:\Sigma\Omega X \xrightarrow{ev} X \xrightarrow{i_1} X\vee Y\] 
\[ev_2:\Sigma\Omega Y \xrightarrow{ev} Y \xrightarrow{i_2} X\vee Y\] 
where $ev$ is the canonical evaluation map. Ganea \cite{ganea65} showed that there is a homotopy fibration 
\[\Omega X\ast\Omega Y \xrightarrow{[ev_1,ev_2]} X\vee Y \xrightarrow{I} X\times Y\]  
where $I$ is the inclusion of the wedge into the product and $[ev_{1},ev_{2}]$ is the Whitehead product of~$ev_{1}$ and $ev_{2}$. In our case, the null homotopy for $\delta\circ I$ in Lemma~\ref{dev} implies that there is a lift 
\begin{equation} 
  \label{devlift} 
    \begin{tikzcd}[row sep=3em, column sep=3em]
        & \Omega C \ast \Omega(C\wedge D) \arrow[d, "{[ev_1,ev_2]}"] \\
        A\amsrtimes B \arrow[ur, "\lambda"] \arrow[r, "\delta"] & C\vee(C\wedge D)
    \end{tikzcd}
\end{equation}
for some map $\lambda$. In particular, $\delta$ factors through a Whitehead product. To go further we specialize. Suppose that $A=S^{2m-1}$, $C=S^{m}$ and 
\(f:S^{2m-1} \rightarrow S^{m}\) is some map. Suppose also that $B=D=S^{k-1}$ and $g$ is the identity map. Then~(\ref{devlift}) takes the form 
\begin{equation} 
  \label{devlift2} 
    \begin{tikzcd}[row sep=3em, column sep=3em]
        & \Omega S^{m}\ast\Omega S^{m+k-1} \arrow[d, "{[ev_1,ev_2]}"] \\
        S^{2m-1}\amsrtimes S^{k-1} \arrow[ur, "\lambda"] \arrow[r, "\delta"] & S^{m}\vee S^{m+k-1}.
    \end{tikzcd}
\end{equation}
As before, let \(i_1\) and \(i_2\) be the inclusions of the left and right wedge summands into \(S^{m}\vee S^{m+k-1}\), respectively. 

\begin{lem} 
   \label{devid} 
   The deviation $\delta$ is homotopic to the composite \[S^{2m-1}\amsrtimes S^{k-1} \xlongrightarrow{q} S^{2m+k-2} \xlongrightarrow{r} S^{2m+k-2} \xlongrightarrow{[i_1,i_2]} S^{m}\vee  S^{m+k-1}\] for some map of degree $r\in\Z$. 
\end{lem} 

\begin{proof} 
Observe that the $(2m+k-2)$-skeleton of $\Omega S^{m}\ast\Omega S^{m+k-1}$ is $S^{2m+k-2}$. Let 
\[\iota:S^{2m+k-2} \rightarrow \Omega S^{m}\ast\Omega S^{m+k-1}\] be the inclusion of the bottom cell. Since $S^{2m-1}\amsrtimes S^{k-1}$ is homotopy equivalent to $S^{2m-1}\vee S^{2m+k-2}$, for connectivity and dimension reasons the map \(S^{2m-1}\amsrtimes S^{k-1} \xrightarrow{\lambda} \Omega S^{m}\ast\Omega S^{m+k-1}\) factors as the composite \[S^{2m-1}\amsrtimes S^{k-1} \xlongrightarrow{q} S^{2m+k-2} \xlongrightarrow{r} S^{2m+k-2} \xlongrightarrow{\iota} \Omega S^{m}\ast\Omega S^{m+k-1}\] for some map of degree \(r\). On the other hand, the composite \[S^{2m+k-2} \xlongrightarrow{\iota} \Omega S^{m}\ast\Omega S^{m+k-1} \xlongrightarrow{[ev_{1},ev_{2}]} S^{m}\vee S^{m+k-1}\] is homotopic to the Whitehead product $[i_{1},i_{2}]$. Thus there is a sequence \[\delta\simeq [ev_{1},ev_{2}]\circ\lambda\simeq [ev_{1},ev_{2}]\circ\iota\circ r\circ q\simeq [i_{1},i_{2}]\circ r\circ q\] which gives the assertion. 
\end{proof} 

\begin{rem} 
The choices made for the specialization of the spaces $A$, $B$, $C$, $D$ and the maps $f$ and~$g$ is partly due to the applications in mind, and partly so that the deviation $\delta$ has the precise form in Lemma~\ref{devid}. Note that if $A$ is a sphere of dimension $d$ for $m<d<2m-1$ then $\lambda$ would be null homotopic for connectivity and dimension reasons. But in this range it is a classical result that $f:S^d\rightarrow S^n$ is a suspension, in which case it was already clear that $\delta$ is null homotopic, as suspensions are in particular co-\(H\)-maps. If $d>2m-1$ then $\lambda$ becomes more complex since it involves maps from spheres into the $(d+k-2)$-skeleton of $\Omega S^{m}\ast\Omega S^{m+k-1}$. If $d< 3m+k-3$ then this skeleton is still $S^{2m+k-2}$ but now~$\lambda$ involves a torsion homotopy group of that sphere, and if $d\geq 3m+k-3$ then the skeleton involves more cells. 
\end{rem} 

Summarizing the results in this section we obtain the following.  

\begin{prop} 
    \label{devidsphere} 
    Suppose that there is a map \(f:S^{2m-1} \rightarrow S^{m}\) and fix an integer $k\geq 2$. Then the deviation from the diagram 
    \begin{equation*}
        \begin{tikzcd}[row sep=3em, column sep = 3em]
        S^{2m-1}\amsrtimes S^{k-1} \arrow[d, "f\amsrtimes 1"] \arrow[r, "e"] & S^{2m-1}\vee S^{2m+k-2} \arrow[d, "f\vee\Sigma^{k-1}f"] \\
        S^{m}\amsrtimes S^{k-1} \arrow[r, "e"] & S^{m}\vee S^{m+k-1}
    \end{tikzcd}
    \end{equation*}
    homotopy commuting is homotopic to the composite \[S^{2m-1}\amsrtimes S^{k-1} \xlongrightarrow{q} S^{2m+k-2} \xlongrightarrow{r} S^{2m+k-2} \xlongrightarrow{[i_1,i_2]} S^{m}\vee  S^{m+k-1}\] for some integer $r$.~$\qqed$ 
\end{prop}

\section{Properties of the Homotopy Equivalence $e$} 
\label{sec:einverse} 

In this section we construct an explicit inverse for the homotopy equivalence 
\[e:S^{m}\amsrtimes S^{k-1}\rightarrow S^{m}\vee S^{m+k-1}\] for \(k\geq2\) and relate it to the deviation \(\delta\) and the map $f\amsrtimes 1$ of Section~\ref{sec:deviation}. An ingredient map will play a special role, which we call \(j\). 

\begin{lem} 
   \label{ej} 
   For $k\geq 2$ there is a map \(S^{m+k-1} \xrightarrow{j} S^{m}\amsrtimes S^{k-1}\) such that $q\circ j$ is homotopic to the identity map and $\pi\circ j$ is null homotopic. If $k\leq m-1$ then $j$ is a co-$H$-map. 
\end{lem} 

\begin{proof} 
The definition of the homotopy equivalence $e$ implies that there is a homotopy commutative square 
\begin{equation*}
    \begin{tikzcd}[row sep=3em, column sep = 3em]
        S^m\amsrtimes S^{k-1} \arrow[d, "\pi"] \arrow[r, "e"] & S^m \vee S^{m+k-1} \arrow[d, "p_1"] \\
        S^m \arrow[r, equal] & S^m
    \end{tikzcd}
\end{equation*}
where $p_{1}$ is the pinch map to the first wedge summand. Thus the homotopy fibre of $\pi$ is homotopy equivalent to the homotopy fibre of $p_{1}$, which by~\cite{ganea65} is homotopy equivalent to $S^{m+k-1}\amsrtimes\Omega S^{m}$, which in turn is homotopy equivalent to $\Omega S^{m}\ast S^{k-1}$. Let~$j$ be the composite \[j:S^{m+k-1} \xlongrightarrow{\iota} \Omega S^{m}\ast S^{k-1} \longrightarrow S^{m}\amsrtimes S^{k-1}\] where the map \(\iota\) is again the inclusion of the bottom cell and the right map is from the fibre to the total space in the homotopy fibration for $\pi$. Then $\pi\circ j$ is null homotopic. 

By the Blakers-Massey Theorem, the homotopy fibration 
\[
    \Omega S^{m}\ast S^{k-1} \rightarrow S^{m}\amsrtimes S^{k-1} \xrightarrow{\pi} S^{m}
\] 
is a homotopy cofibration in dimensions~$\leq 2m+k-2$. In particular, as $m+k-2\leq 2m+k-2$ for any $m\geq 2$, the composite \(S^{m+k-1} \xrightarrow{j} S^{m}\amsrtimes S^{k-1} \xrightarrow{\pi} S^{m}\) is a homotopy cofibration in dimensions~$\leq m+k-1$. Consequently, as the spaces in this composite are of dimension~$\leq m+k-1$, the composite is a homotopy cofibration in all dimensions. Observe that $j_{\ast}$ induces an isomorphism on homology in degree $m+k-1$, as does $q_{\ast}$, so $q\circ j$ induces a homology isomorphism in all degrees and is therefore a homotopy equivalence. As $q\circ j$ is a self-map of $S^{m+k-1}$, being a homotopy equivalence implies that it is homotopic to a map of degree \(\pm1\). If $q\circ j\simeq -1$ then adjust $j$ by pre-composing it with the map of degree~$-1$. Then $q\circ j\simeq 1$ and this adjustment does not affect the fact that $\pi\circ j$ is null homotopic. 

Finally, observe that \(S^{m+k-1} \xrightarrow{j} S^{m}\amsrtimes S^{k-1} \simeq S^{m}\vee S^{m+k-1}\) is in the stable range if $k\leq m-1$, implying that it is the suspension of a map \(S^{m+k-2} \rightarrow S^{m-1}\amsrtimes S^{k-1} \simeq S^{m-1}\vee S^{m+k-2}\). Thus, if $k\leq m-1$, then $j$ is a co-$H$-map. 
\end{proof}  

Next, we relate $i$ and $j$ to the homotopy equivalence $e$ in Section~\ref{sec:deviation}. Write \(i_1:X\rightarrow X\vee Y\) and \(i_2:Y\rightarrow X\vee Y\) for the inclusions of the left and right wedge summands respectively. Recalling the notation for the wedge sum, observe that \(i_1 \perp i_2:X\vee Y\rightarrow X\vee Y\) is the identity map. 

\begin{lem} 
   \label{einv} 
   Let \(k\geq 2\). The following hold: 
   \begin{itemize} 
      \item[(i)] the composite \( S^{m} \xrightarrow{i} S^{m}\amsrtimes S^{k-1} \xrightarrow{e} S^{m}\vee S^{m+k-1}\) is homotopic to~$i_{1}$;  
      \item[(ii)] if \(k\leq m-1\) then the composite \(S^{m+k-1} \xrightarrow{j} S^{m}\amsrtimes S^{k-1} \xrightarrow{e} S^{m}\vee S^{m+k-1}\) is homotopic to~$i_{2}$;  
      \item[(iii)] if \(k\leq m-1\) then the map \(S^{m}\vee S^{m+k-1} \xrightarrow{i\perp j} S^{m}\amsrtimes S^{k-1}\) is a homotopy equivalence that is the inverse of $e$. 
   \end{itemize} 
\end{lem} 

\begin{proof} 
Consider the diagram 
\begin{equation*}
    \begin{tikzcd}[row sep=3em, column sep = 3em]
        S^m\arrow[d, "i"] \arrow[r, "\sigma"] & S^m \vee S^m \arrow[d, "i\vee i"] \arrow[dr, "1\vee\ast"] & \\
        S^m\amsrtimes S^{k-1}  \arrow[r, "\overline{\sigma}"] & (S^m\amsrtimes S^{k-1})\vee(S^m\amsrtimes S^{k-1}) \arrow[r, "\pi\vee q"] & S^m \vee S^{m+k-1}
    \end{tikzcd}
\end{equation*}
Observe that as $i$ is the inclusion of the bottom cell, it is a co-$H$-map, so the left square homotopy commutes. The right-hand triangle homotopy commutes since $\pi$ is a left inverse for $i$ and $q\circ i$ is null homotopic for connectivity and dimension reasons. The lower row is the definition of $e$ while the upper composite $(1\vee\ast)\circ\sigma$ is the inclusion $i_{1}$ of the left wedge summand. The homotopy commutativity of the diagram therefore implies that $e\circ i\simeq i_{1}$. This proves (i).

Next, consider the diagram 
\begin{equation*}
    \begin{tikzcd}[row sep=3em, column sep = 3em]
        S^{m+k-1}\arrow[d, "j"] \arrow[r, "\sigma"] & S^{m+k-1} \vee S^{m+k-1} \arrow[d, "j\vee j"] \arrow[dr, "\ast\vee1"] & \\
        S^m\amsrtimes S^{k-1}  \arrow[r, "\overline{\sigma}"] & (S^m\amsrtimes S^{k-1})\vee(S^m\amsrtimes S^{k-1}) \arrow[r, "\pi\vee q"] & S^m \vee S^{m+k-1}
    \end{tikzcd}
\end{equation*}
Since $k\leq m-1$, Lemma~\ref{ej} implies that $j$ is a co-$H$-map, so the left square homotopy commutes. The right-hand triangle homotopy commutes by Lemma~\ref{ej}. Again, the lower row is the definition of $e$ while the upper composite $(\ast\vee 1)\circ\sigma$ is the inclusion $i_{2}$ of the right wedge summand. The homotopy commutativity of the diagram therefore implies that $e\circ j\simeq i_{2}$. This proves (ii).

For (iii), it follows from (i) and (ii) that $e\circ(i\perp j)\simeq(e\circ i)\perp(e\circ j)\simeq i_{1}\perp i_{2}$, which is precisely the identity map on $S^{m}\vee S^{m+k-1}$.  
\end{proof} 

To prevent confusion, we use \(j^\dagger\) to denote the version of \(j\) for \(S^{2m+k-2} \rightarrow S^{2m-1}\amsrtimes S^{k-1}\). The following two lemmas are instrumental in what follows; they relate $j^\dagger$ to the behaviour of the deviation $\delta$ from Lemma \ref{devid}. 

\begin{lem} 
   \label{Dj} 
   For $k\geq 2$, the composite \(S^{2m+k-2} \xrightarrow{j^\dagger} S^{2m-1}\amsrtimes S^{k-1} \xrightarrow{\delta} S^{m}\vee S^{m+k-2}\)  is homotopic to $r\cdot [i_{1},i_{2}]$, where $r$ is the integer appearing in Proposition~\ref{devidsphere}. 
\end{lem} 

\begin{proof} 
By Proposition~\ref{devidsphere}, the deviation $\delta$ is homotopic to the composite \[S^{2m-1}\amsrtimes S^{k-1} \xlongrightarrow{q} S^{2m+k-2} \xlongrightarrow{r} S^{2m+k-2} \xlongrightarrow{[i_1,i_2]} S^{m}\vee  S^{m+k-1}\] for some map of degree $r$. By Lemma~\ref{ej}, $q\circ j^\dagger$ is homotopic to the identity map on $S^{2m+k-2}$, so we obtain $\delta\circ j^\dagger\simeq r\cdot [i_{1},i_{2}]$. 
\end{proof} 

Finally, we use Lemmas \ref{ej}, \ref{einv} and \ref{Dj} to relate $j^\dagger$ to the map $f\amsrtimes 1$ in Section~\ref{sec:deviation}. 

\begin{lem} 
   \label{halfsmashj} 
   If $2\leq k\leq 2m-2$ then there is a homotopy commutative diagram 
    \begin{equation*}
        \begin{tikzcd}[row sep=3em, column sep = 3em]
            S^{2m+k-2} \arrow[r, "j^\dagger"] \arrow[d, equal] & S^{2m-1} \amsrtimes S^{k-1} \arrow[r, "f\amsrtimes 1"] & S^m \amsrtimes S^{k-1} \arrow[d, "e"] \\
            S^{2m+k-2}  \arrow[rr, "{i_2\circ\Sigma^{k-1}f +r\cdot [i_1,i_2]}"] && S^m \vee S^{m+k-1}
        \end{tikzcd}
    \end{equation*}
    where $r$ is the integer appearing in Proposition~\ref{devidsphere}. 
\end{lem} 

\begin{proof}  
By definition, $\delta=e\circ(f\amsrtimes 1)-((f\vee (f\wedge 1))\circ e$. As the identity map $1$ is for $S^{k-1}$, we may rewrite this as $\delta=e\circ(f\amsrtimes 1)-(f\vee\Sigma^{k-1} f)\circ e$. Rearranging gives $e\circ (f\amsrtimes 1)\simeq (f\vee\Sigma^{k-1} f)\circ e +\delta$. Now precompose with $j^\dagger$. Since $k\leq 2m-2$, by Lemma~\ref{ej}, $j^\dagger$ is a co-$H$-map, giving \[e\circ (f\amsrtimes 1)\circ j^\dagger\simeq ((f\vee\Sigma^{k-1} f)\circ e+\delta)\circ j^\dagger\simeq (f\vee\Sigma^{k-1} f)\circ e\circ j^\dagger + \delta\circ j^\dagger.\] By Proposition~\ref{einv}~(ii), $e\circ j\simeq i_{2}$, implying by the naturality of $i_{2}$ that \[(f\vee\Sigma^{k-1} f)\circ e\circ j^\dagger\simeq (f\vee\Sigma^{k-1} f)\circ i_{2}\simeq i_{2}\circ\Sigma^{k-1} f.\] By Lemma~\ref{Dj}, $\delta\circ j^\dagger\simeq r\cdot [i_{1},i_{2}]$. Therefore \[e\circ(f\amsrtimes 1)\circ j^\dagger\simeq i_{2}\circ\Sigma^{k-1} f + r\cdot [i_{1},i_{2}],\] and hence the diagram in the statement of the Lemma homotopy commutes. 
\end{proof} 

\section{Gyrations} 
\label{sec:gyration} 

Let $M$ be a simply-connected Poincar\'{e} Duality complex of dimension $n$. Let $\overline{M}$ be the $(n-1)$-skeleton of $M$. Then there is a homotopy cofibration \[S^{n-1} \xrightarrow{f_M} \overline{M} \rightarrow M\] where $f_M$ is the attaching map for the top-cell. Let $k\geq 2$ be an integer and take a map \(\tau:S^{k-1}\rightarrow \mathrm{SO}(n)\), then using the standard linear action of \(\mathrm{SO}(n)\) on \(S^{n-1}\) define the map \[t:S^{n-1}\times S^{k-1}\rightarrow S^{n-1}\times S^{k-1}\] by \(t(a, x)=(\tau(x)\cdot a,x)\). 

\begin{defn}\label{def:gy}
    Let \(k\geq2\) be an integer and let \(M\) be an \(n\)-dimensional Poincar\'e Duality complex. Define the \textit{\(k\)-gyration of \(M\) by \(\tau\)} to be the space defined by the (strict) pushout
        \begin{equation}\label{gyrationpo}
            \begin{tikzcd}[row sep=3em, column sep=3em]
                S^{n-1}\times S^{k-1} \arrow[r, "1\times \iota"] \arrow[d, "(f_M\times 1)\circ t"] & S^{n-1}\times D^k \arrow[d] \\
                \overline{M}\times S^{k-1} \arrow[r] & \mathcal{G}^k_\tau(M) 
            \end{tikzcd}
        \end{equation} 
    where $\iota$ is the inclusion of the boundary of the disc. When the context is clear, we will usually just write \textit{gyration} for \(\mathcal{G}^k_\tau(M)\). 
\end{defn}

If $\tau$ is trivial, then \(t\) is the identity map and this pushout is a $(k-1,n)$-type surgery on $M\times S^{k-1}$. Otherwise, the surgery is twisted by the action of $\tau$ (considered as a diffeomorphism). In either case, the gyration is an $(n+k-1)$-dimensional Poincar\'{e} Duality complex. Since the disc $D^{k}$ is contractible, from~(\ref{gyrationpo}) we obtain a homotopy pushout 
\begin{equation} 
  \label{gyrationhpo} 
    \begin{tikzcd}[row sep=3em, column sep = 3em]
        S^{n-1}\times S^{k-1} \arrow[d, "(f_M\times1)\circ t"] \arrow[r, "\pi"] & S^{n-1} \arrow[d] \\
        \overline{M}\times S^{k-1} \arrow[r] & \mathcal{G}^{k}_{\tau}(M)
    \end{tikzcd}
\end{equation}
where $\pi$ is the projection. The clockwise direction around~(\ref{gyrationhpo}) is null homotopic when restricted to~$S^{k-1}$, so the commutativity of the diagram implies the same is true in the counter-clockwise direction around the diagram. Moreover, observe that $t$ is the identity map when restricted to~$S^{k-1}$. Therefore, if in general \(j_{2}:B\longrightarrow A\times B\) is the inclusion of the second factor, then there is a homotopy commutative diagram in which the rows are homotopy cofibrations 
\begin{equation} 
    \label{tau'def} 
    \begin{tikzcd}[row sep=3em, column sep = 3em]
        S^{k-1} \arrow[d, equal] \arrow[r, "{j_2}"] & S^{n-1}\times S^{k-1} \arrow[d, "t"] \arrow[r] & S^{n-1}\amsrtimes S^{k-1} \arrow[d, "t'"] \\
        S^{k-1} \arrow[d, equal] \arrow[r, "{j_2}"] & S^{n-1}\times S^{k-1} \arrow[d, "f_M\times1"] \arrow[r] & S^{n-1}\amsrtimes S^{k-1} \arrow[d, "f_M\amsrtimes1"] \\
        S^{k-1} \arrow[r, "{j_2}"] & \overline{M}\times S^{k-1} \arrow[r] & \overline{M}\amsrtimes S^{k-1}.
    \end{tikzcd}
\end{equation}
The map $t'$ is an induced map of cofibres and the map of cofibrations in the lower rectangle follows from the naturality of the right half-smash. Thus, 
collapsing out $S^{k-1}$ in~(\ref{gyrationhpo}) results in a homotopy pushout 
\begin{equation} 
  \label{gyrationhalfsmash} 
    \begin{tikzcd}[row sep=3em, column sep = 3em]
        S^{n-1}\amsrtimes S^{k-1} \arrow[d, "(f_M\amsrtimes1)\circ t'"] \arrow[r, "\pi"] & S^{n-1} \arrow[d] \\
        \overline{M}\amsrtimes S^{k-1} \arrow[r] & \mathcal{G}^{k}_{\tau}(M).
    \end{tikzcd}
\end{equation}
Lemma~\ref{ej} implies that there is a homotopy cofibration \(S^{n+k-2} \xrightarrow{j} S^{n-1}\amsrtimes S^{k-1} \xrightarrow{\pi} S^{n-1}\). 
Writing this as a homotopy pushout 
\[\begin{tikzcd}[row sep=3em, column sep = 3em]
     S^{n+k-2} \arrow[d, "j"] \arrow[r] & \ast \arrow[d] \\
     S^{n-1}\amsrtimes S^{k-1} \arrow[r, "\pi"] & S^{n-1}
\end{tikzcd}\]
and juxtaposing it over~(\ref{gyrationhalfsmash}) shows that $\mathcal{G}^{k}_{\tau}(M)$ is the homotopy pushout of the trivial map \(\namedright{S^{n+k-2}}{}{\ast}\) and the composite \[\phi_{\tau}: S^{n+k-2} \xrightarrow{j} S^{n-1}\amsrtimes S^{k-1} \xrightarrow{t'} S^{n-1}\amsrtimes S^{k-1} \xrightarrow{f_M\amsrtimes 1} \overline{M}\amsrtimes S^{k-1}.\] Thus we obtain the following. 

\begin{lem} 
   \label{gyrationcofib} 
   For $k\geq 2$ there is a homotopy cofibration \(S^{n+k-2} \xrightarrow{\phi_{\tau}} \overline{M}\amsrtimes S^{k-1} \rightarrow \mathcal{G}^{k}_{\tau}(M)\).~$\qqed$ 
\end{lem}   

\begin{rem} 
Notice that as $M$ is simply-connected and $n$-dimensional, Poincar\'{e} Duality implies that~$\overline{M}$ is at most $(n-2)$-dimensional, and therefore $\overline{M}\amsrtimes S^{k-1}$ is at most $(n+k-3)$-dimensional. Therefore~$\phi_{\tau}$ attaches the top dimensional cell to the Poincar\'{e} Duality complex $\mathcal{G}^{k}_{\tau}(M)$. 
\end{rem}  

\begin{rem} 
  \label{gyrationcofibremark} 
  As a special case worth noting, if the twisting \(\tau\) is trivial then $t$, and hence $t'$, is the identity map. Thus, writing $\mathcal{G}^{k}_{0}(M)$ for a gyration by the trivial twisting, the attaching map \(\phi_0\) is homotopic to the composite \((f_M\amsrtimes 1)\circ j\).
\end{rem} 

\begin{rem} 
Lemma~\ref{gyrationcofib} identifies the $(n+k-2)$-skeleton of the $(n+k-1)$-dimensional Poincar\'{e} Duality complex $\mathcal{G}^{k}_{\tau}(M)$ as $\overline{M}\amsrtimes S^{k-1}$. This reproduces a result of Basu-Ghosh~\cite[Proposition 6.9]{basu-ghosh} using a different argument, while saying more by identifying the attaching map for the top cell. 
\end{rem}

The goal is to understand the attaching map $\phi_{\tau}$ in order to better understand the 
twisted gyration $\mathcal{G}^{k}_{\tau}(M)$. To do so, we specialize to make use of 
Sections~\ref{sec:deviation} and~\ref{sec:einverse}. Suppose that $M$ is one of 
$\mathbb{C}P^{2}$, $\mathbb{H}P^{2}$ or $\mathbb{O}P^{2}$. Then there are homotopy cofibrations \[S^{3} \xrightarrow{\eta_2} S^{2} \rightarrow \mathbb{C}P^{2} \qquad S^{7} \xrightarrow{\nu_4} S^{4} \rightarrow \mathbb{H}P^{2} \qquad S^{15} \xrightarrow{\sigma_8} S^{8} \rightarrow \mathbb{O}P^{2}\] where $\eta_2$, $\nu_4$ and $\sigma_8$ are maps of Hopf invariant one. Collectively, these may be described by a homotopy cofibration \[S^{2m-1} \xrightarrow{f} S^{m} \rightarrow \mathbb{F}P^{2}\]  where $f=\eta_2$ and $\mathbb{F}=\mathbb{C}$ if $m=2$, $f=\nu_4$ and $\mathbb{F}=\mathbb{H}$ if $m=4$, and $f=\sigma_8$ and $\mathbb{F}=\mathbb{O}$ if $m=8$. With these dimensions for the domain and range of $f$, the map $j$ appearing in the definition of $\phi_{\tau}$  is relabelled as~$j^{\dagger}$ as in Section~\ref{sec:einverse}. To analyse $\phi_{\tau}=(f\amsrtimes 1)\circ t'\circ j^\dagger$ we proceed to first consider $f\amsrtimes 1$, then $ t'$, and finally put these together and compose with $j^\dagger$.   

The homotopy cofibration for $\mathbb{F}P^{2}$ implies we are in the context of 
Proposition~\ref{devidsphere}, which describes the deviation from $f\amsrtimes 1$ being 
$f\vee\Sigma^{k-1} f$, up to the homotopy equivalence $e$. In this case the indeterminate degree 
map $r$ in the description of the deviation can be made more precise.  

\begin{prop} \label{devidHopf} 
    Let \(f:S^{2m-1} \rightarrow S^{m}\) be one of $\eta_2$, $\nu_4$ or $\sigma_8$. If $k\geq2$ then the deviation \(\delta\) from the diagram 
    \begin{equation*}
        \begin{tikzcd}[row sep=3em, column sep = 3em]
        S^{2m-1}\amsrtimes S^{k-1} \arrow[d, "f\amsrtimes1"] \arrow[r, "e"] & S^{2m-1}\vee S^{2m+k-2} \arrow[d, "f\vee\Sigma^{k-1} f"] \\
        S^m\amsrtimes S^{k-1} \arrow[r, "e"] & S^m\vee S^{m+k-1}
        \end{tikzcd}
    \end{equation*}
    homotopy commuting is homotopic to the composite 
    \[S^{2m-1}\amsrtimes S^{k-1} \xrightarrow{q} S^{2m+k-2} \xrightarrow{[i_1,i_2]} S^m \vee  S^{m+k-1}.\]
\end{prop} 

\begin{proof} 
By Proposition~\ref{devidsphere} we have that $\delta$ is homotopic to the composite \[S^{m-1}\amsrtimes S^{k-1} \xrightarrow{q} S^{2m+k-2} \xrightarrow{r} S^{2m+k-2} \xrightarrow{[i_1,i_2]} S^{m}\vee  S^{m+k-1}\] for some map of degree $r$. So it remains to show that $r=1$. Since the homotopy cofibre of $e\circ(f\amsrtimes 1)$ is $\mathcal{G}^{k}_{0}(\mathbb{F}P^{2})$ by Remark~\ref{gyrationcofibremark}, the homotopy commutative square in Lemma~\ref{halfsmashj} implies that there is a homotopy cofibration diagram 
\begin{equation*}
    \begin{tikzcd}[row sep=3em, column sep = 3em]
    S^{2m+k-2} \arrow[d, equal] \arrow[rr, "(f\amsrtimes 1)\circ j^\dagger"] && S^m\amsrtimes S^{k-1} \arrow[d, "e"] \arrow[r] & \mathcal{G}_0^k(\mathbb{F}P^2) \arrow[d, "\varepsilon"] \\
    S^{2m+k-2} \arrow[rr, "{i_2\circ\Sigma^{k-1}f + r\cdot [i_1,i_2]}"] && S^m\vee S^{m+k-1} \arrow[r] & C
    \end{tikzcd}
\end{equation*}
that defines the space $C$, and where $\varepsilon$ is some induced map of homotopy cofibres. Since $e$ is a homotopy equivalence, it induces an isomorphism in homology, so the Five-Lemma implies that~$\varepsilon$ also induces an isomorphism in homology. Since spaces are simply-connected, $\varepsilon$ is a homotopy equivalence by Whitehead's Theorem. 

Now use $\varepsilon$ to compare cup products. The homotopy cofibration defining $C$ implies that there is a module isomorphism $H^{\ast}(C;\mathbb{Z})\cong\mathbb{Z}\{a,b,c\}$ where $\vert a\vert=m$, $\vert b\vert=m+k-1$ and $\vert c\vert= 2m+k-1$. Since~$\varepsilon$ is a homotopy equivalence we obtain a module isomorphism $H^{\ast}(\mathcal{G}^{k}_{0}(\mathbb{F}P^{2});\mathbb{Z})\cong\mathbb{Z}\{x,y,z\}$ where $\vert x\vert=m$, $\vert y\vert=m+k-1$ and $\vert z\vert=2m+k-1$, and $\varepsilon^{\ast}$ sends $a,b,c$ to $x,y,z$ respectively. Since $\mathcal{G}^{k}_{0}(\mathbb{F}P^{2})$ is a manifold, by Poincar\'{e} Duality we obtain $x\cup y=z$. As $\varepsilon^{\ast}$ is an algebra map, this implies that $a\cup b=c$. Therefore there is an algebra isomorphism $H^{\ast}(C;\mathbb{Z})\cong H^{\ast}(S^{m}\times S^{m+k-1};\mathbb{Z})$, and a homotopy equivalence between the $(2m+k-3)$-skeletons of $C$ and $S^{m}\times S^{m+k-1}$. 

On the other hand, the cup product structure on $C$ is induced by the attaching map $i_{2}\circ\Sigma^{k-1} f +r\cdot [i_{1},i_{2}]$ for the top-cell of $C$. Since the cup product $a\cup b=c$ detects the Whitehead product $[i_{1},i_{2}]$ and $i_{2}\circ\Sigma^{k-1} f$ cannot be a multiple of $[i_{1},i_{2}]$ due to its image being concentrated in the $S^{m}$ wedge summand, it must be the case that $r=1$. 
\end{proof} 

\begin{cor} \label{fj} 
   If $2\leq k\leq 2m-2$ and $f$ is as in Proposition~\ref{devidHopf} then 
   $(f\amsrtimes 1)\circ j^\dagger\simeq (j\circ\Sigma^{k-1} f) + [i,j]$. 
\end{cor} 

\begin{proof} 
By Proposition~\ref{devidHopf}, $e\circ(f\amsrtimes 1)\simeq (f\vee\Sigma^{k-1})\circ e+\delta$. Since $k\leq 2m-2$, Lemma~\ref{ej} implies that $j^\dagger$ is a co-$H$-map. A co-$H$-map distributes on the right, implying that \[e\circ (f\amsrtimes 1)\circ j^\dagger\simeq (f\vee\Sigma^{k-1} f)\circ e\circ j^\dagger+\delta\circ j^\dagger.\] By Lemma~\ref{einv}~(ii), $e\circ j^\dagger\simeq i_{2}$. Therefore the naturality of $i_{2}$ implies that $(f\vee\Sigma^{k-1} f)\circ e\circ j^\dagger\simeq i_{2}\circ\Sigma^{k-1} f$. By Lemma~\ref{Dj}, and using Proposition~\ref{devidHopf}, we obtain $\delta\circ j^\dagger\simeq [i_{1},i_{2}]$. Therefore \[e\circ (f\amsrtimes 1)\circ j^\dagger\simeq i_{2}\circ\Sigma^{k-1} f + [i_{1},i_{2}].\] Apply $e^{-1}$. In general, the sum of two maps in $[\Sigma X,Y]$ distributes when composed with a map \(Y\rightarrow Z\), so in our case we obtain \[(f\amsrtimes 1)\circ j^\dagger\simeq (e^{-1}\circ i_{2}\circ\Sigma^{k-1} f) + (e^{-1}\circ [i_{1},i_{2}]).\] By Lemma~\ref{einv}~(iii), $e^{-1}\simeq i\perp j$, so $e^{-1}\circ i_{2}=j$ and $e^{-1}\circ [i_{1},i_{2}]=[i,j]$. Therefore \[(f\amsrtimes 1)\circ j^\dagger\simeq (j\circ\Sigma^{k-1} f) + [i,j]\] as asserted.
\end{proof} 

\begin{rem} 
\label{trivialattach} 
By Remark~\ref{gyrationcofibremark}, the attaching map for the top-cell of 
$\mathcal{G}^{k}_{0}(\mathbb{F}P^{2})$ is $(f\amsrtimes 1)\circ j^\dagger$. Corollary~\ref{fj} therefore 
gives an alternate description of this attaching map. This will be generalized to the case 
of $\mathcal{G}^{k}_{\tau}(\mathbb{F}P^{2})$ in  Theorem~\ref{thm:twistattach}.
\end{rem} 

Next, we bring in the twist. Consider the self-equivalence \(S^{2m-1}\amsrtimes S^{k-1} \xrightarrow{t'} S^{2m-1}\amsrtimes S^{k-1}\). Let $t'_1$ and~$t'_2$ be the composites \[t'_1:S^{2m-1} \xrightarrow{i} S^{2m-1}\amsrtimes S^{k-1} \xrightarrow{t'} S^{2m-1}\amsrtimes S^{k-1}\] \[t'_2:S^{2m+k-2} \xrightarrow{j^\dagger} S^{2m-1}\amsrtimes S^{k-1} \xrightarrow{t'} S^{2m-1}\amsrtimes S^{k-1}.\] 
Let \(\overline{\tau}\) be the composite 
\begin{equation}\label{taubardef} 
\overline{\tau}\colon S^{2m+k-2} \xrightarrow{j^\dagger} S^{2m-1}\amsrtimes S^{k-1} \xrightarrow{t'} S^{2m-1}\amsrtimes S^{k-1} \xrightarrow{\pi}  S^{2m-1}.
\end{equation} 

\begin{lem} \label{twisting12} 
   Let \(k\geq 2\). The following hold: 
   \begin{itemize} 
        \item[(i)] $t'_1$ is homotopic to $i$; 
        \item[(ii)] if $k\leq 2m-2$, then $t'_2$ is homotopic to $i\circ\overline{\tau} + j^\dagger$; 
        \item[(iii)] if $\tau$ is the trivial twisting then $t'_2= j^\dagger$ and $\overline{\tau}$ is null homotopic. 
   \end{itemize} 
\end{lem} 

\begin{proof} 
First, the left map in the homotopy cofibration 
\(S^{k-1}\longrightarrow S^{2m-1}\times S^{k-1}\longrightarrow S^{2m-1}\amsrtimes S^{k-1}\)  
has a left inverse, so the connecting map for the homotopy 
cofibration is null homotopic, implying that for any space $Z$ the induced map 
\([S^{2m-1}\amsrtimes S^{k-1},Z]\longrightarrow [S^{2m-1}\times S^{k-1},Z]\) 
is an injection. Therefore, the definition of $t'$ in~(\ref{tau'def}) implies that its homotopy class is determined by the homotopy class of $t$. Thus to show that $t'_{1}=t'\circ i$ is homotopic to $i$ it suffices to show that $t\circ i_{1}\simeq i_{1}$, where 
\(i_{1}:S^{2m-1}\longrightarrow S^{2m-1}\times S^{k-1}\) 
is the inclusion of the first factor. But by definition, \(S^{2m-1}\times S^{k-1} \stackrel{t}{\rightarrow} S^{2m-1}\times S^{k-1}\) is given by $t(a,x)=(\tau(x)\cdot a,x)$, implying that $t(a,\ast)=(a,\ast)$, and therefore $t\circ i_{1}=i_{1}$, proving part~(i). 

Next, consider the composite 
\begin{equation}\label{eq:j'def}
j':S^{2m+k-2} \xrightarrow{j^\dagger} S^{2m-1}\amsrtimes S^{k-1} \xrightarrow{t'} S^{2m-1}\amsrtimes S^{k-1} \xrightarrow{e} S^{2m-1}\vee S^{2m+k-2}.
\end{equation} 
By the Hilton-Milnor Theorem, for dimensional and connectivity reasons the homotopy class of~$j'$ is determined by its pinch maps to $S^{2m-1}$ and $S^{2m+k-2}$. That is, if \[p_1: S^{2m-1}\vee S^{2m+k-2} \rightarrow S^{2m-1} \qquad p_2:S^{2m-1}\vee S^{2m+k-2} \rightarrow S^{2m+k-2}\] are the pinch maps to the left and right wedge summands respectively, then $j'\simeq i_{1}\circ p_1\circ j'+i_{2}\circ p_2\circ j'$. Since $t'$ is a self-equivalence, it must induce an isomorphism in homology. In particular, $(t')_{\ast}$ induces an isomorphism on $H_{2m+k-2}$ and therefore so does  $(t'_{2})_{\ast}$, and therefore in turn so does $(j')_{\ast}$. This implies that $p_2\circ j'$ is a homotopy equivalence, and therefore homotopic to a map of degree $\pm 1$. Refining, since we work with \(\mathrm{SO}(n)\), the map \(t\) must preserve orientation, implying that $t_{\ast}$ is the identity on $H_{2m+k-2}$, which in turn implies that $(t')_{\ast}$, $(t'_{2})_{\ast}$ and $(j')_{\ast}$ all induce the identity on $H_{2m+k-2}$. Hence $p_{2}\circ j'$ is homotopic to the identity map on $S^{2m+k-2}$, implying that $i_{2}\circ p_2\circ j'\simeq i_{2}$. 

Let $\overline{\tau}=p_1\circ j'$, and note that post-composing (\ref{eq:j'def}) with \(p_1\) gives the asserted composite for \(\overline{\tau}\) since by Lemma \ref{halfsmashcoH} we have \(p_1\circ e\simeq\pi\). Then we have $j'\simeq i_{1}\circ\overline{\tau} + i_{2}$. By definition of $j'$ we have  $e^{-1}\circ j'\simeq t'_{2}$. On the other hand, as $2\leq k\leq 2m-2$, Lemma~\ref{einv}~(iii) gives $e^{-1}\simeq i\perp j^\dagger$. Therefore $e^{-1}\circ i_{1}\simeq i$ and $e^{-1}\circ i_{2}\simeq j^\dagger$, giving $e^{-1}\circ j'\simeq e^{-1}\circ (i_{1}\circ\overline{\tau} + i_{2})\simeq i\circ\overline{\tau} + j^\dagger$. Hence $t'_{2}\simeq i\circ\overline{\tau} + j^\dagger$, proving part~(ii). 

Finally, if $\tau$ is the trivial twisting then $t$ and hence $t'$ are identity maps, in which case the definitions of $t'_{2}$ and $\overline{\tau}$ give $t'_{2}=j^\dagger$ and $\overline{\tau}=\pi\circ j^{\dagger}$. In the latter case we obtain a null homotopy for $\overline{\tau}$ since $\pi\circ j^{\dagger}$ is null homotopic by Lemma~\ref{ej}.  
\end{proof} 

The map \(\overline{\tau}\) in Lemma \ref{twisting12}(ii) has an additional property related to the $J$-homomorphism. In general, the join $A\ast B$ and the suspension $\Sigma(A\times B)$ are both quotient spaces of $A\times B\times I$. In the unreduced case of the join we identify $(a,b,1)$ to $(a,\ast,1)$ and $(a,b,0)$ to $(\ast,b,0)$, whereas for the suspension we identify $(a,b,1)$ to $(\ast,\ast,1)$ and $(a,b,0)$ to $(\ast,\ast,0)$. Thus the quotient map \(A\times B\times I\rightarrow\Sigma(A\times B)\) factors through a quotient map
\(A\ast B\rightarrow\Sigma(A\times B)\). 
In the reduced case, there is the additional relation in both cases that $(\ast,\ast,t)$ is identified with $(\ast,\ast,0)$. Moreover, there is a homotopy equivalence 
$\Sigma A\wedge B\simeq A\ast B$, thus giving a canonical choice of a map
\(\mathfrak{s}:\Sigma A\wedge B\rightarrow \Sigma(A\times B)\).
The $J$-homomorphism 
\[J\colon\pi_{k-1}(\mathrm{SO}(n))\rightarrow\pi_{n+k-1}(S^{n})\] 
is given by the homotopy class of the composite 
\[J(\tau):S^{n+k-1}\xrightarrow{\mathfrak{s}} \Sigma(S^{n-1}\times S^{k-1})\xrightarrow{\Sigma(1\times\tau)}\Sigma(S^{n-1}\times \mathrm{SO}(n)) \xrightarrow{\Sigma\theta}\Sigma S^{n-1}\simeq S^{n},\] 
where \(\theta\) denotes the usual action of \(\mathrm{SO}(n)\) on \(S^{n-1}\). The image of the $J$-homomorphism was calculated by Adams \cite{adamsIV} and Quillen  \cite{quillen_adamsconj}: 
\begin{equation} \label{imJ}
    im(J)\cong \begin{cases} 0\mbox{ if }k\equiv3,5,6,7\mbox{ (mod 8)} \\ \Z/2 \mbox{ if }k\equiv1,2\mbox{ (mod 8)} \\ \Z/d_s \mbox{ if }k=4s \end{cases}
\end{equation} 
where \(d_s\) is the demoninator of \(\frac{B_{2s}}{4s}\), \(B_{2s}\) being the \(2s\)-th Bernoulli number. 

\begin{prop} \label{prop:jhom}
    If \(2\leq k \leq 2m-2\) then \(\Sigma\overline{\tau}\simeq J(\tau)\), where \(J\) denotes the classical $J$-homomorphism.    
\end{prop}

\begin{proof}
    Take $n=2m$. We begin by relating $\mathfrak{s}$ to maps associated with the half-smash. Let 
    \[\mathfrak{q}:S^{2m-1}\times S^{k-1} \rightarrow S^{2m-1}\amsrtimes S^{k-1}\] 
    be the quotient map to the half-smash. Consider the composite 
    \[\mathfrak{j}:S^{2m+k-1}\xrightarrow{\mathfrak{s}}\Sigma(S^{2m-1}\times S^{k-1})\xrightarrow{\Sigma\mathfrak{q}} \Sigma(S^{2m-1}\amsrtimes S^{k-1}).\] 
    Since $\Sigma(S^{2m-1}\amsrtimes S^{k-1})\simeq S^{2m}\vee S^{2m+k-1}$ and $k\leq 2m-2$, the map $\mathfrak{j}$ is in the stable range and is a suspension, $\mathfrak{j}\simeq\Sigma\mathfrak{j}'$. 
    As $\mathfrak{j}$ is a suspension, regarding $\Sigma(S^{2m-1}\amsrtimes S^{k-1})$ as $S^{2m}\vee S^{2m+k-1}$, the Hilton-Milnor Theorem implies that the homotopy class of $\mathfrak{j}$ is determined by its composition with the pinch maps to $S^{2m-1}$ and $S^{2m+k-1}$. The pinch map to $S^{2m-1}$ factors as the composite 
    \[S^{2m+k-1}\xrightarrow{\mathfrak{s}} \Sigma(S^{2m-1}\times S^{2k-1}) \xrightarrow{\Sigma\pi_{1}} \Sigma S^{2m},\] 
    where $\pi_{1}$ is the projection. This is null homotopic by definition of $\mathfrak{s}$. The pinch map to $S^{2m+k-1}$ is homotopic to the identity map since it induces the identity map in homology. As the same is true of the map $j^{\dagger}$, we have $\mathfrak{j}\simeq \Sigma j^{\dagger}$. Hence $\Sigma j^{\dagger}\simeq\mathfrak{q}\circ\mathfrak{s}$.
    
    We now connect $\Sigma\overline{\tau}$ and $J(\tau)$. The self-map \(t:S^{2m-1}\times S^{k-1}\rightarrow S^{2m-1}\times S^{k-1}\) associated to $\tau$ is defined via the action of \(\mathrm{SO}(2m)\) on \(S^{2m-1}\) as well, namely \(t:(a,x)\mapsto(\tau(x)\cdot a,x)\), as in Diagram~(\ref{tau'def}). Thus there is a commutative square
    \begin{equation} \label{dgm:t vs tau-with-theta}
        \begin{tikzcd}[row sep=3em, column sep = 3em]
        S^{2m-1}\times S^{k-1} \arrow[d, "1 \times \tau"] \arrow[r, "t"] & S^{2m-1}\times S^{k-1} \arrow[d, "\pi_1"] \\
        S^{2m-1}\times \mathrm{SO}(2m)\arrow[r, "\theta"] & S^{2m-1}.
        \end{tikzcd}
    \end{equation}
    Consider the following diagram
    \begin{equation} \label{dgm:tau bar and J hom}
        \begin{tikzcd}[row sep=3em, column sep = 3em]
        S^{2m+k-1} \arrow[r, "\mathfrak{s}"] \arrow[dr, "\Sigma j^\dagger", swap]& \Sigma(S^{2m-1}\times S^{k-1}) \arrow[d, "\Sigma\mathfrak{q}"] \arrow[r, "\Sigma t"] & \Sigma(S^{2m-1}\times S^{k-1}) \arrow[d, "\Sigma\mathfrak{q}"] \arrow[r, "\Sigma\pi_1"] & S^{2m} \arrow[d, equal] \\
        & \Sigma(S^{2m-1}\amsrtimes S^{k-1}) \arrow[r, "\Sigma t'"] & \Sigma(S^{2m-1}\amsrtimes S^{k-1}) \arrow[r, "\Sigma\pi"] & S^{2m}.
        \end{tikzcd}
    \end{equation}
    The left-hand triangle homotopy commutes since $\mathfrak{s}\simeq\sigma\circ\Sigma j^{\dagger}$. The middle square homotopy commutes by  the top right square in (\ref{tau'def}). The right square homotopy commutes since $\pi$ and $\pi_{1}$ are both projections. Diagram (\ref{dgm:t vs tau-with-theta}) implies that \(\Sigma\pi_1\circ\Sigma t\simeq \Sigma\theta\circ\Sigma(1\times\tau)\), so the top direction around (\ref{dgm:tau bar and J hom}) is homotopic to \(J(\tau)\), whereas the bottom direction gives \(\Sigma\overline{\tau}\). Thus  $\Sigma\overline{\tau}\simeq J(\tau)$ as asserted.
\end{proof}

\begin{cor}\label{cor:tau-bar_non-trivial}
    Let \(2\leq k \leq 2m-2\) and suppose that \(k\equiv 1\text{ or }2\text{ (mod 8)}\). If \(\tau\in\pi_{k-1}(\mathrm{SO}(2m))\) is non-trivial, then \(\overline{\tau}\) is non-trivial.
\end{cor}

\begin{proof}
    For such \(k\) Bott periodicity gives \(\pi_{k-1}(\mathrm{SO}(2m))\cong\Z/2\) and~(\ref{imJ}) gives \(im(J)\cong\Z/2\), implying that the \(J\)-homomorphism is an isomorphism onto its image. Thus if \(\tau\) is non-trivial then $J(\tau)$ is non-trivial, so Proposition \ref{prop:jhom} implies that \(\Sigma\overline{\tau}\) is non-trivial. Therefore \(\overline{\tau}\) must be non-trivial.
\end{proof}

Returning to our study of gyrations, by Lemma~\ref{gyrationcofib} there is a homotopy cofibration \[S^{2m+k-2} \xrightarrow{\phi_{\tau}} S^{m}\amsrtimes S^{k-1} \rightarrow \mathcal{G}^{k}_{\tau}(\mathbb{F}P^{2})\] for each integer $k\geq 2$, where $\phi_{\tau}=(f\amsrtimes 1)\circ t'\circ j^\dagger$. By definition $t'\circ j^\dagger=t'_{2}$, so if $2\leq k\leq 2m-2$  Lemma~\ref{twisting12} gives a homotopy $t'\circ j^\dagger\simeq i\circ\overline{\tau} + j^\dagger$, or in the case of the trivial twisting, $t'_{2}=j^\dagger$. This proves the following. 

\begin{lem}  \label{prethm:twistattach} 
   If $2\leq k\leq 2m-2$ then the attaching map \(\phi_\tau\) for the top-cell of $\mathcal{G}^{k}_{\tau}(\mathbb{F}P^{2})$ is given by the composite 
   \[S^{2m+k-2} \xrightarrow{i\circ\overline{\tau} + j^\dagger} S^{2m-1}\amsrtimes S^{k-1} \xrightarrow{f\amsrtimes 1} S^{m}\amsrtimes S^{k-1}.\]  If $\tau$ is the trivial twisting then the attaching map for the top-cell of \(\mathcal{G}^{k}_{0}(\mathbb{F}P^{2})\) is \((f\amsrtimes 1)\circ j^\dagger\).~$\qqed$ 
\end{lem} 

We conclude Part I of this paper by bringing our description together, giving to a formulation of the attaching map for the top-cell in $\mathcal{G}^{k}_{\tau}(\mathbb{F}P^{2})$ that is easier to compute with.

\begin{thm} \label{thm:twistattach} 
    Let $2\leq k\leq 2m-2$, \(\tau:S^{k-1}\rightarrow\mathrm{SO}(2m)\), and let \(\phi_\tau:S^{2m+k-2} \rightarrow S^{m}\amsrtimes S^{k-1}\) denote the attaching map for the top-cell of the gyration $\mathcal{G}^{k}_{\tau}(\mathbb{F}P^{2})$. Then there is a homotopy \[\phi_\tau\simeq(i\circ f\circ\overline{\tau}) + ( j\circ\Sigma^{k-1} f) + [i,j]\] where \(\Sigma\overline{\tau}\) is in the image of the \(J\)-homomorphism. Explicitly, the summands are: 
    \begin{gather*} 
        S^{2m+k-2} \xrightarrow{\overline{\tau}} S^{2m-1} \xrightarrow{f} S^{m} \xrightarrow{i} S^{m}\amsrtimes S^{k-1} \\   
        S^{2m+k-2} \xrightarrow{\Sigma^{k-1} f} S^{m+k-1} \xrightarrow{j} S^{m}\amsrtimes S^{k-1} \\  
        S^{2m+k-2} \xrightarrow{[i,j]} S^{m}\amsrtimes S^{k-1}. 
    \end{gather*}     
    In particular, if $\tau$ is the trivial twisting then \(\phi_0\simeq(j\circ\Sigma^{k-1} f) + [i,j]\). 
\end{thm} 

\begin{proof} 
Noting that $\overline{\mathbb{F}P^{2}}\simeq S^{m}$, by Lemma~\ref{prethm:twistattach} the attaching map for the top-cell of $\mathcal{G}^{k}_{\tau}(\mathbb{F}P^{2})$ is given by the composite 
\[S^{2m+k-2} \xrightarrow{i\circ\overline{\tau} + j^\dagger} S^{2m-1}\amsrtimes S^{k-1} \xrightarrow{f\amsrtimes 1} S^{m}\amsrtimes S^{k-1}.\] 
In general, the sum of two maps in $[\Sigma X,Y]$ distributes on the left when composed with a map \(Y \rightarrow Z\). In our case, this gives \[(f\amsrtimes 1)\circ(i\circ\overline{\tau} + j^\dagger)\simeq((f\amsrtimes 1)\circ i\circ\overline{\tau})  + ((f\amsrtimes 1)\circ j^\dagger).\] The naturality of $i$ implies that $(f\amsrtimes 1)\circ i\simeq i\circ f$. Therefore $(f\amsrtimes 1)\circ i\circ\overline{\tau}\simeq i\circ f\circ\overline{\tau}$. By Corollary~\ref{fj}, $(f\amsrtimes 1)\circ j^\dagger\simeq (j\circ\Sigma^{k-1}  f) + [i,j]$. Thus  \[(f\amsrtimes 1)\circ(i\circ\overline{\tau} + j^\dagger)\simeq (i\circ f\circ\overline{\tau}) + (j\circ\Sigma^{k-1} f) + [i,j],\] as asserted. The case of the trivial attaching map follows from Remark~\ref{trivialattach} (or, in the 
argument above, setting~$\overline{\tau}$ to be the constant map). 
\end{proof} 

For what is to come in the next section, it is convenient to rephrase Theorem \ref{thm:twistattach} by letting \(\varphi_\tau\) be the composite 
\[
    \varphi_{\tau}:S^{2m+k-2}\xrightarrow{\phi_{t}} S^{m}\amsrtimes S^{k-1} \xrightarrow{e} S^{m}\vee S^{m+k-1}
\]
and thinking of the attaching map for the top-cell of a gyration as a map from a sphere to a wedge of spheres. We record this in the following Corollary, which is a direct consequence of Theorem \ref{thm:twistattach} and Lemma \ref{einv}.
    
\begin{cor} \label{cor:phipsi}
    Let $2\leq k\leq 2m-2$, \(\tau:S^{k-1}\rightarrow\mathrm{SO}(2m)\), and let \(\varphi_\tau:S^{2m+k-2} \rightarrow S^{m}\vee S^{m+k-1}\) denote the adjusted attaching map for the top-cell of the gyration $\mathcal{G}^{k}_{\tau}(\mathbb{F}P^{2})$. There is a homotopy \[\varphi_\tau\simeq(i_1\circ f\circ\overline{\tau}) + (i_2\circ\Sigma^{k-1} f) + [i_1,i_2].\] Rearranging, we have the identity 
    \begin{equation*}
        \varphi_\tau\simeq\varphi_0+\psi_\tau 
    \end{equation*} 
    where \(\psi_\tau:=i_1\circ f\circ\overline{\tau}\). \qed
\end{cor}

Corollary \ref{cor:phipsi} makes clear that the attaching maps \(\varphi_\tau\) differ from the map from the map for the trivial twisting (namely \(\varphi_0\)) by the addition of a summand \(\psi_\tau\) whose homotopy class depends on \(\tau\) and~\(f\).

\section*{\textbf{Part II: Gyrations of Projective Planes}}

From this point onwards, all manifolds are smooth and oriented unless otherwise stated. Given the structure established in Part I, we now focus on computations. The motivating questions are as in the Introduction.

\begin{question}[GSI]
    For a given \(k\geq2\) and \(n\)-manifold \(M\), do we have \(\mathcal{G}^{k}_{\tau}(M)\simeq\mathcal{G}^{k}_{\omega}(M)\) for all twistings \(\tau,\omega\in\pi_{k-1}(\mathrm{SO}(n))\)? 
\end{question}

\begin{question}[GSII]
    For a given \(k\geq2\) and \(n\)-manifold \(M\), how many different homotopy types can \(\mathcal{G}^{k}_{\tau}(M)\) have as the homotopy class of \(\tau\) is varied? 
\end{question}

\noindent The goal of Part II is to enumerate the homotopy types of \(\mathcal{G}^{k}_{\tau}(\mathbb{F}P^{2})\) for $\mathbb{F}$ being each of $\mathbb{C}$, $\mathbb{H}$ and $\mathbb{O}$.  

\section{Initial Observations and Examples}\label{sec:exgystab}

After making some initial observations this section considers gyration stability for $\mathcal{G}^{2}_{\tau}(\mathbb{C}P^{2})$ as an illustrative example. 
First observe that given two twistings \(\tau\) and \(\omega\) and an explicit homotopy between them, this cascades through the definitions (cf. Definition \ref{def:gy}) and one easily checks the following fact.

\begin{lem}\label{lem:easycheck}
    Let \(k\geq2\) be an integer, \(M\) an \(n\)-manifold and let \(\tau,\omega:S^{k-1}\rightarrow\mathrm{SO}(n)\). If \(\tau\simeq \omega\) then there is a homotopy equivalence \(\mathcal{G}^{k}_{\tau}(M)\simeq\mathcal{G}^{k}_{\omega}(M)\). \hfill \qed 
\end{lem}

Since the homotopy type of a gyration therefore depends (in part) on the homotopy class of the twisting, the homotopy groups of \(\mathrm{SO}(n)\) play an important role. In particular, if the relevant group is trivial then gyration stability follows.

\begin{prop} \label{prop:certaink}
    For a given \(n\)-manifold \(M\) and integer \(k\) such that \(2\leq k\leq n-1\), \(M\) is \(\mathcal{G}^k\)-stable if \(k\equiv3,5,6,7\text{ (mod 8)}\). 
\end{prop}

\begin{proof}
   For \(k\) in this range, by Bott periodicity we have \(\pi_{k-1}(\mathrm{SO}(n))\cong0\mbox{ if }k\equiv3,5,6,7\pmod{8}\). Hence Lemma~\ref{lem:easycheck} implies \(\mathcal{G}^{k}_{\tau}(M)\simeq\mathcal{G}^{k}_{\omega}(M)\) for all $\tau,\omega\in\pi_{k-1}(SO(n))$, and thus that $M$ is $\mathcal{G}^{k}$-stable.
\end{proof}

Proposition \ref{prop:certaink} is not an `if and only if' statement. To see why, consider gyrations of spheres. 

\begin{exa}[Spheres are \(\mathcal{G}^k\)-stable for all \(k\)] \label{ex:spheres}
For a sphere \(S^n\), we have \(\overline{S^n}\simeq\ast\), so by Lemma \ref{gyrationcofib} the \(k\)-gyration is given by the homotopy cofibration
\[
    S^{n+k-2} \xrightarrow{\phi_\tau} \ast \amsrtimes S^{k-1} \rightarrow \mathcal{G}^k_\tau(S^n)
\]
in which we observe that \(\ast\amsrtimes S^{k-1}\simeq\ast\). Therefore \(\mathcal{G}^k_\tau(S^n)\simeq S^{n+k-1}\) for all twistings \(\tau\).
\end{exa}

We now turn to the case of gyration stability for $\mathcal{G}^{2}_{\tau}(\mathbb{C}P^{2})$. Recall that for $\tau\in\pi_{1}(SO(4))$ there is a homotopy cofibration 
\(S^{4} \xrightarrow{\varphi_{\tau}} S^{2}\vee S^{3} \rightarrow \mathcal{G}^{2}_{\tau}(\mathbb{C}P^{2})\).

\begin{prop}\label{prop:cp2}
    There exists a homotopy equivalence \(\varepsilon:S^2\vee S^3 \rightarrow S^2\vee S^3\) such that \[\varepsilon \circ \varphi_0 \simeq \varphi_1\] where `0' denotes the trivial twisting and `1' denotes the twisting from the generator of \(\pi_1(\mathrm{SO}(4))\cong \Z/2\).  
\end{prop}

\begin{proof}
    Recalling our notation from the preamble to Part I, consider the map  
    \[\varepsilon\colon S^2\vee S^3 \rightarrow S^2\vee S^3\] 
    defined by \(\mathbb{1}+i_1\circ(*\perp \eta_2)\),where \(\mathbb{1}\) denotes the identity on the wedge. On homology this map induces an isomorphism since $(\eta_{2})_{\ast}=0$. Therefore $\varepsilon$ is a homotopy equivalence by Whitehead's Theorem. 

    By Corollary~\ref{cor:phipsi} with $k=2$ and $f=\eta_{2}$, the map \(\varphi_0:S^{4}\rightarrow S^{2}\vee S^{3}\) satisfies $\varphi_{0}\simeq i_{2}\circ\eta_{3}+[i_{1},i_{2}]$, where \(\eta_3\) denotes \(\Sigma \eta_2\). 
    Now consider the composite \(\varepsilon\circ\varphi_0:S^4\rightarrow S^2 \vee S^3\). In general, if \(\alpha,\beta:\Sigma A \rightarrow X\) and \(\gamma:X \rightarrow Y\) then $\gamma$ distributes on the left, i.e. $\gamma\circ(\alpha+\beta)\simeq \gamma\circ \alpha+ \gamma\circ \beta$. Therefore 
    \begin{equation}\label{eq:epsilonphi_cp2}
        \varepsilon\circ\varphi_0 \simeq \varepsilon\circ(i_2\circ\eta_3) + \varepsilon\circ[i_1,i_2].
    \end{equation} 
    Write \(\eta_2^2=\eta_2 \circ\eta_3\). Since $\epsilon=\mathbb{1}+i_1\circ(*\perp \eta_2)$, we see that \(\varepsilon\circ(i_2\circ\eta_3)\simeq (i_2\circ\eta_3)+(i_1\circ \eta_2^2)\), and noting that the Whitehead product in general satisfies \(\gamma\circ[\alpha,\beta]\simeq [\gamma\circ\alpha,\gamma\circ\beta]\) and \([\alpha,\beta+\beta']\simeq[\alpha,\beta]+[\alpha,\beta']\), we have 
    \[
        \varepsilon\circ[i_1,i_2]\simeq[\varepsilon\circ i_1,\varepsilon\circ i_2]\simeq[i_1,i_2+i_1\circ\eta_2]\simeq [i_{1},i_{2}]+[i_{1},i_{1}\circ\eta_{2}]\simeq [i_1,i_2]+i_1\circ[\iota,\eta_2]
    \] 
    where \(\iota\) denotes the identity map on \(S^2\). Thus (\ref{eq:epsilonphi_cp2}) becomes 
    \[
        \varepsilon\circ\varphi_0 \simeq (i_2\circ\eta_3)+(i_1\circ \eta_2^2) + [i_1,i_2]+i_1\circ[\iota,\eta_2]\simeq \varphi_0+i_1\circ(\eta_2^2 + [\iota,\eta_2])
    \]
    where the last homotopy comes from regrouping the summands. Thus if \(\psi_1\simeq i_1\circ(\eta_2^2+ [\iota,\eta_2])\) then $\varepsilon\circ\varphi_{0}\simeq\varphi_{0}+\psi_{1}$, and as $\varphi_{0}+\psi_{1}\simeq\varphi_{1}$ by Corollary~\ref{cor:phipsi}, we obtain $\varepsilon\circ\varphi_{0}\simeq\varphi_{1}$, as asserted.

    It remains to show that \(\psi_1\simeq i_1\circ(\eta_2^2+ [\iota,\eta_2])\).
    Since \(\pi_1(\mathrm{SO}(4))\cong\Z/2\) there is only one non-trivial twisting, for which we must have \(\overline{\tau}\simeq\eta_2\). Therefore, by definition of $\psi_{1}$ in Corollary~\ref{cor:phipsi}, \(\psi_1 \simeq i_1\circ\eta_2^2\). Moreover, \([\iota,\eta_2]\) is null homotopic since it represents a class in $\pi_{4}(S^{2})\cong\mathbb{Z}/2\mathbb{Z}$, which is stable, but Whitehead products suspend trivially. Therefore 
    \[
        i_1\circ(\eta_2^2 + [\iota,\eta_2])\simeq i_1\circ(\eta_2^2 + \ast)\simeq i_1\circ\eta_2^2 \simeq \psi_1,
    \] 
    as required.
\end{proof}

\begin{thm}\label{thm:g2cp2} 
    \(\C P^2\) is \(\mathcal{G}^2\)-stable, i.e. there is a homotopy equivalence \(\mathcal{G}_0^2(\C P^2)\simeq\mathcal{G}_1^2(\C P^2)\).
\end{thm}

\begin{proof} 
    From the homotopy $\varepsilon\circ\varphi_{0}\simeq\varphi_{1}$ in Proposition \ref{prop:cp2} we obtain a homotopy cofibration diagram
    \begin{equation*}
        \begin{tikzcd}[row sep=3em, column sep = 3em]
            S^4 \arrow[d, equal] \arrow[r, "\varphi_0"] & S^2\vee S^3 \arrow[r] \arrow[d, "\varepsilon"] & \mathcal{G}_0^2(\C P^2) \arrow[d, dashed] \\
            S^4 \arrow[r, "\varphi_1"] & S^2\vee S^3 \arrow[r] & \mathcal{G}_1^2(\C P^2)
        \end{tikzcd}
    \end{equation*}
    where the dashed arrow is an induced map of homotopy cofibres. This diagram induces a map of long exact sequences in homology. Therefore, as the left-hand vertical map is the identity and the middle map is a homotopy equivalence, the Five-Lemma implies that the dashed map induces an isomorphism in homology and so is a homotopy equivalence by Whitehead's Theorem since all spaces are simply-connected.
\end{proof}

Theorem~\ref{thm:g2cp2} shows that it is possible to have two twistings that are not in the same homotopy class but nevertheless produce gyrations that are homotopy equivalent. The reasoning used for \(\C P^2\) can be generalised to tackle gyration stability for \(\mathcal{G}_\tau^k(\H P^2)\) and \(\mathcal{G}_\tau^k(\O P^2)\) on a case-by-case basis, which will be done in Sections~\ref{sec:G(hp2)}, \ref{sec:G(op2)-k=<8} and~\ref{sec:G(op2)-9to14}. First however, we establish a general framework for the forthcoming arguments. 

\section{Compositions with Self-Equivalences of Wedges of Spheres}
\label{sec:selfequiv}

Fix a choice of field $\mathbb{F}\in\lbrace \H, \O \rbrace$ (i.e.~\(m=4\text{ or } 8\)) and index $k$ in the range \(2\leq k\leq 2m-2\). In this section, we wish to determine conditions for gyrations of \(\mathbb{F} P^2\) to be homotopy equivalent. We begin with a Lemma linking this question to self-equivalences of wedges of spheres.

\begin{lem}\label{lem:blakersmasseyargument}
    Let \(\tau\) and \(\omega\) be distinct homotopy classes in \(\pi_{k-1}(\mathrm{SO}(2m))\). There is a homotopy equivalence \(\mathcal{G}^{k}_{\tau}(\mathbb{F}P^{2})\simeq\mathcal{G}^{k}_{\omega}(\mathbb{F}P^{2})\) if and only if there exists a self-equivalence \(\varepsilon:S^{m}\vee S^{m+k-1}\xrightarrow{\simeq}S^{m}\vee S^{m+k-1}\) such that \(\varepsilon\circ\varphi_\tau\simeq\pm\varphi_\omega\).
\end{lem}

\begin{proof}
    Suppose first that there exists a self-equivalence \(\varepsilon\) such that \(\varepsilon\circ\varphi_\tau\simeq\pm\varphi_\omega\). From this homotopy we obtain a homotopy cofibration diagram
    \begin{equation*}
        \begin{tikzcd}[row sep=3em, column sep = 3em]
            S^{2m+k-2} \arrow[d, "\pm"] \arrow[r, "\varphi_\tau"] & S^{m}\vee S^{m+k-1}  \arrow[r] \arrow[d, "\varepsilon"] & \mathcal{G}^{k}_{\tau}(\mathbb{F}P^{2}) \arrow[d, dashed] \\
            S^{2m+k-2} \arrow[r, "\varphi_\omega"] & S^{m}\vee S^{m+k-1} \arrow[r] & \mathcal{G}^{k}_{\omega}(\mathbb{F}P^{2}) 
        \end{tikzcd}
    \end{equation*} 
    where the dashed arrow is an induced map of homotopy cofibres. This diagram induces a map of long exact sequences in homology. Therefore, as the left-hand and middle vertical maps are homotopy equivalences, the Five-Lemma implies that the dashed map induces an isomorphism in homology and so is a homotopy equivalence by Whitehead's Theorem since all spaces are simply-connected. 

    Conversely, assume that there is a homotopy equivalence 
    \[
        \epsilon:\mathcal{G}^{k}_{\tau}(\mathbb{F}P^{2}) \rightarrow \mathcal{G}^{k}_{\omega}(\mathbb{F}P^{2}).
    \] 
    Both spaces are simply-connected $CW$-complexes of dimension $2m+k-1$ whose $(2m+k-2)$-skeleton is homotopy equivalent to $S^{m}\amsrtimes S^{k-1}$. The restriction of $\epsilon$ to $(2m+k-2)$-skeletons gives a map 
    \[
        \epsilon^\prime:S^{m}\amsrtimes S^{k-1} \rightarrow S^{m}\amsrtimes S^{k-1}.
    \] 
    Since \(\epsilon\) induces an isomorphism on homology so does $\epsilon^\prime$, and $S^{m}\amsrtimes S^{k-1}$ is simply-connected, so $\epsilon^\prime$ is therefore a homotopy equivalence by Whitehead's Theorem. We place \(\epsilon\) in the context of self-equivalences of wedges of spheres by using the homotopy equivalence \(e:S^{m}\amsrtimes S^{k} \rightarrow S^{m}\vee S^{m+k-1}\) and defining a new equivalence
    \[
        \varepsilon=e\circ\epsilon^\prime\circ e^{-1}:S^{m}\vee S^{m+k-1}\rightarrow S^{m}\vee S^{m+k-1}.
    \] 
    Let \(q_\tau:S^{m}\vee S^{m+k-1}\rightarrow\mathcal{G}^{k}_{\tau}(\mathbb{F}P^{2})\) and \(q_\omega:S^{m}\vee S^{m+k-1}\rightarrow\mathcal{G}^{k}_{\omega}(\mathbb{F}P^{2})\) denote the respective skeletal inclusions, and let \(F_\tau\) and \(F_\omega\) be their homotopy fibres. Then there is a homotopy fibration diagram
    \begin{equation}\label{dgm:blakersmassey1}
        \begin{tikzcd}[row sep=3em, column sep = 3em]
            F_\tau \arrow[d, dashed, "\alpha"] \arrow[r] & S^{m}\vee S^{m+k-1} \arrow[r, "q_\tau"] \arrow[d, "\varepsilon"] & \mathcal{G}^{k}_{\tau}(\mathbb{F}P^{2}) \arrow[d, "\epsilon"] \\
            F_\omega \arrow[r] & S^{m}\vee S^{m+k-1}  \arrow[r, "q_\omega"] & \mathcal{G}^{k}_{\omega}(\mathbb{F}P^{2})
        \end{tikzcd}
    \end{equation}
    where the dashed arrow \(\alpha\) is an induced map of homotopy fibres. The diagram induces a map of long exact sequences of homotopy groups. Therefore, as the middle and right vertical maps are homotopy equivalences, the Five-Lemma implies that the dashed map induces an isomorphism on homotopy groups, and so is a homotopy equivalence by Whitehead's Theorem.
    
    Since \(\varphi_\tau\) and \(q_\tau\) are consecutive maps in a homotopy cofibration, the composite \(q_\tau\circ\varphi_\tau\) is null homotopic, so there is a lift \(\beta\) that makes the following diagram homotopy commutative  
    \begin{equation}\label{dgm:blakersmassey2}
        \begin{tikzcd}[row sep=3em, column sep = 3em]
            & S^{2m+k-2} \arrow[d, "\varphi_\tau"] \arrow[dl, swap, "\beta"] & \\
            F_\tau \arrow[r] & S^{m}\vee S^{m+k-1} \arrow[r, "q_\tau"] & \mathcal{G}^{k}_{\tau}(\mathbb{F}P^{2}).
        \end{tikzcd}
    \end{equation}
    Since \(\mathcal{G}^{k}_{\tau}(\mathbb{F}P^{2})\) is \((m-1)\)-connected and \(S^{2m+k-2}\) is \((2m+k-3)\)-connected, the lift \(\beta\) is a \((3m+k-3)\)-equivalence by the Blakers-Massey Theorem (cf. \cite{ark}*{Theorem 5.6.4}). Arguing identically with \(\omega\) in place of \(\tau\), there is also a \((3m+k-3)\)-equivalence \(\beta':S^{2m+k-2}\rightarrow F_\omega\).  
    Thus, for dimension reasons, $\alpha\circ\beta$ lifts through $\beta'$ to give a homotopy commutative diagram 
    \begin{equation}\label{dgm:blakersmassey3}
        \begin{tikzcd}[row sep=3em, column sep = 3em]
            S^{2m+k-2} \arrow[r, "\beta"] \arrow[d, "\gamma"]  & F_{\tau} \arrow[d, "\alpha"] \\
            S^{2m+k-2} \arrow[r, "\beta'"] &  F_\tau 
        \end{tikzcd}
    \end{equation}
    for some map $\gamma$. As $\alpha$ is a homotopy equivalence and both $\beta$ and $\beta'$ are \((3m+k-3)\)-equivalences, the homotopy commutativity of~(\ref{dgm:blakersmassey3}) implies that $\gamma$ is also a \((3m+k-3)\)-equivalence. In particular, $\gamma$ induces an isomorphism on $\pi_{2m+k-2}$, and therefore must be a homotopy equivalence. Hence \mbox{\(\gamma\simeq\pm1\)}. Juxtaposing~(\ref{dgm:blakersmassey3}) and (\ref{dgm:blakersmassey1}), and using the factorizations of $\varphi_{\tau}$ and $\varphi_{\omega}$ in (\ref{dgm:blakersmassey2}), we obtain a homotopy cofibration diagram
    \begin{equation}
        \begin{tikzcd}[row sep=3em, column sep = 3em]
            S^{2m+k-2} \arrow[d, "\pm1"] \arrow[r, "\varphi_\tau"] & S^{m}\vee S^{m+k-1} \arrow[r, "q_\tau"] \arrow[d, "\varepsilon"] & \mathcal{G}^{k}_{\tau}(\mathbb{F}P^{2}) \arrow[d, "\epsilon"] \\
            S^{2m+k-2} \arrow[r, "\varphi_\omega"] & S^{m}\vee S^{m+k-1}  \arrow[r, "q_\omega"] & \mathcal{G}^{k}_{\omega}(\mathbb{F}P^{2}).
        \end{tikzcd}
    \end{equation}
    The homotopy commutativity of the left-hand square gives that \(\varepsilon\circ\varphi_\tau\simeq\pm\varphi_\omega\).
\end{proof}

The next step is to identify candidates for a self-homotopy equivalence of $S^{m}\vee S^{m+k-1}$. Recall that we denote the inclusion of wedge summands by \(i_1:S^m\rightarrow S^{m}\vee S^{m+k-1}\) and \(i_2:S^{m+k-1}\rightarrow S^{m}\vee S^{m+k-1}\) and that \(\iota_m\) denotes the identity map on the sphere \(S^m\).

\begin{lem} \label{lem:prewedgeselfequiv} 
    Let \(\tau\) be a homotopy classes in \(\pi_{k-1}(\mathrm{SO}(2m))\). If \(\varepsilon:S^{m}\vee S^{m+k-1} \rightarrow S^{m}\vee S^{m+k-1}\) is a homotopy equivalence then: 
    \begin{enumerate}
        \item[(i)] the restriction of $\varepsilon$ to $S^{m}$ is homotopic to $(-1)^{i}\cdot i_{1}$ for some integer $i$;
        \item[(ii)] the restriction of $\varepsilon$ to $S^{m+k-1}$ is homotopic to $(i_{1}\circ\lambda)+(-1)^{j}\cdot i_{2}$ for some \(\lambda\in\pi_{m+k-1}(S^m)\) and some integer $j$; 
        \item[(iii)] for \(i,j\) and \(\lambda\) as in (i) and (ii) there is a homotopy 
        \[\varepsilon\circ\varphi_\tau\simeq i_{1}\circ\bigg((-1)^{i}\cdot f\circ\overline{\tau}+\lambda\circ\Sigma^{k-1} f+(-1)^{i}\cdot[\iota_m,\lambda]\bigg)+(-1)^{j}\cdot (i_{2}\circ\Sigma^{k-1}f)+(-1)^{i+j}\cdot [i_{1},i_{2}].\]
    \end{enumerate}
\end{lem} 

\begin{proof}
The map $\varepsilon$ is a map out of a wedge so it is determined up to homotopy by its restriction to each summand, i.e. \(\varepsilon\circ i_1\) and \(\varepsilon\circ i_2\). Note that as $\varepsilon$ is a homotopy equivalence it induces an isomorphism 
in homology, so $\varepsilon\circ i_1$ and $\varepsilon\circ i_2$ induce isomorphisms on $H_{m}$ and $H_{m+k-1}$ respectively. 

For (i), observe that $\varepsilon\circ i_1$ factors through the $m$-skeleton of $S^{m}\vee S^{m+k-1}$ -- this implies that it is homotopic to a composite  \[S^{m} \xlongrightarrow{d} S^{m} \xlongrightarrow{i_{1}} S^{m}\vee S^{m+k-1}\] for some map of degree $d$. Since $\varepsilon\circ i_1$ induces an isomorphism on $H_{m}$ we must have $d=\pm 1$. Thus there is a homotopy $\varepsilon\circ i_1\simeq\pm i_{1}$. 

For (ii), since $m\geq 2$ the Hilton-Milnor Theorem implies that $\varepsilon\circ i_2$ is determined by composing with the pinch maps to $S^{m}$ and $S^{m+k-1}$. The composition to $S^{m}$ gives a map \(\lambda:S^{m+k-1} \rightarrow S^{m}\) and the composition to $S^{m+k-1}$ again induces an isomorphism in homology, so is homotopic to a map of degree~$\pm 1$. Thus  $\varepsilon\circ i_2\simeq (i_{1}\circ\lambda)\pm i_{2}$. 

Summarising, for such a self-equivalence \(\varepsilon\) we may write 
\begin{equation} \label{eq:general epsilon}
\varepsilon\circ i_1\simeq(-1)^i\cdot i_1\text{\; and \;} \varepsilon\circ i_2\simeq (i_1\circ\lambda )+ (-1)^j\cdot i_2
\end{equation} for some integers \(i,j\in\lbrace0,1\rbrace\).

For (iii), by Corollary~\ref{cor:phipsi}, \(\varphi_\tau\simeq(i_1\circ f\circ\overline{\tau}) + (i_2\circ\Sigma^{k-1} f) + [i_1,i_2]\). By left-distributivity we obtain 
\begin{equation} \label{vareps1} 
    \varepsilon\circ\varphi_\tau\simeq(\varepsilon\circ i_1\circ f\circ\overline{\tau}) + (\varepsilon\circ i_2\circ\Sigma^{k-1} f) + (\varepsilon\circ [i_1,i_2]).
\end{equation}   
First, consider $\varepsilon\circ i_1\circ f\circ\overline{\tau}$. By (\ref{eq:general epsilon}) we have $\varepsilon\circ i_1\simeq (-1)^i\cdot i_1$, so \[\varepsilon\circ i_1\circ f\circ\overline{\tau}\simeq (-1)^i\cdot i_1\circ f\circ\overline{\tau}.\] 
Second, consider the composite $\varepsilon\circ i_2\circ\Sigma^{k-1} f$. In general, if \(\alpha,\beta:\Sigma A \rightarrow X\) and \(\Sigma g:\Sigma B \rightarrow \Sigma A\) then $\Sigma g$ distributes on the right: $(\alpha+\beta)\circ\Sigma g\simeq \alpha\circ\Sigma g + \beta\circ\Sigma g$. Since we demand $k\geq 2$, the map $\Sigma^{k-1} f$ is a suspension, and by (\ref{eq:general epsilon}) we obtain 
\[
    \varepsilon\circ i_2\circ\Sigma^{k-1} f\simeq ((i_{1}\circ\lambda)\pm i_{2})\circ\Sigma^{k-1} f \simeq  i_1\circ\lambda\circ\Sigma^{k-1} f + (-1)^j\cdot i_2\circ\Sigma^{k-1} f.
\]  
Third, consider $\varepsilon\circ [i_1,i_2]$. Recalling properties of the Whitehead product, we have in our case that 
\begin{align*}
\varepsilon\circ [i_1,i_2] & \simeq [\varepsilon\circ i_1, \varepsilon\circ i_2] \\
& \simeq [(-1)^i\cdot i_1,(i_1\circ\lambda )+ (-1)^j\cdot i_2] \\ 
& \simeq (-1)^i\cdot[i_1,i_1\circ\lambda]+(-1)^{i+j}\cdot[i_1,i_2] \\
& \simeq (-1)^i\cdot i_1\circ [\iota_m,\lambda]+ (-1)^{i+j}\cdot[i_1,i_2].
\end{align*}
Substituting the three parts into~(\ref{vareps1}) and then rearranging gives 
\[\begin{split} 
\varepsilon\circ\varphi_\tau & \simeq (-1)^i\cdot(i_1\circ f\circ\overline{\tau})+(i_1\circ\lambda\circ\Sigma^{k-1} f) + (-1)^j\cdot(i_2\circ\Sigma^{k-1} f) + (-1)^i\cdot i_1\circ [\iota_m,\lambda] + (-1)^{i+j}\cdot[i_1,i_2] \\ 
&  = i_{1}\circ\bigg((-1)^{i}\cdot f\circ\overline{\tau}+\lambda\circ\Sigma^{k-1} f+(-1)^{i}\cdot[\iota_m,\lambda]\bigg)+(-1)^{j}\cdot (i_{2}\circ\Sigma^{k-1}f)+(-1)^{i+j}\cdot [i_{1},i_{2}],
\end{split}\] 
proving part (iii).
\end{proof} 

\begin{lem} \label{lem:wedgeselfequiv} 
    Let \(\tau\) and \(\omega\) be distinct homotopy classes in \(\pi_{k-1}(\mathrm{SO}(2m))\).
    \begin{enumerate}
        \item[(i)]  There is a self-equivalence \(\varepsilon:S^{m}\vee S^{m+k-1} \rightarrow S^{m}\vee S^{m+k-1}\) such that \(\varepsilon\circ\varphi_\tau\simeq\varphi_\omega\) if and only if there exists a \(\lambda\in\pi_{m+k-1}(S^m)\) such that $f\circ\overline{\tau} +\lambda\circ\Sigma^{k-1} f +[\iota_m,\lambda]\simeq f\circ\overline{\omega}$; 
        \item[(ii)] There is a self-equivalence \(\varepsilon:S^{m}\vee S^{m+k-1} \rightarrow S^{m}\vee S^{m+k-1}\) such that \(\varepsilon\circ\varphi_\tau\simeq -\varphi_\omega\) if and only if there exists a \(\lambda\in\pi_{m+k-1}(S^m)\) such that $f\circ\overline{\tau} +\lambda\circ\Sigma^{k-1} f +[\iota_m,\lambda]\simeq -f\circ\overline{\omega}$.
    \end{enumerate} 
\end{lem} 

\begin{proof}  
Suppose there is a self-equivalence \(\varepsilon:S^{m}\vee S^{m+k-1} \rightarrow S^{m}\vee S^{m+k-1}\) such that \(\varepsilon\circ\varphi_\tau\simeq\pm\varphi_\omega\). 
By Lemma~\ref{lem:prewedgeselfequiv} there exists a \(\lambda\in\pi_{m+k-1}(S^m)\) that gives a homotopy
\[\varepsilon\circ\varphi_\tau\simeq i_{1}\circ\bigg((-1)^{i}\cdot f\circ\overline{\tau}+\lambda\circ\Sigma^{k-1} f+(-1)^{i}\cdot[\iota_m,\lambda]\bigg)+(-1)^{j}\cdot (i_{2}\circ\Sigma^{k-1}f)+(-1)^{i+j}\cdot [i_{1},i_{2}]\]
and by Corollary \ref{cor:phipsi},   
\[\varphi_\omega\simeq i_{1}\circ f\circ\overline{\omega}+i_{2}\circ\Sigma^{k-1} f+[i_{1},i_{2}].\] 
Since the terms $i_{1}\circ(\  \ )$, $i_{2}\circ (\ \ )$ and $[i_{1},i_{2}]$ are linearly independent, we may compare coefficients in the expressions for $\varepsilon\circ\varphi_{\tau}$ and $\pm\varphi_{\omega}$. 
\begin{itemize} 
\item[(i)] If $\varepsilon\circ\varphi_\tau\simeq\varphi_{\omega}$ then comparing the $i_{2}$-terms gives $(-1)^{j}=1$, so $j=0$, in which case comparing the $[i_{1},i_{2}]$-terms gives $(-1)^{i+j}=1$, implying that $i=0$, and then comparing the $i_{1}$-terms results in $f\circ\overline{\tau} +\lambda\circ\Sigma^{k-1} f +[\iota_m,\lambda]\simeq f\circ\overline{\omega}$; 
\item[(ii)] If $\varepsilon\circ\varphi_\tau\simeq -\varphi_{\omega}$ then comparing the $i_{2}$-terms gives $(-1)^{j}=-1$, so $j=1$, in which case comparing the $[i_{1},i_{2}]$-terms gives $(-1)^{i+j}=-1$, implying that $i=0$, and then comparing the $i_{1}$-terms results in $f\circ\overline{\tau}+\lambda\circ\Sigma^{k-1} f+[\iota_m,\lambda]\simeq -f\circ\overline{\omega}$. 
\end{itemize}

Conversely, suppose that there exists a \(\lambda\in\pi_{m+k-1}(S^m)\) such that $f\circ\overline{\tau} +\lambda\circ\Sigma^{k-1} f +[\iota_m,\lambda]\simeq f\circ\overline{\omega}$. Define \(\varepsilon:S^{m}\vee S^{m+k-1} \rightarrow S^{m}\vee S^{m+k-1}\) by letting the restrictions of $\varepsilon$ to $S^{m}$ and $S^{m+k-1}$ be $i_{1}$ and $(i_{1}\circ\lambda)+i_{2}$ respectively. Then $i=j=0$ in Lemma~\ref{lem:prewedgeselfequiv}~(i) and~(ii), implying that the same is true in the expression for $\varepsilon\circ\varphi_{\tau}$. Thus by Lemma \ref{lem:prewedgeselfequiv}~(iii) there is a homotopy 
\[\begin{split}
\varepsilon\circ\varphi_\tau & \simeq i_{1}\circ\bigg(f\circ\overline{\tau}+\lambda\circ\Sigma^{k-1} f+[\iota_m,\lambda]\bigg)+ (i_{2}\circ\Sigma^{k-1}f)+ [i_{1},i_{2}] \\ 
& \simeq i_{1}\circ f\circ\overline{\omega}+ (i_{2}\circ\Sigma^{k-1}f)+ [i_{1},i_{2}] \\ 
& \simeq\varphi_{\omega}. 
\end{split}\]
Similarly, if $f\circ\overline{\tau} +\lambda\circ\Sigma^{k-1} f +[\iota_m,\lambda]\simeq -f\circ\overline{\omega}$ then define the self-equivalence $\varepsilon$ by letting its restrictions to $S^{m}$ and $S^{m+k-1}$ be $i_{1}$ and $(i_{1}\circ\lambda) -i_{2}$ respectively, so that $i=0$ and $j=1$, and obtain $\varepsilon\circ\varphi_{\tau}\simeq-\varphi_{\omega}$. 
\end{proof} 

Lemmas \ref{lem:blakersmasseyargument} and \ref{lem:wedgeselfequiv} are used to prove the following key proposition.

\begin{prop}\label{prop:lambda_delta}
    Let \(\tau\) and \(\omega\) be distinct homotopy classes in \(\pi_{k-1}(\mathrm{SO}(2m))\). There is a homotopy equivalence \(\mathcal{G}^{k}_{\tau}(\mathbb{F}P^{2})\simeq\mathcal{G}^{k}_{\omega}(\mathbb{F}P^{2})\) if and only if there exists a homotopy class \(\lambda\in\pi_{m+k-1}(S^m)\) such that 
    \begin{equation}\label{eq:delta_def}
        f\circ\overline{\tau}+\lambda\circ\Sigma^{k-1}f+[\iota_m,\lambda]\simeq\pm f\circ\overline{\omega}.
    \end{equation}    
\end{prop}

\begin{proof}
    If there is a homotopy equivalence \(\mathcal{G}^{k}_{\tau}(\mathbb{F}P^{2})\simeq\mathcal{G}^{k}_{\omega}(\mathbb{F}P^{2})\), then by Lemma \ref{lem:blakersmasseyargument} there exists a self-equivalence \(\varepsilon:S^{m}\vee S^{m+k-1} \rightarrow S^{m}\vee S^{m+k-1}\) such that \(\varepsilon\circ\varphi_\tau\simeq\pm\varphi_\omega\). Lemma~\ref{lem:wedgeselfequiv} then implies that there is a homotopy $f\circ\overline{\tau}+\lambda\circ\Sigma^{k-1}f+[\iota_m,\lambda]\simeq\pm f\circ\overline{\omega}$.
    
   Conversely, suppose there exists a \(\lambda\in\pi_{m+k-1}(S^m)\) such that (\ref{eq:delta_def}) holds. Then Lemma \ref{lem:wedgeselfequiv} implies that there is a self-equivalence \(\varepsilon\colon\namedright{S^{m}\vee S^{m+k-1}}{}{S^{m}\vee  S^{m+k-1}}\) such that 
   \(\varepsilon\circ\varphi_\tau \simeq \pm\varphi_\omega\). Lemma \ref{lem:blakersmasseyargument} then implies that \(\mathcal{G}^{k}_{\tau}(\mathbb{F}P^{2})\simeq\mathcal{G}^{k}_{\omega}(\mathbb{F}P^{2})\). 
\end{proof}

Proposition~\ref{prop:lambda_delta} reduces the problem of detecting gyration stability of projective planes to the problem of finding a $\lambda\in\pi_{m+k-1}(S^{m})$ such that  
$f\circ\overline{\tau}+\lambda\circ\Sigma^{k-1}f+[\iota_m,\lambda]\simeq\pm f\circ\overline{\omega}$. We formulate this as a computational strategy for the coming sections. 

\begin{strat}\label{strategy}
    Fix \(m\) and \(k\). Given twistings \(\tau,\omega\in\pi_{k-1}(\mathrm{SO}(2m))\), by~(\ref{taubardef}) there are associated maps \(\overline{\tau},\overline{\omega}\in\pi_{2m+k-2}(S^{2m-1})\).
    \begin{enumerate}
        \item[Step\! 1:] Let \(x_1,\dots, x_r\in\pi_{2m+k-2}(S^{2m-1})\) be a generating set. As \(f=\eta_2,\nu_4\) or \(\sigma_8\) (for \(m=2,4\) or \(8\), respectively) the morphism \(f\circ-:\pi_{2m+k-2}(S^{2m-1})\rightarrow\pi_{2m+k-2}(S^m)\) is injective,  implying that \(f\circ x_1,\ldots,f\circ x_{r}\) are distinct elements of \(\pi_{2m+k-2}(S^m)\). Thus we may write 
        \[f\circ\overline{\tau}\simeq a_1\cdot(f\circ  x_1) + a_2\cdot(f\circ x_2) + \dots + a_r\cdot(f\circ x_r)\] 
        and 
        \[f\circ\overline{\omega}\simeq b_1\cdot(f\circ x_1) + b_2\cdot(f\circ x_2) + \dots + b_r\cdot(f\circ x_r)\] 
        where \(a_i\) and \(b_i\) are integers modulo the order of \(f\circ x_i\).  
        \vspace{5pt}
        \item[Step\! 2:] Find \(\lambda_1,\lambda_2,\dots\lambda_r\in\pi_{m+k-1}(S^m)\) such that
        \(\lambda_i\circ\Sigma^{k-1}f+[\iota_m,\lambda_i]\simeq f\circ x_i\) for all \(1\leq i\leq r\). 
        \vspace{5pt}
        \item[Step\! 3:] Take \(\lambda=(b_1-a_1)\cdot\lambda_1 + (b_2-a_2)\cdot\lambda_2 + \dots + (b_r-a_r)\cdot\lambda_r\). By Steps 1 and 2, and using right-distributivity, this choice of \(\lambda\) gives 
        \[
           f\circ\overline{\tau}+\lambda\circ\Sigma^{k-1}f+[\iota_m,\lambda]\simeq b_1\cdot(f\circ x_1) + b_2\cdot(f\circ x_2) + \dots + b_r\cdot(f\circ x_r) \simeq f\circ\overline{\omega}
        \]
        and therefore satisfies Proposition \ref{prop:lambda_delta}. Thus \(\mathcal{G}^{k}_{\tau}(\mathbb{F}P^{2})\simeq\mathcal{G}^{k}_{\omega}(\mathbb{F}P^{2})\).
    \end{enumerate}
    In arguments proving gyration \textit{instability} we will show that Step 2 fails. This is done by proving that there is at least one generator \(x_i\) for which there is no appropriate \(\lambda_i\), or that for some $\lambda\in\pi_{m+k-1}(S^{m})$ the term \(\lambda\circ\Sigma^{k-1}f+[\iota_m,\lambda]\) fails to produce the necessary congruence relations between the coefficients \(a_i\) and~\(b_i\) from Step 1. In either case it follows that \(\mathcal{G}^{k}_{\tau}(\mathbb{F}P^{2})\not\simeq\mathcal{G}^{k}_{\omega}(\mathbb{F}P^{2})\).  
\end{strat}

\section{Computations for \(\mathcal{G}_\tau^k(\mathbb{H}P^2)\)}
\label{sec:G(hp2)}

Throughout Sections \ref{sec:G(hp2)}, \ref{sec:G(op2)-k=<8} and \ref{sec:G(op2)-9to14} we will repeatedly use results regarding compositions of elements in the homotopy groups of spheres. We shall follow Toda's notation, in particular that \(\eta_n=\Sigma^{n-2}\eta_2\) for \(n\geq 2\), \(\nu_n=\Sigma^{n-4}\nu\) for \(n\geq 4\) and that there exists a non-trivial homotopy class \(\nu'\in\pi_6(S^3)\).
For a cyclic group \(\Gamma\) we write \(\Gamma\langle x \rangle\) for \(\Gamma\) with the explicit choice of generator \(x\in\Gamma\).  
\medskip 
 
For \(\H P^2\) fix \(m=4\) and \(f=\nu_4\), and consider \(k\) in the range \(2\leq k \leq 6\). Proposition \ref{prop:certaink} implies \(\mathcal{G}^k\)-stability when \(k=3,5\text{ or }6\), so we are left to investigate gyrations when \(k=2\) and \(k=4\). 
\medskip 
 
\noindent 
\textbf{The $k=2$ case.} 
Since \(\pi_1(\mathrm{SO}(8))\cong\Z/2\), there are two gyrations \(\mathcal{G}_0^2(\H P^2)\) and \(\mathcal{G}_1^2(\H P^2)\) corresponding to the trivial and non-trivial twistings, respectively.  
Following Strategy \ref{strategy}, given two twistings \newline \(\tau,\omega\in\pi_1(\mathrm{SO}(8))\cong\Z/2\), the elements \(\overline{\tau}\), \(\overline{\omega}\) and \(\lambda\) lie in the following homotopy groups
\begin{equation}\label{pi8(S7)_pi5(S4)}
    \overline{\tau},\overline{\omega}\in\pi_{8}(S^7) \cong \Z/2\langle\eta_7\rangle
    \text{\; and \;}
    \lambda\in\pi_5(S^4) \cong \Z/2\langle\eta_4\rangle.
\end{equation}
For \(\pi_8(S^4)\), the following proposition records a synthesis of facts from \cite{toda}*{Proposition 5.8} and its proof, for ease of reference.

\begin{prop}[Toda] \label{prop:toda_pi8^4}
    For \(\pi_8(S^4)\) the  following hold:
    \begin{enumerate}
        \item[(i)] \(\pi_8(S^4)\cong\Z/2\langle \nu_4\circ\eta_7 \rangle \oplus \Z/2\langle \Sigma\nu'\circ\eta_7 \rangle\);
        \item[(ii)] the composite \(\eta_4\circ\nu_5\) is homotopic to \(\Sigma\nu'\circ\eta_7\);
        \item[(iii)] the kernel of the suspension map \(E:\pi_8(S^4)\rightarrow\pi_9(S^5)\) is generated by \(\Sigma\nu'\circ\eta_7\). \qed
    \end{enumerate}
\end{prop}

\noindent Thus, Step 1 of Strategy~\ref{strategy} amounts to the lemma below.

\begin{lem}\label{lem:G2(HP2)_tau-bar}
    Given \(\tau\in\pi_{1}(\mathrm{SO}(8))\), if \(\tau\) is non-trivial then \(\nu_{4}\circ\overline{\tau}\simeq\nu_{4}\circ\eta_7\).
\end{lem}

\begin{proof}
    If \(\tau\) is non-trivial, then Corollary \ref{cor:tau-bar_non-trivial} implies that \(\overline{\tau}\) is also non-trivial. Thus as $\overline{\tau}\in\pi_{8}(S^{7})$ by~(\ref{pi8(S7)_pi5(S4)}), we have  \(\overline{\tau}\simeq\eta_7\). The lemma then follows immediately.
\end{proof} 
 
As for Step 2 of Strategy \ref{strategy}, we first show the following.

\begin{lem} \label{lem:eta-nu}
    The composite \(\eta_4\circ\nu_5+[\iota_4,\eta_4]\) is null homotopic.
\end{lem}

\begin{proof}
    By \cite{hilton_whitehead}, the Whitehead product \([\iota_4,\eta_4]\) is non-trivial. As it represents a class in $\pi_{8}(S^{4})$ and suspends trivially, Proposition \ref{prop:toda_pi8^4}(iii) implies that is homotopic to \(\Sigma\nu'\circ\eta_7\). Also, Proposition \ref{prop:toda_pi8^4}(ii) says that \(\eta_4\circ\nu_5\simeq \Sigma\nu'\circ\eta_7\) while Proposition \ref{prop:toda_pi8^4}(i) says that \(\Sigma\nu'\circ\eta_7\) has order two, implying that \[\eta_4\circ\nu_5+[\iota_4,\eta_4]\simeq 2 \cdot( \Sigma\nu'\circ\eta_7)\simeq \ast.\qedhere\]  
\end{proof}

This enables us to show that Step 2 of Strategy \ref{strategy} fails, resulting in the following. 

\begin{thm} \label{thm:g2hp2}
    \(\H P^2\) is not \(\mathcal{G}^2\)-stable, i.e. \(\mathcal{G}_0^2(\H P^2)\not\simeq\mathcal{G}_1^2(\H P^2)\).
\end{thm}

\begin{proof}
    We will show that there is no \(\lambda\in\pi_{5}(S^{4})\cong\mathbb{Z}/2\langle\eta_{4}\rangle\) satisfying Proposition \ref{prop:lambda_delta}. Note that since the homotopy group \(\pi_8(S^4)\) is isomorphic to a direct sum of \(\Z/2\) summands there is no distinction between \(+1\) and~\(-1\), so we may operate in the `+' case of Proposition \ref{prop:lambda_delta} without loss of generality. 
    
    Proving \(\mathcal{G}^2\)-instability therefore reduces to checking that there exists no \(\lambda\) such that 
    \begin{equation}\label{g2hp2eqn}
        \nu_{4}\circ\eta_7+\lambda\circ\nu_5+[\iota_4,\lambda]\simeq\ast.
    \end{equation} 
    By~(\ref{pi8(S7)_pi5(S4)}), either $\lambda\simeq\ast$ or $\lambda\simeq\eta_{4}$. If $\lambda\simeq\ast$ then~(\ref{g2hp2eqn}) implies that $\nu_{4}\circ\eta_{7}$ is null homotopic, contradicting Proposition~\ref{prop:toda_pi8^4}~(i). If $\lambda\simeq\eta_{4}$ then Lemma \ref{lem:eta-nu} shows that \(\eta_4\circ\nu_5+[\iota_4,\eta_4]\) is null  homotopic, in which case~(\ref{g2hp2eqn}) again shows that $\nu_{4}\circ\eta_{7}$ is null homotopic, a contradiction.
\end{proof}

\begin{rem}
    Note that Theorem~\ref{thm:g2hp2} is subtly stronger that what was known before. Duan~\cite{duan} showed that \(\mathcal{G}_0^2(\H P^2)\) and \(\mathcal{G}_1^2(\H P^2)\) are not diffeomorphic, Theorem~\ref{thm:g2hp2} shows that they are also not homeomorphic or even homotopy equivalent.
\end{rem}

\noindent 
\textbf{The $k=4$ case:} 
There is a gyration $\mathcal{G}^{4}_{\tau}(\mathbb{H}P^{2})$ for each $\tau\in\pi_{3}(SO(8))\cong\mathbb{Z}$. It will be shown in Theorem~\ref{thm:g4hp2} that they are all homotopy equivalent. We start with two preparatory statements: an elementary lemma, followed by a classical result regarding Whitehead products and suspensions. 

\begin{lem} \label{lem:coprime_order}
    Let \(\Sigma\alpha\in\pi_i(S^n)\) and \(\Sigma\beta\in\pi_j(S^i)\) be two suspensions of finite order. If their orders are coprime, then the composition \(\Sigma\beta\circ\Sigma\alpha\) is null homotopic.
\end{lem}

\begin{proof}
    Let \(\Sigma\alpha\) have order \(a\) and \(\Sigma\beta\) have order \(b\). Since both maps are suspensions we have \[a\cdot(\Sigma\beta\circ\Sigma\alpha)\simeq\Sigma\beta\circ(a\cdot\Sigma\alpha)\simeq\Sigma\beta\circ\ast\simeq\ast\] and similarly that \[b\cdot(\Sigma\beta\circ\Sigma\alpha)\simeq(b\cdot\Sigma\beta)\circ\Sigma\alpha\simeq\ast\circ\Sigma\alpha\simeq\ast.\] By assumption, \(gcd(a,b)=1\), so B\'ezout's identity implies that there exist integers \(s\) and \(t\) such that \(sa+tb=1\). This together with the above null homotopies implies that \[\Sigma\beta\circ\Sigma\alpha\simeq(sa+tb)\cdot(\Sigma\beta\circ\Sigma\alpha)\simeq sa\cdot(\Sigma\beta\circ\Sigma\alpha)+tb\cdot(\Sigma\beta\circ\Sigma\alpha)\simeq\ast\]
    where right-distributivity comes from the fact that both maps are suspensions.
\end{proof}

\begin{lem}[\cite{whitehead-elements}*{Theorem X.8.18}]\label{lem:whitehead_prods_and_susps}
    Let \(\alpha\in\pi_{p+1}(X)\), \(\beta\in\pi_{q+1}(X), \gamma\in\pi_i(S^p)\text{ and }\delta\in\pi_j(S^q)\). Then 
    \[
    \pushQED{\qed} 
        [\alpha\circ\Sigma\gamma,\beta\circ\Sigma\delta] \simeq [\alpha,\beta]\circ\Sigma(\gamma\wedge\delta).
    \qedhere
    \popQED 
    \]
\end{lem} 

The next two propositions, also separated for later ease of reference, describe the relevant homotopy groups and generators. Note that some generators are not labelled as they will not be needed later. The statements and notation are taken from \cite{toda}, except that we use \(\widehat{x}\) to denote the 2-primary component of a homotopy class \(x\).

\begin{prop}[Toda]\label{prop:nu_n}
    For \(n>4\) there are group isomorphisms
    \[ 
        \pi_{n+3}(S^n) \cong \Z/24\langle\nu_n\rangle \cong\Z/8\langle\widehat{\nu_n}\rangle\oplus\Z/3\langle\alpha_1(n)\rangle, 
    \]
    where the second isomorphism comes from writing \(\nu_7\simeq\widehat{\nu_7}+\alpha_1(7)\). \qed
\end{prop}

\begin{prop}[Toda] \label{prop:G4HP2_htpygps}
    There are group isomorphisms:
    \begin{enumerate}
        \item[(i)] \(\pi_7(S^4) \cong\Z\langle\nu_4\rangle\oplus\Z/4\langle\Sigma\nu'\rangle\oplus\Z/3\);
        \item[(ii)] \(\pi_{10}(S^4) \cong \Z/8\langle\nu_4\circ\widehat{\nu_7}\rangle\oplus\Z/3\langle\nu_4\circ\alpha_1(7)\rangle\oplus\Z/3\). \qed
    \end{enumerate}
\end{prop}

We now move to Step 1 of Strategy \ref{strategy}.

\begin{lem}\label{lem:where is tau-bar? HP2}
    For any \(\tau\in\pi_{3}(\mathrm{SO}(8))\) there exist integers \(a_1\) and \(a_2\), modulo 8 and 3 respectively, such that \(\nu_4\circ\overline{\tau}\simeq a_1\cdot(\nu_4\circ\widehat{\nu_7})+a_2\cdot(\nu_4\circ\alpha_1(7))\).
\end{lem}

\begin{proof} 
   With $k=m=4$, by~(\ref{taubardef}) we have $\overline{\tau}\in\pi_{10}(S^{7})$. By Proposition \ref{prop:nu_n} we obtain
    \[
        \overline{\tau}\in\pi_{10}(S^7) \cong \Z/24\langle\nu_7\rangle\cong\Z/8\langle\widehat{\nu_7}\rangle\oplus\Z/3\langle\alpha_1(7)\rangle.
    \] 
    Therefore $\overline{\tau}=a_{1}\cdot\widehat{\nu}_{7}+a_{2}\cdot\alpha_{1}(7)$ for some integers $a_{1}$ and $a_{2}$ modulo $8$ and $3$ respectively. The statement of the lemma follows by left distributivity. 
\end{proof}

We also have the following homotopies, due to \cite{toda}*{Propositions 5.8 and 5.11}, which will be crucial for the computations for this case: 
\begin{equation}\label{eq:Todarel1} 
\Sigma^2\nu'\simeq 2\cdot\widehat{\nu_5}\text{\; and \;}\Sigma\nu'\circ\nu_7\simeq\ast.
\end{equation} 

\noindent The following lemma constitutes Step 2 of Strategy \ref{strategy}.

\begin{lem} \label{lem:lambda_pi7(S4)}
There exist \(\lambda_1,\lambda_2\in\pi_7(S^4)\) such that:
    \begin{enumerate}
        \item[(i)] \(\lambda_1\circ\nu_7+[\iota_4,\lambda_1]\simeq\nu_4\circ\widehat{\nu_7}\);
        \item[(ii)] \(\lambda_2\circ\nu_7+[\iota_4,\lambda_2]\simeq\nu_4\circ\alpha_1(7)\).
    \end{enumerate}
\end{lem}

\begin{proof}
    Before finding the asserted homotopy classes we consider an element \(\Lambda\in\pi_7(S^4)\) which, via Proposition \ref{prop:G4HP2_htpygps}(ii), can be written as
    \[
        \Lambda=x\cdot(\nu_4)+y\cdot(\Sigma\nu')
    \] 
    for integers \(x\) and \(y\), with \(y\) considered modulo 4. Both the composite $\Lambda\circ\nu_{7}$ and the Whitehead product \([\iota_4,\Lambda]\) lie in the homotopy group \(\pi_{10}(S^4)\); the first aim is to express these in terms of the generators given in Proposition \ref{prop:G4HP2_htpygps}(ii). 
    
    Take \(\Lambda\circ\nu_7\). Since $\nu_{7}$ is a suspension, right-distributivity gives 
    \(\Lambda\circ\nu_7\simeq x\cdot(\nu_4\circ\nu_7)+y\cdot(\Sigma\nu'\circ\nu_7)\).  
    By~(\ref{eq:Todarel1}),
    \(\Sigma\nu'\circ\nu_7\simeq\ast\). Recalling that \(\nu_7\simeq\widehat{\nu_7}+\alpha_1(7)\), by left-distributivity we therefore have 
    \begin{equation}\label{eq:lambda-nu7 - k=4}
    \Lambda\circ\nu_7\simeq x\cdot(\nu_4\circ\widehat{\nu_7})+x\cdot(\nu_4\circ\alpha_1(7)).
    \end{equation} 
    
    Next consider \([\iota_4,\Lambda]\). The Whitehead product is additive, so we obtain 
    \[[\iota_4,\Lambda]\simeq x\cdot[\iota_4,\nu_4]+y\cdot[\iota_4,\Sigma\nu'].\]
    Applying Lemma \ref{lem:whitehead_prods_and_susps} to \([\iota_4,\Sigma\nu']\) and using $2\cdot\widehat{\nu}_{7}\simeq\Sigma^{4}\nu'$ from~(\ref{eq:Todarel1}), there are homotopies
    \begin{equation}\label{eq:Todarel2} [\iota_4,\Sigma\nu']\simeq[\iota_4,\iota_4]\circ\Sigma^{4}\nu'\simeq2\cdot([\iota_4,\iota_4]\circ\widehat{\nu_7}).\end{equation} 
    By \cite{toda}*{(5.8)}, there is a homotopy \([\iota_4,\iota_4]\simeq2\cdot\nu_4-\Sigma\nu'\) so it follows from (\ref{eq:Todarel1}) and~(\ref{eq:Todarel2}) that
    \[
    [\iota_4,\Sigma\nu']\simeq4\cdot(\nu_4\circ\widehat{\nu_7}).
    \] 
    For \([\iota_4,\nu_4]\), by \cite{toda}*{(5.13)} or \cite{behrens_goodwillie}*{Proposition 3.6.1}, we have 
    \([\iota_4,\nu_4]\simeq\pm2\cdot(\nu_4\circ\nu_7)\).
    Putting this together gives  
    \begin{equation}\label{eq:lambda-whitehead - k=4}
    [\iota_4,\Lambda]\simeq(\pm2x+4y)\cdot(\nu_4\circ\nu_7).
    \end{equation}
    Combining (\ref{eq:lambda-nu7 - k=4}) and (\ref{eq:lambda-whitehead - k=4}) therefore gives the homotopy
    \[
        \Lambda\circ\nu_7+[\iota_4,\Lambda]\simeq(x\pm2x+4y)\cdot(\nu_4\circ\widehat{\nu_7})+x\cdot(\nu_4\circ\alpha_1(7)).
    \]    
    
    For part (i), to obtain \(\lambda_1\circ\nu_7+[\iota_4,\lambda_1]\simeq\nu_4\circ\widehat{\nu_7}\) we require \(x\pm2x+4y\equiv1\text{ (mod 8)}\) and \(x\equiv0\text{ (mod 3)}\). There are two cases depending on the sign $\pm$: in the `\(+\)' case take \(x=3\) and \(y=0\), giving \(\lambda_1=3\cdot\nu_4\), and in the `\(-\)' case take \(x=y=3\), giving \(\lambda_1=3\cdot\nu_4+3\cdot\Sigma\nu'\). 
    \smallskip 
    
    As for part (ii), to obtain \(\lambda_2\circ\nu_7+[\iota_4,\lambda_2]\simeq\nu_4\circ\alpha_1(7)\) we require that \(x\pm2x+4y\equiv0\text{ (mod 8)}\) and that \(x\equiv1\text{ (mod 3)}\). In the `\(+\)' case take \(x=4\) and \(y=1\), giving \(\lambda_2=4\cdot\nu_4+\Sigma\nu'\), and in the `\(-\)' case take \(x=4\) and \(y=3\), giving \(\lambda_2=4\cdot\nu_4+3\cdot\Sigma\nu'\). 
\end{proof}

Executing Step 3 of Strategy~\ref{strategy} provides the following result for \(\mathcal{G}^4\)-stability of \(\H P^2\).

\begin{thm} \label{thm:g4hp2}
    \(\H P^2\) is \(\mathcal{G}^4\)-stable, i.e.~\(\mathcal{G}_\tau^4(\H P^2)\simeq\mathcal{G}_\omega^4(\H P^2)\) for all twistings \(\tau,\omega\in\pi_3(\mathrm{SO}(8))\). 
\end{thm}

\begin{proof}
    By Lemma \ref{lem:where is tau-bar? HP2}, for any twistings \(\tau\) and \(\omega\) we have 
    \[\nu_4\circ\overline{\tau}\simeq a_1\cdot(\nu_4\circ\widehat{\nu_7})+a_2\cdot(\nu_4\circ\alpha_1(7))\text{\; and \;}\nu_4\circ\overline{\omega}\simeq b_1\cdot(\nu_4\circ\widehat{\nu_7})+b_2\cdot(\nu_4\circ\alpha_1(7))\] 
    for some integers \(a_1\) and \(b_1\) modulo 8, and \(a_2\) and \(b_2\) modulo 3. Take \(\lambda=(b_1-a_1)\cdot\lambda_1+(b_2-b_1)\cdot\lambda_2\), where $\lambda_{1}$ and $\lambda_{2}$ are as in Lemma \ref{lem:lambda_pi7(S4)}. Then Step~$3$ of Strategy~\ref{strategy} implies that $\mathcal{G}^{4}_{\tau}(\mathbb{H}P^{2})\simeq\mathcal{G}^{4}_{\omega}(\mathbb{H}P^{2})$. This holds for any choice of $\tau$ and $\omega$, so $\mathbb{H}P^{2}$ is \(\mathcal{G}^4\)-stable. 
\end{proof}

\section{Computations for \(\mathcal{G}_\tau^k(\mathbb{O}P^2)\) when \(k\leq8\)}
\label{sec:G(op2)-k=<8}

For \(\O P^2\) fix $m=8$ and \(f=\sigma_8\), and consider $k$ in the range $2\leq k\leq 14$. Proposition \ref{prop:certaink} implies $\mathcal{G}^{k}$-stability when $k=3,5,6,7,11,13$ or $14$, so we are left to investigate \(k=2,4,8,9,10\text{ and }12\). In this section we treat the first three of these cases. 

Here and in Section  \ref{sec:G(op2)-9to14}, following Toda we use the notation \(\sigma_n=\Sigma^{n-8}\sigma_8\) for \(n>8\) and note there is a homotopy class \(\sigma'\in\pi_{14}(S^7)\) that satisfies the identity \(\Sigma^2\sigma'\simeq2\cdot\widehat{\sigma_9}\).
\medskip 

\noindent 
\textbf{The $k=2$ case.} 
As \(\pi_1(\mathrm{SO}(16))\cong\Z/2\), there two gyrations, which (similar to previous sections) will be denoted by \(\mathcal{G}_0^2(\mathbb{O}P^2)\) and \(\mathcal{G}_1^2(\mathbb{O}P^2)\). Furthermore, this also implies that Step 1 of the general strategy reduces to the statement that for any twisting \(\tau\in\pi_1(\mathrm{SO}(16))\) either \(\sigma_8\circ\overline{\tau}\simeq\ast\) or \(\sigma_8\circ\overline{\tau}\simeq\sigma_8\circ\eta_{15}\). We will show that Step 2 fails, via the following lemma. 

\begin{lem} \label{lem:g2op2nonhomotopy} 
\(\eta_8\circ\sigma_9+[\iota_8,\eta_8]\not\simeq\sigma_8\circ\eta_{15}\). 
\end{lem}

\begin{proof} 
    First note that by \cite{toda}*{(5.15)}, the composite \(\sigma_8\circ\eta_{15}\) is not homotopic to a suspension. On the other hand, the composite $\eta_{8}\circ\sigma_{9}$ is a suspension, being homotopic to $\Sigma(\eta_{7}\circ\sigma_{8})$. 
    By \cite{toda}*{p.~63} the kernel of the suspension map \(E:\pi_{16}(S^8)\rightarrow\pi_{17}(S^9)\) is generated by \(\Sigma\sigma'\circ\eta_{15}\). As this kernel is generated by a suspension, the Whitehead product \([\iota_8,\eta_8]\), being an element of this kernel, is also a suspension. This implies that \(\eta_8\circ\sigma_9+[\iota_8,\eta_8]\) is a suspension since it is a sum of suspensions. Thus \(\eta_8\circ\sigma_9+[\iota_8,\eta_8]\not\simeq\sigma_8\circ\eta_{15}\) since the left side is a suspension while right side is not. 
\end{proof} 

\begin{thm} \label{thm:g2op2}
    \(\O P^2\) is not \(\mathcal{G}^2\)-stable, i.e., \(\mathcal{G}_0^2(\mathbb{O}P^2)\not\simeq\mathcal{G}_1^2(\mathbb{O}P^2)\).
\end{thm} 

\begin{proof} 
    We will show that there is no \(\lambda\in\pi_{9}(S^{8})\cong\mathbb{Z}/2\langle\eta_{8}\rangle\) satisfying Proposition \ref{prop:lambda_delta}. Proving \(\mathcal{G}^2\)-instability therefore reduces to checking that there exists no \(\lambda\) such that 
    \begin{equation}\label{g2op2eqn}
        \sigma_{8}\circ\eta_{15}+\lambda\circ\sigma_9+[\iota_8,\lambda]\simeq\ast.
    \end{equation} 
    If $\lambda\simeq\ast$ then~(\ref{g2op2eqn}) cannot hold since, by~\cite{toda}*{Theorem 7.1}, $\sigma_{8}\circ\eta_{15}\not\simeq\ast$. If $\lambda\simeq\eta_{8}$ then Lemma~\ref{lem:g2op2nonhomotopy} shows that~(\ref{g2op2eqn}) cannot hold. Therefore there is no $\lambda$ such that~(\ref{g2op2eqn}) holds, as required.  
\end{proof}

\noindent 
\textbf{The $k=4$ case}: There is a gyration $\mathcal{G}^{4}_{\tau}(\mathbb{O}P^{2})$ for each $\tau\in\pi_{3}(SO(16))\cong\mathbb{Z}$. We first give three results describing the relevant homotopy groups, generators and relations. 

\begin{prop}[Toda]\label{prop:sigma_n}
    For \(n>8\) there are group isomorphisms
    \[
        \pi_{n+7}(S^n) \cong \Z/240\langle\sigma_n\rangle \cong \Z/16\langle\widehat{\sigma}_n\rangle\oplus\Z/5\langle\widetilde{\alpha_1}(n)\rangle\oplus\Z/3\langle\alpha_2(n)\rangle
    \]
    where the second isomorphism comes from writing \(\sigma_{n}\simeq\widehat{\sigma}_{n}+\widetilde{\alpha}_1(n)+\alpha_2(n)\). The notation \(\widetilde{\alpha_1}(n)\) is used to distinguish the 5-torsion class from the 3-torsion class \(\alpha_1(n)\). \qed
\end{prop}

\begin{prop}[Toda]\label{prop:G4OP2_htpygps}
    There is a group isomorphism
    \[
        \pushQED{\qed} 
        \pi_{18}(S^8) \cong \Z/8\langle\sigma_8\circ\widehat{\nu}_{15}\rangle\oplus\Z/8\langle\widehat{\nu_8}\circ\widehat{\sigma}_{11}\rangle\oplus\Z/3\langle\sigma_8\circ\alpha_1(15)\rangle\oplus\Z/3\langle\beta_1(8)\rangle\oplus\Z/2
        \qedhere\popQED
    \]
\end{prop} 

\begin{lem}[Toda, Lemma 5.14]\label{lem:[i_8,i_8]}
    There is a homotopy \([\iota_8,\iota_8]\simeq2\cdot\sigma_8-\Sigma\sigma'\). \qed
\end{lem} 

Step 1 of Strategy \ref{strategy} is given by the following lemma.

\begin{lem}\label{lem:G4OP2_tau-bar}
    For any \(\tau\in\pi_{3}(\mathrm{SO}(16))\) there exist integers \(a_1\) and \(a_2\), modulo 8 and 3 respectively, such that \(\sigma_8\circ\overline{\tau}\simeq a_1\cdot(\sigma_8\circ\widehat{\nu}_{15})+a_2\cdot(\sigma_8\circ\alpha_1(15))\).
\end{lem}

\begin{proof} 
    With $k=4$ and $m=8$, by~(\ref{taubardef}) we have $\overline{\tau}\in\pi_{18}(S^{15})$. By Proposition~\ref{prop:nu_n} we obtain 
    \[
        \overline{\tau}\in\pi_{18}(S^{15}) \cong \Z/8\langle\widehat{\nu}_{15}\rangle\oplus\Z/3\langle\alpha_1(15)\rangle.
    \] 
    Therefore $\overline{\tau}=a_{1}\cdot\widehat{\nu}_{15}+a_{2}\cdot \alpha_{1}(15)$ for some integers $a_{1}$ and $a_{2}$ modulo $8$ and $3$ respectively. The statement of the lemma follows by left distributivity.
\end{proof}

Now we turn to Step 2 of Strategy \ref{strategy}. For \(\lambda\in\pi_{11}(S^8)\), by Proposition \ref{prop:nu_n} we have 
\[
    \lambda\simeq x\cdot\widehat{\nu}_8+y\cdot\alpha_1(8)
\] 
for integers \(x\) and \(y\) considered modulo 8 and 3, respectively. 

\begin{lem}\label{lem:xyxi} 
    There is an odd integer $\xi$ such that 
    \[\lambda\circ\sigma_{11}+[\iota_8,\lambda]\simeq2x\cdot(\sigma_8\circ\widehat{\nu}_{15})+(x-x\xi)\cdot(\widehat{\nu}_8\circ\widehat{\sigma}_{11})+2y\cdot(\sigma_8\circ\alpha_1(15)).\] 
\end{lem}

\begin{proof}  
    As in Proposition \ref{prop:sigma_n}, write \(\sigma_{11}\simeq\widehat{\sigma}_{11}+\widetilde{\alpha}_1(11)+\alpha_2(11)\). Since all these homotopy classes are suspensions we may distribute on the left. Using this and repeated applications of Lemma \ref{lem:coprime_order} we obtain
    \begin{align*}
        \lambda\circ\sigma_{11} & \simeq (x\cdot\widehat{\nu}_8+y\cdot\alpha_1(8))\circ(\widehat{\sigma}_{11}+\widetilde{\alpha_1}(11)+\alpha_2(11)) \\
        & \simeq x\cdot(\widehat{\nu}_8\circ\widehat{\sigma}_{11})+y\cdot(\alpha_1(8)\circ\alpha_2(15)).
    \end{align*}
    By \cite{toda}*{Lemma 13.8}, the composite \(\alpha_1(8)\circ\alpha_2(15)\) is homotopic to \((-3)\cdot\beta_1(8)\). By Proposition~\ref{prop:G4OP2_htpygps} the class \(\beta_1(8)\) has order 3, so we obtain 
    \begin{equation}\label{eq:lambda_sigma - k=4}
        \lambda\circ\sigma_{11}\simeq x\cdot(\widehat{\nu}_8\circ\widehat{\sigma}_{11}).
    \end{equation}
    
    For \([\iota_8,\lambda]\), Proposition \ref{prop:nu_n} implies that \(\lambda\) is suspension, so applying Lemma \ref{lem:whitehead_prods_and_susps} gives \[[\iota_8,\lambda]\simeq[\iota_8,\iota_8]\circ\Sigma^7\lambda.\] By Lemma \ref{lem:[i_8,i_8]} we have \([\iota_8,\iota_8]\simeq2\cdot\sigma_8-\Sigma\sigma'\). Further, \cite{toda}*{(7.19)} implies that there exists an odd integer \(\xi\) such that \(\Sigma\sigma'\circ\widehat{\nu}_{15}\simeq\xi\cdot\widehat{\nu}_8\circ\widehat{\sigma}_{11}\). Therefore  
    \begin{equation}\label{eq:[iota8,lamda] - k=4}
        [\iota_8,\lambda]\simeq(2\cdot\sigma_8-\Sigma\sigma')\circ(x\cdot\widehat{\nu}_{15}+y\cdot\alpha_1(15))\simeq 2x\cdot(\sigma_8\circ\widehat{\nu}_{15})-x\xi\cdot(\widehat{\nu}_8\circ\widehat{\sigma}_{11})+2y\cdot(\sigma_8\circ\alpha_1(15))
    \end{equation}
    since \(\Sigma\sigma'\circ\alpha_1(15)\) is null homotopic by Lemma \ref{lem:coprime_order}.
    
    Combining (\ref{eq:lambda_sigma - k=4}) and  (\ref{eq:[iota8,lamda] - k=4}) implies that 
    \[\lambda\circ\sigma_{11}+[\iota_8,\lambda]\simeq2x\cdot(\sigma_8\circ\widehat{\nu}_{15})+(x-x\xi)\cdot(\widehat{\nu}_8\circ\widehat{\sigma}_{11})+2y\cdot(\sigma_8\circ\alpha_1(15))\] 
    as asserted. 
\end{proof} 

\begin{rem}\label{rem:xi_postpone} 
In this case, Step 3 of Strategy~\ref{strategy} aims for a homotopy 
\begin{equation} \label{xiremarkeqn} 
    \sigma_{8}\circ\overline{\tau}+\lambda\circ\sigma_{11}+[\iota_{8},\lambda]\simeq\pm\sigma_{8}\circ\overline{\omega}. 
\end{equation} 
Both sides of the homotopy are elements in $\pi_{18}(S^{8})$, which by Lemma~\ref{lem:G4OP2_tau-bar} has generators that include $\sigma_{8}\circ\overline{\nu}_{15}$, $\widehat{\nu}_{8}\circ\overline{\sigma}_{11}$ and $\sigma_{8}\circ\alpha_{1}(15)$. By Lemma~\ref{lem:xyxi}, the left side of~(\ref{xiremarkeqn}) involves $\widehat{\nu}_{8}\circ\widehat{\sigma}_{11}$ with coefficient $(x-x\xi)$. By Lemma~\ref{lem:G4OP2_tau-bar}, the right side of~(\ref{xiremarkeqn}) involves $\widehat{\nu}_{8}\circ\widehat{\sigma}_{11}$ with coefficient $0$. Therefore it must be the case that \(x-x\xi\equiv0\text{ (mod 8)}\) if such a homotopy holds. If \(\xi\equiv1\text{ (mod 8)}\) then \(x\) may take any value, but otherwise \(x\) is forced to be even. 
\end{rem}

\begin{prop}\label{prop:xi}
    Let \(\xi\) be the odd integer of Lemma \ref{lem:xyxi} and let \(\tau\in\pi_3(\mathrm{SO}(16))\) be an arbitrary twisting.
    \begin{enumerate}
        \item[(i)] If \(\xi\equiv1\text{ (mod 8)}\) then \(\mathcal{G}_\tau^4(\O P^2)\) can take exactly two possible homotopy types.
        \item[(ii)] If \(\xi\equiv5\text{ (mod 8)}\) then \(\mathcal{G}_\tau^4(\O P^2)\) can take exactly three possible homotopy types.
        \item[(iii)] If \(\xi\equiv3\text{ or }7\text{ (mod 8)}\) then \(\mathcal{G}_\tau^4(\O P^2)\) can take exactly five possible homotopy types.
    \end{enumerate}
\end{prop}

\begin{proof}
    By Lemma \ref{lem:G4OP2_tau-bar}, for twistings \(\tau,\omega\in\pi_3(\mathrm{SO}(16))\) we may write \[\sigma_8\circ\overline{\tau}\simeq a_1\cdot(\sigma_8\circ\widehat{\nu}_{15})+a_2\cdot(\sigma_8\circ\alpha_1(15))\text{\; and \;}\sigma_{8}\circ\overline{\omega}\simeq b_1\cdot(\sigma_8\circ\widehat{\nu}_{15})+b_2\cdot(\sigma_8\circ\alpha_1(15))\] for some integers \(a_1\) and \(b_1\) modulo 8, and \(a_2\) and \(b_2\) modulo 3. By Proposition \ref{prop:lambda_delta}, there is a homotopy equivalence \(\mathcal{G}_\tau^4(\O P^2)\simeq\mathcal{G}_\omega^4(\O P^2)\) if and only if there exists a \(\lambda\in\pi_{11}(S^8)\) that gives rise to a homotopy \(\sigma_8\circ\overline{\tau}+\lambda\circ\sigma_{11}+[\iota_8,\lambda]\simeq\pm\sigma_8\circ\overline{\omega}\). By Lemma~\ref{lem:xyxi}, writing \(\lambda\simeq x\cdot\widehat{\nu}_8+y\cdot\alpha_1(8)\) and applying the congruence \(x-x\xi\equiv0\text{ (mod 8)}\) in Remark \ref{rem:xi_postpone}, we obtain 
    \[
        \sigma_8\circ\overline{\tau}+\lambda\circ\sigma_{11}+[\iota_8,\lambda] \simeq(a_1+2x)\cdot(\sigma_8\circ\widehat{\nu}_{15})+(a_2+2y)\cdot(\sigma_8\circ\alpha_1(15)).
    \] 
    Therefore there is a homotopy equivalence \(\mathcal{G}_\tau^4(\O P^2)\simeq\mathcal{G}_\omega^4(\O P^2)\) if and only if
    \[(a_1+2x)\cdot(\sigma_8\circ\widehat{\nu}_{15})+(a_2+2y)\cdot(\sigma_8\circ\alpha_1(15))\simeq b_1\cdot(\sigma_8\circ\widehat{\nu}_{15})+b_2\cdot(\sigma_8\circ\alpha_1(15)). 
    \]
    Thus a homotopy equivalence exists if and only if $a_{1}+2x\equiv \pm b_1\, \text{(mod 8)}$ and $a_{2}+2y\equiv \pm b_2\, \text{(mod 3)}$. Further, given any \(a_2\) and \(b_2\) modulo $3$, there always exists a \(y\) such that \(a_2+2y\equiv\pm b_2\pmod{3}\). Hence a homotopy equivalence exists if and only if $a_{1}+2x\equiv \pm b_1\, \text{(mod 8)}$.
    
    First observe that $a_{1}+2x\equiv \pm b_1\, \text{(mod 8)}$ implies that $a_{1}\equiv b_1\, \text{(mod 2)}$. So if $a_{1}\not\equiv b_1\, \text{(mod 2)}$ then \(\mathcal{G}_\tau^4(\O P^2)\not\simeq\mathcal{G}_\omega^4(\O P^2)\). In particular, if $a_{1}$ is even and $b_{1}$ is odd then \(\mathcal{G}_\tau^4(\O P^2)\not\simeq\mathcal{G}_\omega^4(\O P^2)\). This implies that \(\O P^2\) is not \(\mathcal{G}^4\)-stable and hence directly answers GSI in the negative. 

    We now turn to GSII and enumerating the possible homotopy types for \(\mathcal{G}^4_\tau(\O P^2)\). This depends on the possible choices of $x$ that give $a_{1}+2x\equiv \pm b_1\, \text{(mod 8)}$, while Remark \ref{rem:xi_postpone} implies that $x$ must also satisfy \(x-x\xi\equiv0\text{ (mod 8)}\). Since $a_{1}$ is an integer modulo 8 there are at most eight possible homotopy types; we label each one by the value of \(a_1\) and write \(\mathcal{G}_0^4(\O P^2),\mathcal{G}_1^4(\O P^2),\dots,\mathcal{G}_7^4(\O P^2)\). There are three cases, depending on \(\xi\) modulo~8.
    
    \textit{Part (i):} if \(\xi\equiv1\pmod{8}\) then \(x-x\xi\equiv0\pmod{8}\) holds for all \(x\). Taking $x=1$, we obtain $a_{1}+2\equiv \pm b_1\, \text{(mod 8)}$ if and only if $a_{1}\equiv b_1\, \text{(mod 2)}$. Thus there are homotopy equivalences: \[\mathcal{G}_0^4(\O P^2)\simeq\mathcal{G}_2^4(\O P^2)\simeq\mathcal{G}_4^4(\O P^2)\simeq\mathcal{G}_6^4(\O P^2)\text{\; and \;}\mathcal{G}_1^4(\O P^2)\simeq\mathcal{G}_3^4(\O P^2)\simeq\mathcal{G}_5^4(\O P^2)\simeq\mathcal{G}_7^4(\O P^2).\] 
    On the other hand, we have already seen that if $a_{1}\not\equiv b_1\, \text{(mod 2)}$ then \(\mathcal{G}_\tau^4(\O P^2)\not\simeq\mathcal{G}_\omega^4(\O P^2)\). Thus there are exactly two homotopy types in this case.

    \textit{Parts (ii) and (iii):} if \(\xi\not\equiv1\pmod{8}\) then we are in one of two situations. If \(\xi\equiv5\pmod{8}\) then the condition \(x-x\xi\equiv0\text{ (mod 8)}\) occurs if and only if \(x\) is even.  Taking $x=2$, we obtain  $a_{1}+4\equiv \pm b_1\, \text{(mod 8)}$ if and only if \(a_1\equiv \pm b_1\pmod{4}\). Thus there are homotopy equivalences
    \[
    \mathcal{G}_0^4(\O P^2)\simeq\mathcal{G}_4^4(\O P^2)\text{, \;}\mathcal{G}_1^4(\O P^2)\simeq\mathcal{G}_3^4(\O P^2)\simeq\mathcal{G}_5^4(\O P^2)\simeq\mathcal{G}_7^4(\O P^2)\text{\; and \;}\mathcal{G}_2^4(\O P^2)\simeq\mathcal{G}_6^4(\O P^2).
    \] 
    On the other hand, no other value of $x$ will result in additional homotopy equivalences between these three homotopy types. This proves (ii). 
    If \(\xi\equiv3\text{ or }7\pmod{8}\) then the condition \(x-x\xi\equiv0\text{ (mod 8)}\) occurs if and only if \(x\equiv0\text{ or }4\pmod{8}\). This implies that \(a_1+2x\equiv a_1\pmod{8}\) and therefore there is a congruence $a_{1}+2x\equiv \pm b_1\, \text{(mod 8)}$ if and only if \(a_1\equiv \pm b_1\pmod{8}\). Thus there are homotopy equivalences
    \[
        \mathcal{G}_0^4(\O P^2)\text{, \;}\mathcal{G}_1^4(\O P^2)\simeq\mathcal{G}_7^4(\O P^2)\text{, \;}\mathcal{G}_2^4(\O P^2)\simeq\mathcal{G}_6^4(\O P^2)\text{, \;}\mathcal{G}_3^4(\O P^2)\simeq\mathcal{G}_5^4(\O P^2)\text{\; and \;}\mathcal{G}_4^4(\O P^2).
    \] 
    This proves (iii). 
\end{proof}

Proposition \ref{prop:xi} shows that for \(k=4\) the answer to GSII for \(\O P^2\) is at least 2, so it immediately implies the following.

\begin{thm}\label{thm:g4op2}
    \(\O P^2\) is not \(\mathcal{G}^4\)-stable. \qed
\end{thm}

\noindent 
\textbf{The $k=8$ case}. 
There is a gyration $\mathcal{G}^{8}_{\tau}(\mathbb{O}P^{2})$ for each $\tau\in\pi_{7}(SO(16))\cong\mathbb{Z}$. The next two statements describe the relevant homotopy groups, generators and relations.

\begin{prop}[Toda]\label{prop:G8OP2_htpygps}
    There are group isomorphisms:
    \begin{enumerate}
        \item[(i)] \(\pi_{15}(S^8) \cong \Z\langle\sigma_8\rangle\oplus\Z/8\langle\Sigma\sigma'\rangle\oplus\Z/5\oplus\Z/3\);
        \item[(ii)] \(\pi_{22}(S^8) \cong \Z/16\langle\sigma_8\circ\widehat{\sigma}_{15}\rangle\oplus\Z/8\langle\Sigma\sigma'\circ\widehat{\sigma}_{15}\rangle\oplus\Z/5\langle\sigma_8\circ\widetilde{\alpha}_1(15)\rangle\oplus\Z/3\langle\sigma_8\circ\alpha_2(15)\rangle\oplus\Z/4\oplus\Z/3\). \qed
    \end{enumerate}
\end{prop}

\begin{lem}[Toda, p.101]\label{lem:[iota,sigma]}
    There is a homotopy \([\iota_8,\sigma_8]\simeq\pm\Big(2\cdot(\sigma_8\circ\widehat{\sigma}_{15})-(\Sigma\sigma'\circ\widehat{\sigma}_{15})\Big)\). \qed
\end{lem}

Step 1 of Strategy \ref{strategy} is given by the following lemma.

\begin{lem}\label{lem:G8OP2_tau-bar}
    For any \(\tau\in\pi_{7}(\mathrm{SO}(16))\) there exist integers \(a_1\), \(a_2\) and \(a_3\), modulo 16, 5 and 3 respectively, such that \(\sigma_8\circ\overline{\tau}\simeq a_1\cdot(\sigma_8\circ\widehat{\sigma}_{15})+a_2\cdot(\sigma_8\circ\widetilde{\alpha}_1(15))+a_3\cdot(\sigma_8\circ\alpha_2(15))\).    
\end{lem}

\begin{proof} 
    With $k=m=8$, by~(\ref{taubardef}) we have $\overline{\tau}\in\pi_{22}(S^{15})$. By Proposition~\ref{prop:nu_n} we obtain
    \[
        \overline{\tau}\in\pi_{22}(S^{15}) \cong \Z/16\langle\widehat{\sigma}_{15}\rangle\oplus\Z/5\langle\widetilde{\alpha}_1(15)\rangle\oplus\Z/3\langle\alpha_2(15)\rangle.
    \]
    Therefore $\overline{\tau}=a_{1}\cdot\widehat{\sigma}_{15}+a_{2}\cdot \widehat{\alpha}_{1}(15)+a_{3}\cdot\alpha_{1}(15)$ for some integers $a_{1}, a_{2}$ and $a_{3}$ modulo $16$, $5$ and $3$ respectively. The statement of the lemma follows by left distributivity.
\end{proof}

The following lemma constitutes Step 2. We will argue similarly to Lemma \ref{lem:lambda_pi7(S4)}.

\begin{lem}\label{lem:lambdacomps - k=8}
    There exist \(\lambda_1,\lambda_2,\lambda_3\in\pi_{15}(S^8)\) such that:
    \begin{enumerate}
        \item[(i)] \(\lambda_1\circ\sigma_{15}+[\iota_8,\lambda_1]\simeq\sigma_8\circ\widehat{\sigma}_{15}\);
        \item[(ii)] \(\lambda_2\circ\sigma_{15}+[\iota_8,\lambda_2]\simeq\sigma_8\circ\widetilde{\alpha}_1(15)\);
        \item[(iii)] \(\lambda_3\circ\sigma_{15}+[\iota_8,\lambda_3]\simeq\sigma_8\circ\alpha_2(15).\)
    \end{enumerate}
\end{lem}

\begin{proof}
    Before finding the asserted homotopy classes we consider an element \(\Lambda\in\pi_{15}(S^8)\) which, by Proposition \ref{prop:G8OP2_htpygps}(i), may be written as 
    \[
        \Lambda\simeq w\cdot\sigma_8+x\cdot\Sigma\sigma'
    \] 
    for integers \(w\) and \(x\), with \(x\) considered modulo 8. The composite $\Lambda\circ\sigma_{15}$ and the Whitehead product \([\iota_8,\Lambda]\) both lie in the homotopy group \(\pi_{22}(S^8)\); the first aim is to express these in terms of the generators given in Proposition \ref{prop:G8OP2_htpygps}(ii). 

    First consider \(\Lambda\circ\sigma_{15}\). Since $\sigma_{15}$ is a suspension, right-distributivity gives 
    \[\Lambda\circ\sigma_{15}\simeq w\cdot(\sigma_8\circ\sigma_{15})+x\cdot(\Sigma\sigma'\circ\sigma_{15}).\]
    Writing \(\sigma_{15}\simeq\widehat{\sigma}_{15}+\widetilde{\alpha}_1(15)+\alpha_2(15)\), repeated application of Lemma \ref{lem:coprime_order} thus gives a homotopy
    \begin{equation}\label{eq:lambda_sigma - k=8}
    \Lambda\circ\sigma_{15} \simeq w\cdot(\sigma_8\circ\widehat{\sigma}_{15})+w\cdot(\sigma_8\circ\widetilde{\alpha}_1(15))+w\cdot(\sigma_8\circ\alpha_2(15))+x\cdot(\Sigma\sigma'\circ\widehat{\sigma}_{15}). 
    \end{equation} 

    Next, consider \([\iota_8,\Lambda]\). By additivity we may consider $[\iota_{8},\Sigma\sigma']$ and $[\iota_{8},\sigma_{8}]$ separately. Since \(\Sigma\sigma'\) is a suspension, applying Lemma \ref{lem:whitehead_prods_and_susps} and using the fact that \(\Sigma^8\sigma'\simeq2\cdot\widehat{\sigma}_{15}\) gives 
    \begin{equation*}    
    [\iota_8,\Sigma\sigma'] \simeq [\iota_8,\iota_8]\circ\Sigma^8\sigma'\simeq2\cdot([\iota_8,\iota_8]\circ\widehat{\sigma}_{15}).
    \end{equation*}
    Further, Lemma \ref{lem:[i_8,i_8]} then shows 
    \[[
        \iota_8,\Sigma\sigma']\simeq 4\cdot(\sigma_8\circ\widehat{\sigma}_{15})-2\cdot(\Sigma\sigma'\circ\widehat{\sigma}_{15}).
    \]
    Hence, applying Lemma \ref{lem:[iota,sigma]} for \([\iota_8,\sigma_8]\), we obtain
    \begin{equation}\label{eq:lambda_whitehead - k=8}
        [\iota_8,\Lambda]\simeq(\pm2w+4x)\cdot(\sigma_8\circ\widehat{\sigma}_{15})-(\pm w+2x)\cdot(\Sigma\sigma'\circ\widehat{\sigma}_{15}).
    \end{equation}
    Now combining (\ref{eq:lambda_sigma - k=8}) and (\ref{eq:lambda_whitehead - k=8}), we have
    \begin{align}\label{eq:general Lambda - m=8 k=4}
    \begin{split}
        \Lambda\circ\sigma_{15}+[\iota_8,\Lambda] \simeq &~(w\pm2w+4x)\cdot(\sigma_8\circ\widehat{\sigma}_{15})-(\pm w+x)\cdot(\Sigma\sigma'\circ\widehat{\sigma}_{15}) \\ &~ +w\cdot(\sigma_8\circ\widetilde{\alpha}_1(15))+w\cdot(\sigma_8\circ\alpha_2(15)).
    \end{split}
    \end{align}    
    
    For part (i) there are two cases, depending on the `\(\pm\)' signs in (\ref{eq:lambda_whitehead - k=8}). In the `\(+\)' case, (\ref{eq:general Lambda - m=8 k=4}) becomes
    \[
        \lambda_1\circ\sigma_{15}+[\iota_8,\lambda_1] \simeq(3w+4x)\cdot(\sigma_8\circ\widehat{\sigma}_{15})+(-w-x)\cdot(\Sigma\sigma'\circ\widehat{\sigma}_{15}) + w\cdot(\sigma_8\circ\widetilde{\alpha}_1(15))+w\cdot(\sigma_8\circ\alpha_2(15)).
    \] 
    So to obtain \(\lambda_1\circ\sigma_{15}+[\iota_8,\lambda_1]\simeq\sigma_8\circ\widehat{\sigma}_{15}\) as in (i), we must have \(3w+4x\equiv1\text{ (mod 16)}\), \(x\equiv-w\text{ (mod 8)}\), \(w\equiv0\text{ (mod 5)}\) and \(w\equiv0\text{ (mod 5)}\). Taking \(w=15\) and \(x=1\) solves this system, giving \(\lambda_1= 15\cdot\sigma_8+\Sigma\sigma'\). 
    In the `\(-\)' case (\ref{eq:general Lambda - m=8 k=4}) becomes
    \[
        \lambda_1\circ\sigma_{15}+[\iota_8,\lambda_1]\simeq(-w+4x)\cdot(\sigma_8\circ\widehat{\sigma}_{15})+(w-x)\cdot(\Sigma\sigma'\circ\widehat{\sigma}_{15})+w\cdot(\sigma_8\circ\widetilde{\alpha}_1(15))+w\cdot(\sigma_8\circ\alpha_2(15)).
    \] 
    So to obtain \(\lambda_1\circ\sigma_{15}+[\iota_8,\lambda_1]\simeq\sigma_8\circ\widehat{\sigma}_{15}\) as in (i), we must have \(-w+4x\equiv1\text{ (mod 16)}\), \(x\equiv w\text{ (mod 8)}\), \(w\equiv0\text{ (mod 5)}\) and \(w\equiv0\text{ (mod 3)}\). Taking \(w=75\) and \(x=3\) solves this system, giving \(\lambda_1=75\cdot\sigma_8+3\cdot\Sigma\sigma'\). Thus, in either case, there is a \(\lambda_1\) that satisfies (i). 
    \smallskip 
    
    We now move to part (ii). To obtain \(\lambda_2\circ\sigma_{15}+[\iota_8,\lambda_2]\simeq\sigma_8\circ\widetilde{\alpha}_1(15)\), with the left side written as in~(\ref{eq:general Lambda - m=8 k=4}), we need a solution to the system of congruences given by 
    \[w\pm2w+4x\equiv0\text{ (mod 16)}, \pm w+x\equiv0\text{ (mod 8)}, w\equiv1\text{ (mod 5)}\text{\; and \;}w\equiv0\text{ (mod 3)}.\] Taking \(w=96\) and \(x=0\) solves the system, giving \(\lambda_2=96\cdot\sigma_8\). 
    \smallskip 
    
    Finally, we consider part (iii). To obtain \(\lambda_3\circ\sigma_{15}+[\iota_8,\lambda_3]\simeq\sigma_8\circ\alpha_2(15)\), with the left side written as in (\ref{eq:general Lambda - m=8 k=4}), we need to solve \[w\pm2w+4x\equiv0\text{ (mod 16)}, \pm w+x\equiv0\text{ (mod 8)}, w\equiv0\text{ (mod 5)}\text{\; and \;}w\equiv1\text{ (mod 3)}.\] Taking \(w=160\) and \(x=0\) solves the system, giving \(\lambda_3=160\cdot\sigma_8\). 
\end{proof}

Finally, we proceed to Step 3 of Strategy \ref{strategy}.

\begin{thm}\label{thm:g8op2}
    \(\O P^2\) is \(\mathcal{G}^8\)-stable, i.e.~\(\mathcal{G}_\tau^8(\O P^2)\simeq\mathcal{G}_\omega^8(\O P^2)\) for all twistings \(\tau,\omega\in\pi_7(\mathrm{SO}(16))\).
\end{thm}

\begin{proof}
    By Lemma \ref{lem:G8OP2_tau-bar} for any two twistings \(\tau\) and \(\omega\) we may write 
    \[
        \sigma_8\circ\overline{\tau}\simeq a_1\cdot(\sigma_8\circ\widehat{\sigma}_{15})+a_2\cdot(\sigma_8\circ\widetilde{\alpha}_1(15))+a_3\cdot(\sigma_8\circ\alpha_2(15))
    \]
    and
    \[
        \sigma_8\circ\overline{\omega}\simeq b_1\cdot(\sigma_8\circ\widehat{\sigma}_{15})+b_2\cdot(\sigma_8\circ\widetilde{\alpha}_1(15))+b_3\cdot(\sigma_8\circ\alpha_2(15))
    \] 
    for some integers \(a_1\) and \(b_1\) considered modulo 16, \(a_2\) and \(b_2\) modulo 5, and \(a_3\) and \(b_3\) modulo 3. Take \(\lambda=(b_1-a_1)\cdot\lambda_1+(b_2-b_1)\cdot\lambda_2+(a_3-b_3)\cdot\lambda_3\), where $\lambda_{1}$, $\lambda_{2}$ and $\lambda_{3}$ are as in Lemma \ref{lem:lambdacomps - k=8}. Then Step 3 of Strategy~\ref{strategy} implies that $\mathcal{G}^{8}_{\tau}(\mathbb{O}P^{2})\simeq\mathcal{G}^{8}_{\omega}(\mathbb{O}P^{2})$ holds for any choice of $\tau$ and $\omega$, so $\mathbb{O}P^{2}$ is \(\mathcal{G}^8\)-stable.
\end{proof}

\section{Computations for \(\mathcal{G}_\tau^k(\mathbb{O}P^2)\) when \(9 \leq k \leq 14\)}
\label{sec:G(op2)-9to14}

We now turn to the three remaining cases when~\(k=9\), \(k=10\) and \(k=12\).
\medskip 

\noindent 
\textbf{The $k=9$ case.} There is a gyration $\mathcal{G}^{9}_{\tau}(\mathbb{O}P^{2})$ for each $\tau\in\pi_{8}(SO(16))\cong\mathbb{Z}/2$. The next two statements describe the relevant homotopy groups, generators, and relations. 

\begin{prop}[Toda]\label{prop:G9OP2_htpygps}
    There are group isomorphisms:
    \begin{enumerate}
        \item[(i)] \(\pi_{23}(S^{15})\cong\Z/2\langle\overline{\nu}_{15}\rangle\oplus\Z/2\langle\varepsilon_{15}\rangle\);
        \item[(ii)] \(\pi_{16}(S^8) \cong \Z/2\langle\sigma_8\circ\eta_{15}\rangle\oplus\Z/2\langle\Sigma\sigma'\circ\eta_{15}\rangle\oplus\Z/2\langle\overline{\nu}_8\rangle\oplus\Z/2\langle\varepsilon_8\rangle\);
        \item[(iii)] \(\pi_{23}(S^8) \cong \Z/2\langle\sigma_8\circ\overline{\nu}_{15}\rangle\oplus\Z/2\langle\sigma_8\circ\varepsilon_{15}\rangle\oplus\Z/2\langle\Sigma\sigma'\circ\overline{\nu}_{15}\rangle\oplus\Z/2\langle\Sigma\sigma'\circ\varepsilon_{15}\rangle\oplus\Z/2\oplus\Z/120\). \qed
    \end{enumerate}
\end{prop}

\begin{lem}[\cite{toda}*{Lemma 6.4}]\label{lem:eta_sigma}
    There is a homotopy \(\eta_9\circ\sigma_{10}\simeq\overline{\nu}_9+\varepsilon_9\), and for \(n\geq10\) there are homotopies \(\eta_n\circ\sigma_{n+1}\simeq\sigma_n\circ\eta_{n+7}\simeq\overline{\nu}_n+\varepsilon_n\). \qed
\end{lem}

Step 1 of Strategy \ref{strategy} is given by the following lemma.

\begin{lem}\label{lem:G9OP2_tau-bar}
    Given \(\tau\in\pi_{8}(\mathrm{SO}(16))\), if \(\tau\) is non-trivial then \(\sigma_{8}\circ\overline{\tau}\simeq\sigma_{8}\circ\overline{\nu}_{15}+\sigma_{8}\circ\varepsilon_{15}\).
\end{lem}

\begin{proof}
     If \(\tau\) is non-trivial, then Corollary \ref{cor:tau-bar_non-trivial} implies that \(\overline{\tau}\) is also non-trivial. With $k=9$ and $m=8$, by its definition in~(\ref{taubardef}), we have $\overline{\tau}\in\pi_{23}(S^{15})$. By Proposition \ref{prop:jhom} we have \(\Sigma\overline{\tau}\in im(J)\), which by \cite{ravenel_cobordism}*{Theorem 1.1.13}\footnote{Stated citing \cite{adamsIV} and \cite{quillen_adamsconj}} is generated by \(\eta_{16}\circ\sigma_{17}\). By Lemma \ref{lem:eta_sigma}, \(\eta_{16}\circ\sigma_{17} \simeq\overline{\nu}_{16}+\varepsilon_{16}\). As we are in the stable range, we may desuspend, and thus
    \begin{equation}\label{imJ_in_pi23(S15)} 
    \overline{\tau}\in\Z/2\langle\overline{\nu}_{15}+\varepsilon_{15}\rangle\subset\pi_{23}(S^{15}).   
    \end{equation}
    The result then follows immediately by left-distributivity.
\end{proof}

We now move to Step 2. 

\begin{lem}\label{lem:lambdas - k=9}
    There exists \(\lambda_1\in\pi_{16}(S^8)\) such that \(\lambda_1\circ\sigma_{16}+[\iota_8,\lambda_1]\simeq\sigma_8\circ\overline{\nu}_{15}+\sigma_8\circ\varepsilon_{15}\).
\end{lem}

\begin{proof}
    We prove the asserted homotopy for \(\lambda_1=(\sigma_8\circ\eta_{15})+\overline{\nu}_8+\varepsilon_8\in\pi_{16}(S^8)\). Consider the composite \(\lambda_1\circ\sigma_{16}\). By \cite{toda}*{Lemma 10.7} there are null homotopies for \(\overline{\nu}_8\circ\sigma_{16}\) and \(\varepsilon_8\circ\sigma_{16}\), and so 
    \[
        \lambda_1\circ\sigma_{16}\simeq(\sigma_8\circ\eta_{15}\circ\sigma_{16})+(\overline{\nu}_8\circ\sigma_{16})+(\varepsilon_8\circ\sigma_{16})\simeq(\sigma_8\circ\eta_{15}\circ\sigma_{16})+\ast+\ast.
    \] 
    By Lemma \ref{lem:eta_sigma}, there is a homotopy \(\sigma_8\circ\eta_{15}\circ\sigma_{16}\simeq\sigma_8\circ\overline{\nu}_{15}+\sigma_8\circ\varepsilon_{15}\) and hence 
    \begin{equation*}
        \lambda_1\circ\sigma_{16}\simeq
        \sigma_8\circ\overline{\nu}_{15}+\sigma_8\circ\varepsilon_{15}.
    \end{equation*}
    
    It therefore remains to show that the Whitehead product \([\iota_8,\lambda_1]\) is null homotopic. By additivity we consider each of the three summands of \(\lambda_1\) separately. First, since the elements \(\overline{\nu}_8\) and \(\varepsilon_8\) are both suspensions, by Lemma \ref{lem:whitehead_prods_and_susps} we have 
    \[
        [\iota_8,\overline{\nu}_8]\simeq[\iota_8,\iota_8]\circ\overline{\nu}_{15}\text{\; and \;}[\iota_8,\varepsilon_8]\simeq[\iota_8,\iota_8]\circ\varepsilon_{15}.
    \]
    Recalling Lemma \ref{lem:[i_8,i_8]}, and noting that both $\sigma_{8}\circ\overline{\nu}_{15}$ and $\sigma_{8}\circ\varepsilon_{15}$ have order $2$ by Proposition \ref{prop:G9OP2_htpygps}(iii), we obtain
    \begin{equation}\label{whiteheads1 - k=9}
        [\iota_8,\overline{\nu}_8]\simeq\Sigma\sigma'\circ\overline{\nu}_{15}\text{\; and \;}[\iota_8,\varepsilon_8]\simeq\Sigma\sigma'\circ\varepsilon_{15}.
    \end{equation}
    On the other hand, for \([\iota_8,\sigma_8\circ\eta_{15}]\), by Lemma \ref{lem:whitehead_prods_and_susps} we have  \[[\iota_8,\sigma_8\circ\eta_{15}]\simeq[\iota_8,\sigma_8]\circ\eta_{22}.\] Once again applying Lemma \ref{lem:[iota,sigma]} for $[\iota_{8},
    \sigma_{8}]$ and noting that $\eta_{22}$ has order $2$, we obtain 
    \[[\iota_{8},\sigma_{8}]\circ\eta_{22}\simeq\Sigma\sigma'\circ\sigma_{15}\circ\eta_{22}.\] 
    By Lemma \ref{lem:eta_sigma}, $\sigma_{15}\circ\eta_{22}\simeq\overline{\nu}_{15}+\varepsilon_{15}$. Thus, putting this together and using left distributivity gives
    \begin{equation}\label{whiteheads2 - k=9}
        [\iota_8,\sigma_8\circ\eta_{15}]\simeq\Sigma\sigma'\circ\sigma_{15}\circ\eta_{22}\simeq \Sigma\sigma'\circ\overline{\nu}_{15}+\Sigma\sigma'\circ\varepsilon_{15}.
    \end{equation}
    Combining (\ref{whiteheads1 - k=9}) and (\ref{whiteheads2 - k=9}) gives \[[\iota_8,\lambda_1]\simeq[\iota_8,\sigma_8\circ\eta_{15}]+[\iota_8,\overline{\nu}_8]+[\iota_8,\varepsilon_8]\simeq2\cdot(\Sigma\sigma'\circ\overline{\nu}_{15})+2\cdot(\Sigma\sigma'\circ\varepsilon_{15})\simeq\ast\] where the null homotopy comes from both classes having order two.  
\end{proof}

Step 3 of Strategy \ref{strategy} follows swiftly.

\begin{thm}\label{thm:g9op2}
    \(\O P^2\) is \(\mathcal{G}^9\)-stable, i.e.~\(\mathcal{G}_\tau^9(\O P^2)\simeq\mathcal{G}_\omega^9(\O P^2)\) for all twistings \(\tau,\omega\in\pi_8(\mathrm{SO}(16))\).
\end{thm}

\begin{proof}
    Since \(\pi_8(\mathrm{SO}(16))\cong\Z/2\), there are two choices of twisting, so to prove gyration stability in this case we need only check when \(\tau\) is non-trivial and \(\omega\) is trivial. By Lemma \ref{lem:G9OP2_tau-bar} for $\overline{\tau}$ and Lemma~\ref{twisting12}~(iii) for $\overline{\omega}$, we may write 
    \[
        \sigma_8\circ\overline{\tau}\simeq \sigma_{8}\circ\overline{\nu}_{15}+\sigma_{8}\circ\varepsilon_{15}
        \text{\; and \;} \sigma_8\circ\overline{\omega}\simeq \ast.
    \] 
    It follows that taking \(\lambda_1\) as in Lemma \ref{lem:lambdas - k=9} gives 
    \[
        \sigma_8\circ\overline{\tau}+\lambda_1\circ\sigma_{16}+[\iota_8,\lambda_1] \simeq 2\cdot(\sigma_8\circ\overline{\nu}_{15}+\sigma_8\circ\varepsilon_{15})\simeq\ast\simeq\sigma_8\circ\overline{\omega}
    \] 
    thus proving \(\mathcal{G}^9\)-stability for \(\O P^2\), by Proposition \ref{prop:lambda_delta}.
\end{proof}

\noindent 

\textbf{The $k=10$ case}. There is a gyration $\mathcal{G}^{10}_{\tau}(\mathbb{O}P^{2})$ for each $\tau\in\pi_{9}(SO(16))\cong\mathbb{Z}/2$. 
The next two statements describe the relevant homotopy groups, generators and relations. We write \(\nu_n^3\) to denote the composite \(\nu_n\circ\nu_{n+3}\circ\nu_{n+6}\) for \(n\geq8\), and similarly let \(\eta_n^2\) be \(\eta_n\circ\eta_{n+1}\) for \(n\geq2\).

\begin{prop}[Toda]\label{prop:G10(OP2)_htpygps}
    There are group isomorphisms:
    \begin{enumerate}
        \item[(i)] \(\pi_{24}(S^{15})\cong\Z/2\langle\nu_{15}^3\rangle\oplus\Z/2\langle\eta_{15}\circ\varepsilon_{16}\rangle\oplus\Z/2\);
        \item[(ii)] \(\pi_{17}(S^8) \cong \Z/2\langle\sigma_8\circ\eta_{15}^2\rangle\oplus\Z/2\langle\nu_8^3\rangle\oplus\Z/2\langle\eta_8\circ\varepsilon_9\rangle\oplus\Z/2\oplus\Z/2\);
        \item[(iii)] \(\pi_{24}(S^8) \cong  \Z/2\langle\sigma_8\circ\nu_{15}^3\rangle\oplus\Z/2\langle\sigma_8\circ\eta_{15}\circ\varepsilon_{16}\rangle\oplus\Z/2\oplus\Z/2\oplus\Z/2\oplus\Z/2\oplus\Z/2\). \qed
    \end{enumerate}
\end{prop} 

\begin{lem}[\cite{toda}*{Lemma 6.3}]\label{lem:eta_nu-bar}
    For \(n\geq6\) there are homotopies \(\overline{\nu}_n\circ\eta_{n+8}\simeq\eta_n\circ\overline{\nu}_{n+1}\simeq\nu_n^3\). \qed
\end{lem}

We begin with Step 1 of Strategy \ref{strategy}. 

\begin{lem}\label{lem:G10(OP2)_tau-bar}
    Given \(\tau\in\pi_{9}(\mathrm{SO}(16))\), if \(\tau\) is non-trivial then \(\sigma_{8}\circ\overline{\tau}\simeq(\sigma_{8}\circ\nu_{15}^3)+(\sigma_{8}\circ\eta_{15}\circ\varepsilon_{16})\).
\end{lem}

\begin{proof}
    If \(\tau\) is non-trivial, then Corollary \ref{cor:tau-bar_non-trivial} implies that \(\overline{\tau}\) is also non-trivial. With $k=10$ and $m=8$, by its definition in~(\ref{taubardef}), we have $\overline{\tau}\in\pi_{24}(S^{15})$. By Proposition \ref{prop:jhom} we have \(\Sigma\overline{\tau}\in im(J)\), which by \cite{ravenel_cobordism}*{Theorem 1.1.13} is generated by \(\eta_{15}^2\circ\sigma_{17}\). In turn, by Lemmas \ref{lem:eta_sigma} and \ref{lem:eta_nu-bar}, this composite is homotopic to the class \(\nu_{15}^3+ \eta_{15}\circ\varepsilon_{16}\). As we are in the stable range we may de-suspend, and thus
    \begin{equation}\label{eq:imJ_in_pi24(S15)}
        \overline{\tau}\in\Z/2\langle\nu_{15}^3+ \eta_{15}\circ\varepsilon_{16}\rangle\subset\pi_{24}(S^{15}).
    \end{equation}
    The result then follows immediately by left-distributivity.
\end{proof}

We move on to Step 2 of Strategy \ref{strategy}.

\begin{lem}\label{lem:lambda - k=10}
    There exists \(\lambda_1\in\pi_{17}(S^8)\) such that \(\lambda_1\circ\sigma_{17}+[\iota_8,\lambda_1]\simeq(\sigma_8\circ\nu_{15}^3)+(\sigma_8\circ\eta_{15}\circ\varepsilon_{16})\).
\end{lem}

\begin{proof}
    We prove the identity for the class \(\lambda_1=(\sigma_8\circ\eta_{15}^2)+\nu_8^3+(\eta_8\circ\varepsilon_9)\in\pi_{17}(S^8)\). Consider \(\lambda_1\circ\sigma_{17}\). By \cite{toda}*{Lemma 10.7} there are null homotopies for the composites \(\varepsilon_9\circ\sigma_{17}\) and \(\overline{\nu}_9\circ\sigma_{17}\), and by Lemma \ref{lem:eta_nu-bar} we have \(\eta_8\circ\overline{\nu}_9\simeq\nu_8^3\). Therefore 
    \[
    \eta_8\circ\varepsilon_9\circ\sigma_{17}\simeq\ast\text{\; and \;}\nu_8^3\circ\sigma_{17}\simeq\eta_8\circ\overline{\nu}_9\circ\sigma_{17}\simeq\ast,
    \] implying that \(\lambda_1\circ\sigma_{17}\simeq\sigma_8\circ\eta_{15}^2\circ\sigma_{17}\). By Lemma \ref{lem:eta_sigma}, $\eta_{16}\circ\sigma_{17}\simeq\overline{\nu}_{16}+\varepsilon_{16}$. This together with $\eta_{15}\circ\overline{\nu}_{16}\simeq\nu_{15}^{3}$ gives
    \begin{equation}\label{eq:lambda1_comp_sigma - k=10}
        \lambda_1\circ\sigma_{17}\simeq \sigma_{8}\circ\eta^{2}_{15}\circ\sigma_{17}\simeq(\sigma_{8}\circ\eta_{15}\circ\overline{\nu}_{16})+(\sigma_{8}\circ\eta_{15}\circ\varepsilon_{16})\simeq(\sigma_8\circ\nu_{15}^3) + (\sigma_8\circ\eta_{15}\circ\varepsilon_{16}).
    \end{equation} 
    Thus if $[\iota_{8},\lambda_{1}]$ is null homotopic then the homotopy asserted by the lemma holds. 
    
    It remains to show that $[\iota_{8},\lambda_{1}]$ is null homotopic. By additivity we consider each of the three summands of \(\lambda_1\) separately. Both \(\nu_8^3\) and \(\eta_8\circ\varepsilon_9\) are suspensions, so by Lemma \ref{lem:whitehead_prods_and_susps} we obtain 
    \[
        [\iota_8,\nu_8^3] \simeq [\iota_8,\iota_8]\circ\nu_{15}^3 \text{\; and \;} [\iota_8,\eta_8\circ\varepsilon_9] \simeq [\iota_8,\iota_8]\circ\eta_{15}\circ\varepsilon_{16}.
    \]
    By Lemma \ref{lem:[i_8,i_8]}, and noting that both $\sigma_{8}\circ\nu^{3}_{15}$ and $\sigma_{8}\circ\eta_{15}\circ\varepsilon_{16}$ have order $2$ by Proposition \ref{prop:G10(OP2)_htpygps}~(iii), we obtain 
    \begin{equation}\label{eq:whitehead_susps - k=10}
        [\iota_8,\nu_8^3] \simeq \Sigma\sigma'\circ\nu_{15}^3 \text{\; and \;} [\iota_8,\eta_8\circ\varepsilon_9] \simeq \Sigma\sigma'\circ\eta_{15}\circ\varepsilon_{16}.
    \end{equation}
    For \([\iota_8,\sigma_8\circ\eta_{15}^2]\) we again use Lemma \ref{lem:whitehead_prods_and_susps} to obtain 
     \[[\iota_8,\sigma_8\circ\eta_{15}^2]\simeq[\iota_8,\sigma_8]\circ\eta_{22}^2.\] 
    Once again applying Lemma \ref{lem:[iota,sigma]} for $[\iota_{8}, \sigma_{8}]$, and since $\eta^{2}_{22}$ has order $2$, we obtain 
    \[
        [\iota_8,\sigma_8\circ\eta^{2}_{15}]\simeq\Sigma\sigma'\circ\sigma_{15}\circ\eta_{22}^2. 
    \] 
    By Lemma \ref{lem:eta_sigma}, $\sigma_{15}\circ\eta_{22}\simeq\overline{\nu}_{15}+\varepsilon_{15}$. Left distributivity then gives 
    \[\Sigma\sigma'\circ\sigma_{15}\circ\eta_{22}^2\simeq
        \Sigma\sigma'\circ\overline{\nu}_{15}\circ\eta_{23}+ \Sigma\sigma'\circ\varepsilon_{15}\circ\eta_{23}.
    \]
    Next, Lemma \ref{lem:eta_nu-bar} gives \(\overline{\nu}_{15}\circ\eta_{23}\simeq\nu_{15}^3\), 
    and as $\varepsilon_{15}\circ\eta_{23}$ is in the stable range it is homotopic to the composite $\eta_{15}\circ\varepsilon_{16}$. Therefore, stringing homotopies together, we obtain 
    \begin{equation}\label{eq:whitehead_sigma8comp - k=10}
        [\iota_8,\sigma_8\circ\eta_{15}^2] \simeq 
        \Sigma\sigma'\circ\nu^{3}_{15}+\Sigma\sigma'\circ\eta_{15}\circ\varepsilon_{16}. 
    \end{equation} 
    Combining (\ref{eq:whitehead_susps - k=10}) and (\ref{eq:whitehead_sigma8comp - k=10}) then gives   
    \[[\iota_{8},\lambda_1]\simeq 2\cdot(\Sigma\sigma'\circ\nu^{3}_{15})+2\cdot(\Sigma\sigma'\circ\eta_{15})\circ\varepsilon_{16}.\] 
    Both $\nu^{3}_{15}$ and $\eta_{15}$ have order $2$, so $[\iota_{8},\lambda_{1}]$ is null homotopic, as required.
\end{proof}

\begin{thm}\label{thm:g10op2}
    \(\O P^2\) is \(\mathcal{G}^{10}\)-stable, i.e.~\(\mathcal{G}_\tau^{10}(\O P^2)\simeq\mathcal{G}_\omega^{10}(\O P^2)\) for all twistings \(\tau,\omega\in\pi_9(\mathrm{SO}(16))\).
\end{thm}

\begin{proof}
    Since \(\pi_9(\mathrm{SO}(16))\cong\Z/2\), there are two choices of twisting, so to prove gyration stability in this case we need only check when \(\tau\) is non-trivial and \(\omega\) is trivial. By Lemma \ref{lem:G10(OP2)_tau-bar} for $\overline{\tau}$ and Lemma~\ref{twisting12}~(iii) for $\overline{\omega}$, we may write 
    \[
        \sigma_8\circ\overline{\tau}\simeq (\sigma_{8}\circ\nu_{15}^3)+(\sigma_{8}\circ\eta_{15}\circ\varepsilon_{16})
        \text{\; and \;} \sigma_8\circ\overline{\omega}\simeq \ast.
    \] 
    It follows that taking \(\lambda_1\) as in Lemma \ref{lem:lambda - k=10} gives 
    \[
        \sigma_8\circ\overline{\tau}+\lambda_1\circ\sigma_{16}+[\iota_8,\lambda_1] \simeq 2\cdot((\sigma_{8}\circ\nu_{15}^3)+(\sigma_{8}\circ\eta_{15}\circ\varepsilon_{16}))\simeq\ast\simeq\sigma_8\circ\overline{\omega}
    \] 
    thus proving \(\mathcal{G}^{10}\)-stability for \(\O P^2\), by Proposition \ref{prop:lambda_delta}.
\end{proof}

\noindent \textbf{The $k=12$ case}. There is a gyration $\mathcal{G}^{12}_{\tau}(\mathbb{O}P^{2})$ for each $\tau\in\pi_{11}(SO(16))\cong\mathbb{Z}$. We begin by listing the relevant homotopy groups and generators. Note that the $\alpha_{1}$-class of order 7 is denoted by \(\check{\alpha}_1(n)\) to distinguish it from the order 5 and the order 3 classes \(\widetilde{\alpha}_1(n)\) and \(\alpha_1(n)\), and we follow Toda in using $\alpha'_{3}(n)$ to denote the $3$-primary class with the property that $3\cdot\alpha'_{3}(n)\simeq\alpha_{3}(n)$.

\begin{prop}[Toda]\label{prop:G12(OP2)_htpygps}
    There are group isomorphisms:
    \begin{enumerate}
        \item[(i)] \(\pi_{26}(S^{15}) \cong \Z/8\langle\zeta_{15}\rangle\oplus\Z/9\langle\alpha_3'(15)\rangle\oplus\Z/7\langle\check{\alpha}_1(15)\rangle\);
        \item[(ii)] \(\pi_{19}(S^8) \cong \Z/8\langle\zeta_8\rangle\oplus\Z/9\langle\alpha_3'(8)\rangle\oplus\Z/7\langle\check{\alpha}_1(8)\rangle\oplus\Z/2\langle\overline{\nu}_8\circ\nu_{16}\rangle\);
        \item[(iii)] \(\pi_{26}(S^8) \cong \Z/8\langle\sigma_8\circ\zeta_{15}\rangle\oplus\Z/9\langle\sigma_8\circ\alpha_3'(15)\rangle\oplus\Z/7\langle\sigma_8\circ\check{\alpha}_1(15)\rangle\oplus\Z/8\langle\zeta_8\circ\widehat{\sigma}_{19}\rangle \oplus \newline \Z/3\langle\alpha_3'(8)\circ\alpha_2(19)\rangle\oplus\Z/2\). \qed
    \end{enumerate}
\end{prop}

Step 1 is given by the next lemma.

\begin{lem}\label{lem:G12OP2_tau-bar}
    For any \(\tau\in\pi_{11}(\mathrm{SO}(16))\) there exist integers \(a_1\), \(a_2\) and \(a_3\), modulo 8, 9 and 7 respectively, such that \(\sigma_8\circ\overline{\tau}\simeq a_1\cdot(\sigma_8\circ\zeta_{15})+a_2\cdot(\sigma_8\circ\alpha_3'(15))+a_3\cdot(\sigma_8\circ\check{\alpha}_1(15))\).    
\end{lem}

\begin{proof}  
    With $k=12$ and $m=8$, by~(\ref{taubardef}) we have \(\overline{\tau}\in\pi_{26}(S^{15})\). So by Proposition \ref{prop:G12(OP2)_htpygps}~(i),
    \[          \overline{\tau}\in\Z/8\langle\zeta_{15}\rangle\oplus\Z/9\langle\alpha_3'(15)\rangle\oplus\Z/7\langle\check{\alpha}_1(15)\rangle.
    \] 
    Therefore $\overline{\tau}\simeq a_{1}\cdot\zeta_{15}+a_{2}\cdot\alpha'_{3}(15)+a_{3}\cdot\check{\alpha}_{1}(15)$ for some integers $a_{1}$, $a_{2}$ and $a_{3}$ modulo $8$, $9$ and $7$ respectively. The statement of the lemma follows by left distributivity.  
\end{proof}

Now we turn to Step 2. For \(\lambda\in\pi_{19}(S^8)\), by Proposition \ref{prop:G12(OP2)_htpygps}~(ii) we have  
\begin{equation}\label{k=12lambda}
    \lambda\simeq w\cdot\zeta_8+x\cdot\alpha_3'(8)+y\cdot\alpha_1(8)+z\cdot(\overline{\nu}_8\circ\nu_{16})
\end{equation}   
for integers \(w,x,y\) and \(z\) modulo 8, 9, 7 and 2 respectively. 

\begin{lem}\label{lem:step2fork=12} 
   There is an odd integer $\vartheta$ such that 
   \begin{gather*}
    \begin{split}
    \lambda\circ\sigma_{19}+[\iota_8,\lambda] \simeq &~ 2w\cdot(\sigma_8\circ\zeta_{15})+2x\cdot(\sigma_8\circ\alpha_3'(15)) +2y\cdot(\sigma_8\circ\check{\alpha}_1(15)) \\ &~ +(w-\vartheta w)\cdot(\zeta_8\circ\widehat{\sigma}_{19}) + x\cdot(\alpha'_3(8)\circ\alpha_2(19)).
    \end{split}
   \end{gather*} 
\end{lem}

\begin{proof} 
First, consider the composite \(\lambda\circ\sigma_{19}\). 
Writing \(\sigma_{19}\simeq\widehat{\sigma}_{19}+\widetilde{\alpha}_1(19)+\alpha_2(19)\) and using Lemma \ref{lem:coprime_order} to eliminate compositions of elements of coprime orders, we obtain 
\[
    \lambda\circ\sigma_{19}\simeq \lambda \circ (\widehat{\sigma}_{19}+\widetilde{\alpha}_1(19)+\alpha_2(19)) \simeq w\cdot(\zeta_8\circ\widehat{\sigma}_{19}) + x\cdot(\alpha'_3(8)\circ\alpha_2(19))+ z\cdot(\overline{\nu}_8\circ\nu_{16}\circ\widehat{\sigma}_{19}).
\]
By \cite{toda}*{(7.20)}, \(\nu_{16}\circ\widehat{\sigma}_{19}\) is null homotopic. Therefore 
\begin{equation}\label{eq:lambdacomp - k=12}
    \lambda\circ\sigma_{19}\simeq w\cdot(\zeta_8\circ\widehat{\sigma}_{19}) + x\cdot(\alpha'_3(8)\circ\alpha_2(19)).
\end{equation}

Next consider the Whitehead product \([\iota_8,\lambda]\). Each of the generators in Proposition \ref{prop:G12(OP2)_htpygps}(ii) is a suspension, implying that \(\lambda\) is a suspension. Therefore Lemma \ref{lem:whitehead_prods_and_susps} implies that 
\[[\iota_8,\lambda] \simeq [\iota_8,\iota_8]\circ\Sigma^7\lambda.\]
Applying Lemma \ref{lem:[i_8,i_8]} and using the expression for $\lambda$ in~(\ref{k=12lambda}) gives
\begin{align*}
    \begin{split}
    [\iota_8,\lambda] \simeq &~ (2\cdot\sigma_8-\Sigma\sigma')\circ(w\cdot\zeta_{15}+x\cdot\alpha_3'(15)+y\cdot\check{\alpha}_1(15)+z\cdot(\overline{\nu}_{15}\circ\nu_{23})) \\
    \simeq &~ 2w\cdot(\sigma_8\circ\zeta_{15})+2x\cdot(\sigma_8\circ\alpha_3'(15))+2y\cdot(\sigma_8\circ\check{\alpha}_1(15))+2z\cdot(\sigma_8\circ\overline{\nu}_{15}\circ\nu_{23}) \\ &~ -w\cdot(\Sigma\sigma'\circ\zeta_{15})-x\cdot(\Sigma\sigma'\circ\alpha_3'(15))-y\cdot(\Sigma\sigma'\circ\check{\alpha}_1(15))-z\cdot(\Sigma\sigma'\circ\overline{\nu}_{15}\circ\nu_{23}). 
    \end{split}
\end{align*}
Applying Lemma \ref{lem:coprime_order} to eliminate compositions of coprime elements, and observing that \(\overline{\nu}_{15}\circ\nu_{23}\) is null homotopic by \cite{toda}*{(7.22)}, gives
\begin{equation}\label{eq:lambdawhitehead - k=12}
     [\iota_8,\lambda] \simeq 2w\cdot(\sigma_8\circ\zeta_{15})+2x\cdot(\sigma_8\circ\alpha_3'(15)) +2y\cdot(\sigma_8\circ\check{\alpha}_1(15)) -w\cdot(\Sigma\sigma'\circ\zeta_{15}).
\end{equation}
Now, by \cite{toda}*{Lemma 12.12} there exists an odd integer \(\vartheta\) such that \(\Sigma\sigma'\circ\zeta_{15}\simeq\vartheta\cdot(\zeta_8\circ\widehat{\sigma}_{19})\) and so (\ref{eq:lambdacomp - k=12}) and (\ref{eq:lambdawhitehead - k=12}) combine to give
\begin{gather}\label{eq:lambercomp+lambdawhitehead - k=12}
    \begin{split}
    \lambda\circ\sigma_{19}+[\iota_8,\lambda] \simeq &~ 2w\cdot(\sigma_8\circ\zeta_{15})+2x\cdot(\sigma_8\circ\alpha_3'(15)) +2y\cdot(\sigma_8\circ\check{\alpha}_1(15)) \\ &~ +(w-\vartheta w)\cdot(\zeta_8\circ\widehat{\sigma}_{19}) + x\cdot(\alpha'_3(8)\circ\alpha_2(19))
    \end{split}
\end{gather} 
as asserted. 
\end{proof} 

Lemma \ref{lem:step2fork=12} has two immediate consequences, giving the following supplementary lemma.

\begin{lem}\label{lem:k=12_2consequences}
    Let \(\tau,\omega\in\pi_{11}(\mathrm{SO}(16))\) and let \(\lambda\in\pi_{19}(\mathrm{SO}(16))\) be as in (\ref{k=12lambda}). If there is a homotopy \(\sigma_8\circ\overline{\tau}+\lambda\circ\sigma_{19}+[\iota_8,\lambda]\simeq\pm\sigma_8\circ\overline{\omega}\) then the following congruences hold
    \begin{enumerate}
        \item[(i)] \(x\equiv0\text{ (mod 3)}\);
        \item[(ii)] \(w-\vartheta w\equiv0\text{ (mod 8)}\),
    \end{enumerate}
    where \(\vartheta\) is the odd integer or Lemma \ref{lem:step2fork=12}.
\end{lem}

\begin{proof}
    Consider the homotopy \(\sigma_8\circ\overline{\tau}+\lambda\circ\sigma_{19}+[\iota_8,\lambda]\simeq\pm\sigma_8\circ\overline{\omega}\). By Lemma \ref{lem:step2fork=12} the left side has \(\alpha'_3(8)\circ\alpha_2(19)\) with coefficient \(x\) and \(\zeta_8\circ\widehat{\sigma}_{19}\) with coefficient \((w-\vartheta w)\), whereas the right side has both of these with coefficient 0 by Lemma \ref{lem:G12OP2_tau-bar}. Therefore it must be the case that \(x\equiv0\text{ (mod 3)}\) and \(w-\vartheta w\equiv0\text{ (mod 8)}\), since by Proposition \ref{prop:G12(OP2)_htpygps}~(iii) these classes are of order 3 and 8, respectively.
\end{proof}

\begin{rem}\label{rem:vartheta}  
    This is a similar, slightly more complicated, situation to that of the \(k=4\) case (cf. Remark~\ref{rem:xi_postpone}). Although the precise value of the odd integer \(\vartheta\) is not determined, the congruences of Lemma \ref{lem:k=12_2consequences} result in the following analogue to Proposition \ref{prop:xi}.
\end{rem}

\begin{prop}\label{prop:theta}
    Let \(\vartheta\) be the odd integer of Lemma \ref{lem:step2fork=12} and \(\tau\in\pi_{11}(\mathrm{SO}(16))\) be an arbitrary twisting.
    \begin{enumerate}
        \item[(i)] If \(\vartheta\equiv1\text{ (mod 8)}\) then \(\mathcal{G}_\tau^{12}(\O P^2)\) can take exactly four possible homotopy types.
        \item[(ii)] If \(\vartheta\equiv5\text{ (mod 8)}\) then \(\mathcal{G}_\tau^{12}(\O P^2)\) can take exactly six possible homotopy types.
        \item[(iii)] If \(\vartheta\equiv3\text{ or }7\text{ (mod 8)}\) then \(\mathcal{G}_\tau^{12}(\O P^2)\) can take exactly ten possible homotopy types.
    \end{enumerate}
\end{prop}

\begin{proof}
    By Lemma \ref{lem:G12OP2_tau-bar}, for twistings \(\tau,\omega\in\pi_{11}(\mathrm{SO}(16))\) we may write 
    \[\sigma_8\circ\overline{\tau}\simeq a_1\cdot(\sigma_8\circ\zeta_{15})+a_2\cdot(\sigma_8\circ\alpha_3'(15))+a_3\cdot(\sigma_8\circ\check{\alpha}_1(15))\]
    and
    \[\sigma_8\circ\overline{\omega}\simeq b_1\cdot(\sigma_8\circ\zeta_{15})+b_2\cdot(\sigma_8\circ\alpha_3'(15))+b_3\cdot(\sigma_8\circ\check{\alpha}_1(15))\] 
    for some integers \(a_1\) and \(b_1\) considered modulo 8, \(a_2\) and \(b_2\) modulo 9, and \(a_3\) and \(b_3\) modulo 7. By Proposition \ref{prop:lambda_delta}, there is a homotopy equivalence \(\mathcal{G}_\tau^{12}(\O P^2)\simeq\mathcal{G}_\omega^{12}(\O P^2)\) if and only if there exists a \(\lambda\in\pi_{19}(S^8)\) such that  \(\sigma_8\circ\overline{\tau}+\lambda\circ\sigma_{19}+[\iota_8,\lambda]\simeq\pm\sigma_8\circ\overline{\omega}\). By Lemma~\ref{lem:step2fork=12}, writing \(\lambda\) as 
    \[
        \lambda\simeq w\cdot\zeta_8+x\cdot\alpha_3'(8)+y\cdot\alpha_1(8)+z\cdot(\overline{\nu}_8\circ\nu_{16})
    \]
    as in (\ref{k=12lambda}) and applying the congruences of Lemma \ref{lem:k=12_2consequences} gives a homotopy 
    \begin{equation}\label{eq:delta12}
        \sigma_8\circ\overline{\tau}+\lambda\circ\sigma_{19}+[\iota_8,\lambda] \simeq (a_1+2w)\cdot(\sigma_8\circ\zeta_{15})+(a_2+2x)\cdot(\sigma_8\circ\alpha_3'(15)) + (a_3+2y)\cdot(\sigma_8\circ\check{\alpha}_1(15))
    \end{equation} 
    Comparing coefficients, this implies that \(\mathcal{G}_\tau^{12}(\O P^2)\simeq\mathcal{G}_\omega^{12}(\O P^2)\) if and only if we have the following congruences:
    \[
         a_1+2w\equiv\pm b_1\text{ (mod 8)}
         \text{, \;}
         a_2+2x\equiv\pm b_2\text{ (mod 9)}
         \text{\; and \;}
         a_3+2y\equiv\pm b_3\text{ (mod 7)}.
    \]
    Further, given any \(a_3\) and \(b_3\) modulo $7$, there always exists a \(y\) such that \(a_3+2y\equiv\pm b_3\pmod{7}\). Hence a homotopy equivalence exists if and only if the first two of the above congruences hold.
    
    First observe that $a_{1}+2w\equiv \pm b_1\, \text{(mod 8)}$ implies that $a_{1}\equiv b_1\, \text{(mod 2)}$. So if $a_{1}\not\equiv b_1\, \text{(mod 2)}$ then \(\mathcal{G}_\tau^{12}(\O P^2)\not\simeq\mathcal{G}_\omega^{12}(\O P^2)\). In particular, if $a_{1}$ is even and $b_{1}$ is odd then \(\mathcal{G}_\tau^{12}(\O P^2)\not\simeq\mathcal{G}_\omega^{12}(\O P^2)\). Moreover, the restriction that \(x\equiv0\pmod{3}\) implies that the second congruence reduces to \(a_2\equiv\pm b_2\pmod{3}\), so if \(a_2\equiv0\pmod{3}\) and \(b_2\equiv\pm1\pmod{3}\) then this would also give \(\mathcal{G}_\tau^{12}(\O P^2)\not\simeq\mathcal{G}_\omega^{12}(\O P^2)\). This implies that \(\O P^2\) is not \(\mathcal{G}^{12}\)-stable and hence directly answers GSI in the negative.

    We now turn to GSII and enumerating the possible homotopy types for \(\mathcal{G}^{12}_\tau(\O P^2)\). This depends on the possible choices of $w$ that give $a_{1}+2w\equiv \pm b_1 \pmod{8}$ and whether \(a_2\equiv0\text{ or }\pm1\pmod{3}\). Since $a_{1}$ is an integer modulo 8 this implies that are at most sixteen possible homotopy types; we label each one by the value of \(a_1\) when \(a_2\equiv0\text{ (mod 3)}\), which we write as \(\mathcal{G}_0^{12}(\O P^2),\mathcal{G}_1^{12}(\O P^2),\dots,\mathcal{G}_7^{12}(\O P^2)\), and the second eight for when \(a_2\equiv\pm1\text{ (mod 3)}\) are written as \(\mathcal{G}_{0,\pm}^{12}(\O P^2),\mathcal{G}_{1,\pm}^{12}(\O P^2),\dots,\mathcal{G}_{7,\pm}^{12}(\O P^2)\). There are three cases, which depend on the odd integer \(\vartheta\) modulo~8.

    \textit{Part (i):} if \(\vartheta\equiv1\text{ (mod 8)}\) then \(w-w\vartheta\equiv0\text{ (mod 8)}\) holds for all \(w\). Thus \(a_{1}+2w\equiv \pm b_1 \pmod{8}\) if and only if \(a_1\equiv b_1\text{ (mod 2)}\), so this case there are four possible homotopy types: 
    \begin{gather*}
        \mathcal{G}_0^{12}(\O P^2)\simeq\mathcal{G}_2^{12}(\O P^2)\simeq\mathcal{G}_4^{12}(\O P^2)\simeq\mathcal{G}_6^{12}(\O P^2), \\ \mathcal{G}_1^{12}(\O P^2)\simeq\mathcal{G}_3^{12}(\O P^2)\simeq\mathcal{G}_5^{12}(\O P^2)\simeq\mathcal{G}_7^{12}(\O P^2), \\ \mathcal{G}_{0,\pm}^{12}(\O P^2)\simeq\mathcal{G}_{2,\pm}^{12}(\O P^2)\simeq\mathcal{G}_{4,\pm}^{12}(\O P^2)\simeq\mathcal{G}_{6,\pm}^{12}(\O P^2), \\ \text{\; and \;}  \mathcal{G}_{1,\pm}^{12}(\O P^2)\simeq\mathcal{G}_{3,\pm}^{12}(\O P^2)\simeq\mathcal{G}_{5,\pm}^{12}(\O P^2)\simeq\mathcal{G}_{7,\pm}^{12}(\O P^2).
    \end{gather*}
    On the other hand, we have already seen that if $a_{1}\not\equiv b_1 \text{(mod 2)}$ or \(a_2\not\equiv \pm b_2\pmod{3}\) then \(\mathcal{G}_\tau^{12}(\O P^2)\not\simeq\mathcal{G}_\omega^{12}(\O P^2)\). Thus there are exactly four homotopy types in this case.

    \textit{Parts (ii) and (iii):} if \(\vartheta\not\equiv1\text{ (mod 8)}\) then we are in one of two situations. If \(\vartheta\equiv5\text{ (mod 8)}\), then the demand that \(w-w\vartheta\equiv0\text{ (mod 8)}\) implies \(4w\equiv0\text{ (mod 8)}\), forcing \(w\) to take only even values. Thus \(a_{1}+2w\equiv \pm b_1 \pmod{8}\) if and only if \(a_1\equiv \pm b_1\text{ (mod 4)}\) and hence \(\mathcal{G}^{12}_\tau(\O P^2)\) can assume six different homotopy types, represented by
    \begin{gather*}
    \mathcal{G}_0^{12}(\O P^2)\simeq\mathcal{G}_4^{12}(\O P^2)\text{, \;}\mathcal{G}_1^{12}(\O P^2)\simeq\mathcal{G}_3^{12}(\O P^2)\simeq\mathcal{G}_5^{12}(\O P^2)\simeq\mathcal{G}_7^{12}(\O P^2), 
    \\ 
    \mathcal{G}_{0,\pm}^{12}(\O P^2)\simeq\mathcal{G}_{4,\pm}^{12}(\O P^2)\text{, \;}\mathcal{G}_{1,\pm}^{12}(\O P^2)\simeq\mathcal{G}_{3,\pm}^{12}(\O P^2)\simeq\mathcal{G}_{5,\pm}^{12}(\O P^2)\simeq\mathcal{G}_{7,\pm}^{12}(\O P^2), 
    \\
    \mathcal{G}_2^{12}(\O P^2)\simeq\mathcal{G}_6^{12}(\O P^2)\text{\; and \;}\mathcal{G}_{2,\pm}^{12}(\O P^2)\simeq\mathcal{G}_{6,\pm}^{12}(\O P^2).
    \end{gather*}
    This proves (ii). If instead \(\vartheta\equiv3\text{ or }7\text{ (mod 8)}\) then \(w-w\vartheta\equiv0\text{ (mod 8)}\) implies that \(w\equiv0\text{ or }4\text{ (mod 8)}\), in which case \(a_{1}+2w\equiv \pm b_1 \pmod{8}\) if and only if \(a_1\equiv \pm b_1\text{ (mod 8)}\). So \(\mathcal{G}^{12}_\tau(\O P^2)\) can take ten possible homotopy types:
    \begin{gather*}
    \mathcal{G}_0^{12}(\O P^2)\text{, \;}\mathcal{G}_1^{12}(\O P^2)\simeq\mathcal{G}_7^{12}(\O P^2)\text{, \;}\mathcal{G}_2^{12}(\O P^2)\simeq\mathcal{G}_6^{12}(\O P^2)\text{, \;} \\ \mathcal{G}_3^{12}(\O P^2)\simeq\mathcal{G}_5^{12}(\O P^2)\text{, \;} \mathcal{G}_4^{12}(\O P^2), 
    \\
    \mathcal{G}_{0,\pm}^{12}(\O P^2)\text{, \;}\mathcal{G}_{1,\pm}^{12}(\O P^2)\simeq\mathcal{G}_{7,\pm}^{12}(\O P^2)\text{, \;}\mathcal{G}_{2,\pm}^{12}(\O P^2)\simeq\mathcal{G}_{6,\pm}^{12}(\O P^2)\text{, \;} \\ \mathcal{G}_{3,\pm}^{12}(\O P^2)\simeq\mathcal{G}_{5,\pm}^{12}(\O P^2)\text{\; and \;} \mathcal{G}_{4,\pm}^{12}(\O P^2). \qedhere
    \end{gather*}
    This proves (iii).
\end{proof}

Proposition \ref{prop:theta} shows that for \(k=12\) the answer to GSII for \(\O P^2\) is at least 4, so it immediately implies the following. 

\begin{thm}\label{thm:g12op2}
    \(\O P^2\) is not \(\mathcal{G}^{12}\)-stable. \qed
\end{thm} 

\noindent 
\textbf{The classification of $\mathcal{G}^{2}$-homotopy types}.  
We conclude by combining several results to classify the homotopy types of the gyrations $\mathcal{G}^{k}_{\tau}(\mathbb{F}P^{2})$ for $\mathbb{F}$ one of $\mathbb{C}$, $\mathbb{H}$ or $\mathbb{O}$.

\begin{proof}[Proof of Theorem~\ref{thm:classification}] 
    Part (i) follows since \(\C P^2\) is \(\mathcal{G}^2\)-stable by Theorem \ref{thm:g2cp2}. For part (ii), if \(\tau\simeq\omega\) then $\mathcal{G}^{2}_{\tau}(\mathbb{H}P^{2})\simeq\mathcal{G}^{2}_{\omega}(\mathbb{H}P^{2})$ by Lemma \ref{lem:easycheck}. Conversely, as \(\pi_1(\mathrm{SO}(8))\cong\Z/2\), there are two distinct choices of twisting. Theorem \ref{thm:g2hp2} shows that if \(\tau\not\simeq\omega\) then \(\mathcal{G}_\tau^2(\H P^2)\not\simeq\mathcal{G}_\omega^2(\H P^2)\). The argument for part (iii) is the same as for part (ii), replacing Theorem~\ref{thm:g2hp2} with Theorem \ref{thm:g2op2}. 
\end{proof} 

\bibliographystyle{amsalpha}
\bibliography{bib}

\end{document}